\documentclass{amsart}

\usepackage{etoolbox}

\makeatletter
\let\old@tocline\@tocline
\let\section@tocline\@tocline
\newcommand{\subsection@dotsep}{4.5}
\newcommand{\subsubsection@dotsep}{4.5}
\patchcmd{\@tocline}
{\hfil}
{\nobreak
	\leaders\hbox{$\m@th
		\mkern \subsection@dotsep mu\hbox{.}\mkern \subsection@dotsep mu$}\hfill
	\nobreak}{}{}
\let\subsection@tocline\@tocline
\let\@tocline\old@tocline

\patchcmd{\@tocline}
{\hfil}
{\nobreak
	\leaders\hbox{$\m@th
		\mkern \subsubsection@dotsep mu\hbox{.}\mkern \subsubsection@dotsep mu$}\hfill
	\nobreak}{}{}
\let\subsubsection@tocline\@tocline
\let\@tocline\old@tocline

\let\old@l@subsection\l@subsection
\let\old@l@subsubsection\l@subsubsection

\def\@tocwriteb#1#2#3{%
	\begingroup
	\@xp\def\csname #2@tocline\endcsname##1##2##3##4##5##6{%
		\ifnum##1>\c@tocdepth
		\else \sbox\z@{##5\let\indentlabel\@tochangmeasure##6}\fi}%
	\csname l@#2\endcsname{#1{\csname#2name\endcsname}{\@secnumber}{}}%
	\endgroup
	\addcontentsline{toc}{#2}%
	{\protect#1{\csname#2name\endcsname}{\@secnumber}{#3}}}%

\newlength{\@tocsectionindent}
\newlength{\@tocsubsectionindent}
\newlength{\@tocsubsubsectionindent}
\newlength{\@tocsectionnumwidth}
\newlength{\@tocsubsectionnumwidth}
\newlength{\@tocsubsubsectionnumwidth}
\newcommand{\settocsectionnumwidth}[1]{\setlength{\@tocsectionnumwidth}{#1}}
\newcommand{\settocsubsectionnumwidth}[1]{\setlength{\@tocsubsectionnumwidth}{#1}}
\newcommand{\settocsubsubsectionnumwidth}[1]{\setlength{\@tocsubsubsectionnumwidth}{#1}}
\newcommand{\settocsectionindent}[1]{\setlength{\@tocsectionindent}{#1}}
\newcommand{\settocsubsectionindent}[1]{\setlength{\@tocsubsectionindent}{#1}}
\newcommand{\settocsubsubsectionindent}[1]{\setlength{\@tocsubsubsectionindent}{#1}}

\renewcommand{\l@section}{\section@tocline{1}{\@tocsectionvskip}{\@tocsectionindent}{}{\@tocsectionformat}}%
\renewcommand{\l@subsection}{\subsection@tocline{2}{\@tocsubsectionvskip}{\@tocsubsectionindent}{}{\@tocsubsectionformat}}%
\renewcommand{\l@subsubsection}{\subsubsection@tocline{3}{\@tocsubsubsectionvskip}{\@tocsubsubsectionindent}{}{\@tocsubsubsectionformat}}%
\newcommand{\@tocsectionformat}{}
\newcommand{\@tocsubsectionformat}{}
\newcommand{\@tocsubsubsectionformat}{}
\expandafter\def\csname toc@1format\endcsname{\@tocsectionformat}
\expandafter\def\csname toc@2format\endcsname{\@tocsubsectionformat}
\expandafter\def\csname toc@3format\endcsname{\@tocsubsubsectionformat}
\newcommand{\settocsectionformat}[1]{\renewcommand{\@tocsectionformat}{#1}}
\newcommand{\settocsubsectionformat}[1]{\renewcommand{\@tocsubsectionformat}{#1}}
\newcommand{\settocsubsubsectionformat}[1]{\renewcommand{\@tocsubsubsectionformat}{#1}}
\newlength{\@tocsectionvskip}
\newcommand{\settocsectionvskip}[1]{\setlength{\@tocsectionvskip}{#1}}
\newlength{\@tocsubsectionvskip}
\newcommand{\settocsubsectionvskip}[1]{\setlength{\@tocsubsectionvskip}{#1}}
\newlength{\@tocsubsubsectionvskip}
\newcommand{\settocsubsubsectionvskip}[1]{\setlength{\@tocsubsubsectionvskip}{#1}}

\patchcmd{\tocsection}{\indentlabel}{\makebox[\@tocsectionnumwidth][l]}{}{}
\patchcmd{\tocsubsection}{\indentlabel}{\makebox[\@tocsubsectionnumwidth][l]}{}{}
\patchcmd{\tocsubsubsection}{\indentlabel}{\makebox[\@tocsubsubsectionnumwidth][l]}{}{}

\newcommand{\@sectypepnumformat}{}
\renewcommand{\contentsline}[1]{%
	\expandafter\let\expandafter\@sectypepnumformat\csname @toc#1pnumformat\endcsname%
	\csname l@#1\endcsname}
\newcommand{\@tocsectionpnumformat}{}
\newcommand{\@tocsubsectionpnumformat}{}
\newcommand{\@tocsubsubsectionpnumformat}{}
\newcommand{\setsectionpnumformat}[1]{\renewcommand{\@tocsectionpnumformat}{#1}}
\newcommand{\setsubsectionpnumformat}[1]{\renewcommand{\@tocsubsectionpnumformat}{#1}}
\newcommand{\setsubsubsectionpnumformat}[1]{\renewcommand{\@tocsubsubsectionpnumformat}{#1}}
\renewcommand{\@tocpagenum}[1]{%
	\hfill {\mdseries\@sectypepnumformat #1}}

\let\oldappendix\appendix
\renewcommand{\appendix}{%
	\leavevmode\oldappendix%
	\addtocontents{toc}{%
		\protect\settowidth{\protect\@tocsectionnumwidth}{\protect\@tocsectionformat\sectionname\space}%
		\protect\addtolength{\protect\@tocsectionnumwidth}{2em}}%
}
\makeatother

\makeatletter
\settocsectionnumwidth{2em}
\settocsubsectionnumwidth{2.5em}
\settocsubsubsectionnumwidth{3em}
\settocsectionindent{1pc}%
\settocsubsectionindent{\dimexpr\@tocsectionindent+\@tocsectionnumwidth}%
\settocsubsubsectionindent{\dimexpr\@tocsubsectionindent+\@tocsubsectionnumwidth}%
\makeatother

\settocsectionvskip{10pt}
\settocsubsectionvskip{0pt}
\settocsubsubsectionvskip{0pt}

\settocsectionformat{\bfseries}
\settocsubsectionformat{\mdseries}
\settocsubsubsectionformat{\mdseries}
\setsectionpnumformat{\bfseries}
\setsubsectionpnumformat{\mdseries}
\setsubsubsectionpnumformat{\mdseries}

\let\oldtableofcontents\tableofcontents
\renewcommand{\tableofcontents}{%
	\vspace*{-\linespacing}
	\oldtableofcontents}

\usepackage{soul}
\usepackage{amssymb,latexsym,amsmath,amsthm,graphicx,bm}
\usepackage{sidecap,multicol}
\usepackage[all,cmtip]{xy}
\usepackage{hyperref}
\hypersetup{hidelinks}
\usepackage{breakurl}
\usepackage[capitalize,nameinlink]{cleveref}
\usepackage[mathscr]{euscript}
\usepackage[shortlabels]{enumitem}

\usepackage{tikz}
\usetikzlibrary{arrows.meta,decorations.pathmorphing,decorations.markings,backgrounds,positioning,fit,shapes}

\newcommand{\ZZ}{\mathbb{Z}}

\newcommand{\RR}{\mathbb{R}}

\newcommand{\kk}{\Bbbk}

\newcommand{\cC}{\mathcal{C}}

\newcommand{\cE}{\mathcal{E}}

\newcommand{\cL}{\mathcal{L}}
\newcommand{\cM}{\mathcal{M}}

\newcommand{\cO}{\mathcal{O}}

\newcommand{\cT}{\mathcal{T}}

\newcommand{\scrT}{\mathscr{T}}

\DeclareSymbolFont{toneitalic}{T1}{\familydefault}{m}{it}
\DeclareMathSymbol{\cdel}{\mathord}{toneitalic}{"F0}

\newcommand{\mh}{}
\mathchardef\mh="2D

\DeclareMathOperator{\Hom}{Hom}

\DeclareMathOperator{\End}{End}

\newcommand{\del}{\partial}

\newcommand{\nec}{\mathrm{nec}}

\newcommand{\coev}{\mathrm{coev}}
\newcommand{\id}{\mathrm{id}}
\newcommand{\ad}{\mathrm{ad}}

\DeclareMathOperator{\sheafHom}{\mathcal{H}\kern -1.2pt \mathit{om}}

\newcommand{\tikzfig}[1]{\begin{tikzpicture}[auto,baseline={([yshift=-.5ex]current bounding box.center)}]#1\end{tikzpicture}}
\usetikzlibrary{decorations.pathreplacing}

\newtheorem{theorem}{Theorem}
\newtheorem{proposition}[theorem]{Proposition}
\newtheorem{corollary}[theorem]{Corollary}
\newtheorem{lemma}[theorem]{Lemma}

\newtheorem*{theorem*}{Theorem}
\theoremstyle{definition}
\newtheorem{definition}{Definition}

\theoremstyle{remark}
\newtheorem*{remark}{Remark}
\newtheorem*{example}{Example}

\usepackage[backend=biber, maxbibnames=99, bibencoding=utf8, style=alphabetic]{biblatex}
\title[Smooth CY structures and the NC Legendre transform]{Smooth Calabi-Yau structures and the noncommutative Legendre transform}
\author{Maxim Kontsevich, Alex Takeda and Yiannis Vlassopoulos}

\begin{document}
\tikzset{>={Stealth[scale=1.2]}}
\tikzset{->-/.style={decoration={
			markings,
			mark=at position #1 with {\arrow{>}}},postaction={decorate}}}
\tikzset{-w-/.style={decoration={
			markings,
			mark=at position #1 with {\arrow{Stealth[fill=white,scale=1.4]}}},postaction={decorate}}}
\tikzset{->-/.default=0.65}
\tikzset{-w-/.default=0.65}
\tikzstyle{bullet}=[circle,fill=black,inner sep=0.5mm]
\tikzstyle{circ}=[circle,draw=black,fill=white,inner sep=0.5mm]
\tikzstyle{vertex}=[circle,draw=black,thick,inner sep=0.5mm]
\tikzstyle{dot}=[draw,circle,fill=black,minimum size=0.5mm,inner sep = 0mm, outer sep = 0mm]
\tikzset{darrow/.style={double distance = 4pt,>={Implies},->},
	darrowthin/.style={double equal sign distance,>={Implies},->},
	tarrow/.style={-,preaction={draw,darrow}},
	qarrow/.style={preaction={draw,darrow,shorten >=0pt},shorten >=1pt,-,double,double
		distance=0.2pt}}
	
\begin{abstract}
	We elucidate the relation between smooth Calabi-Yau structures and pre-Calabi-Yau structures. We show that, from a smooth Calabi-Yau structure on an $A_\infty$-category $A$, one can produce a pre-Calabi-Yau structure on $A$; as defined in our previous work, this is a shifted noncommutative version of an integrable polyvector field. We explain how this relation is an analogue of the Legendre transform, and how it defines a one-to-one mapping, in a certain homological sense. For concreteness, we apply this formalism to chains on based loop spaces of (possibly non-simply connected) Poincar\'e duality spaces, and fully calculate the case of the circle.
\end{abstract}

\maketitle
\tableofcontents

\section{Introduction}
This paper is a continuation of our previous work \cite{kontsevich2021precalabiyau}. There, we described a type of algebraic structure that we called a \emph{pre-Calabi-Yau} structure on an $A_\infty$-algebra/category $A$; this is a generalization of both proper and smooth Calabi-Yau structures. In that paper we described how, using the formalism of ribbon quivers (that is, ribbon graphs with acyclic orientation), one can use the pre-CY structure maps to describe the action of certain PROP on the morphism spaces of $A$, and on its Hochschild chains $C_*(A)$. The relevant dg PROP has spaces of operations given by chains on moduli spaces of open-closed surfaces with framed boundaries, with at least one input and one output.

We can motivate this result by using the language of the cobordism hypothesis: a pre-CY structure on $A$ should give a partially defined fully extended 2d oriented TQFT with values in (an $\infty$-categorical version) of $2$-category of algebras and bimodules, assigning $A$ to the point and $HH_*(A)$ to the framed circle. This theory will be partially defined in the sense that it does not assign a value to every cobordism, but rather only to those cobordisms that can be generated by handles of index one only; such a cobordism has at least one input and one output.

If instead one admits cobordisms generated by handles of indices one and two, one can cap the outputs of the cobordism, and obtain all cobordisms with at least one input. That type of TQFT is known to be described by proper Calabi-Yau structures; see Lurie's description \cite{lurie2009classification} of Costello's results in \cite{costello2007topological,costello2005gromovwitten}. In other words, requiring the finiteness condition of $A$ being proper (that is, $H^*(A)$ being finite-rank) allows one to evaluate caps in the TQFT, which get sent to the trace $HH_*(A) \to \kk$ defined by the proper CY structure.

There is another finiteness condition that is dual to properness, which is homological smoothness: $A$ is homologically smooth if the diagonal bimodule $A_\Delta$ is a perfect object in the category of $(A,A)$-bimodules. In the work cited above, Lurie mentions in passing that, from abstract reasons, there should be a dual story to Costello's description of this TQFT: smooth Calabi-Yau structures on $A$ should give a dual type of partially-defined TQFT, which now has a \emph{cup}; this gets sent to a cotrace $\kk \to HH_*(A)$. These TQFTs are also, in practice, described by algebra structures over certain PROPs, given by chains on spaces of surfaces with non-empty incoming/outgoing boundary. The homotopy theory of these objects has been studied in detail elsewhere; see \cite{desmukh2022homotopical} for a recent description of these PROPs and their relation to Deligne-Mumford compactifications.

As a corollary to these statements about cobordisms and TQFTs, it should be the case that a smooth Calabi-Yau structure defines a pre-CY structure. The main purpose of this work is to make this result as explicit as possible: we demonstrate that there is an algorithmic procedure, using the formalism of ribbon quivers we defined previously, which starts from a smooth Calabi-Yau structure on an $A_\infty$-category $A$ and produces the structure maps of a pre-CY structure on the same $A$. The point of having such an explicit description is that many categories/algebras of interest in topology and geometry have such smooth CY structures \cite{ganatra2019cyclic,bozec2021calabiyau1,bozec2021calabiyau2,kaplan2019multiplicative,shende2016calabiyau}. Using the description in this paper allows one to apply the results of \cite{kontsevich2021precalabiyau} to these categories, and compute the TQFT operations that the resulting pre-CY structure gives.

Let us recall the definitions of these objects. A smooth CY structure of dimension $d$ on $A$ is a negative cyclic chain $\omega \in CC^-_*(A)$ which satisfies a nondegeneracy condition: its image in $HH_*(A)$ induces a quasi-isomorphism $A^![d] \to A$ between the inverse dualizing bimodule $A^!$ and a shift of the diagonal bimodule $A_\Delta$. A variant of this definition first appeared in the work of Ginzburg \cite{ginzburg2006calabiyau}, without requiring the lift to negative cyclic homology; often these are referred to as `Ginzburg CY structures' or `weak smooth CY structures' in the literature. Requiring the negative cyclic lift was first proposed, by the first and third named authors of this article, back in 2013, motivated exactly by this TQFT perspective: in order to `close up inputs' with a cup, the cotrace $k \to HH_*(A)$ associated to that cup should factor through the (homotopy) fixed points of the $S^1$-action.

For more recent precise definitions of smooth CY structures in the dg and $A_\infty$-case, see \cite{brav2019relative,brav2018relative,ganatra2013symplectic}. We will need an even more explicit description; we explain a chain-level version of the smooth CY nondegeneracy condition in \cref{sec:smoothCY}. On the other side, the definition of pre-CY structure is already given `at cochain-level': a pre-CY structure of dimension $d$ on an $A_\infty$-category $(A,\mu)$ is the choice of an element
\[ m = \mu + m_{(2)} + m_{(3)} + \dots \in C^*_{[d]}(A) \]
extending the $A_\infty$-structure maps, and satisfying a Maurer-Cartan equation $[m,m] = 0$. We refer the reader to \cite[Sec.3]{kontsevich2021precalabiyau} for the definition of the space $C^*_{[d]}(A)$; let us just mention its noncommutative geometry interpretation: if the space of Hochschild chains $C^*(A)$ is seen as the space of vector fields on some noncommutative space associated to $A$, then the space $C^*_{[d]}(A)$ is the space of polyvector fields (up to some shifts depending on $d$), carrying an analogue of the Schouten-Nijenhuis bracket; a pre-CY structure can be seen as a non-strict version of a Poisson structure; the `bivector field' $m_{(2)}$ does not satisfy the involutivity condition $[m_{(2)},m_{(2)}] = 0$ on the nose, but up to a correction given by $m_{(3)}$, which itself is satisfies an involutivity condition up to a higher correction, and so on.

These relations between CY/pre-CY structures on one side, and derived symplectic/Poisson structures on the other, have been described for the case of dg categories in \cite{yeung2018precalabiyau}, using obstruction theory techniques, and in \cite{pridham2020shifted}, the duality between smooth CY and pre-CY structures is described, using techniques of derived noncommutative geometry. Our goal in this paper is to give a more explicit perspective on this relation, based on the diagrammatic calculus of \cite{kontsevich2021precalabiyau}; in \cref{sec:ncLegendre} we describe maps
\begin{equation} \label{eqn:main}
	(CC^-_d(A))_\mathrm{nondeg} \leftrightarrows (\cM_{d\mathrm{-pre-CY}})_\mathrm{nondeg} 
\end{equation}
going between smooth CY structures of dimension $d$ and pre-CY structures of the same dimension, whose `bivector field' $m_{(2)}$ is nondegenerate. It turns out that these maps are noncommmutative analogues of the Legendre transform and the inverse Legendre transform (between e.g. functions on the total space of a real vector bundle and of its dual). In order to make this analogy more understandable we start with the (mildly noncommutative) case of an odd vector bundle.

The fully noncommutative case is described in terms of certain combinations of ribbon quivers; we evaluate these by inserting the correct structure maps into the vertices, and following the prescriptions in \cite[Sec.6]{kontsevich2021precalabiyau}. By solving an iterative lifting problem, it is possible to construct linear combinations of ribbon quivers $\Gamma_{(2)},\Gamma_{(3)},\dots$, which evaluated on some smooth CY structure $\lambda$ give the pre-CY structure maps.

In general, this procedure is very complicated to implement in practice. However for some simple cases it is possible to use it and calculate pre-CY structures, even by hand. We do this for the case of a particularly simple dg category, the path category $A$ of the triangle. This is equivalent to the dg algebra $k[x^{\pm 1}]$ of chains on the based loop space of the circle. We discuss these path dg categories for general simplicial sets in \cref{sec:pathDg}; an orientation on the geometric realization of such a simplicial set gives a smooth CY structure on $A$. We then specialize to the circle, showing how to understand the chain-level nondegeneracy condition in the case of $A$. Later, in \cref{sec:exampleContinued}, we carry our the computation and calculate the full corresponding pre-CY structure in that case.

Finally, we discuss in what sense the maps \cref{eqn:main} are inverses. The usual Legendre transform defines a one-to-one correspondence between fiberwise convex functions on $E$ and on $E^\vee$; this is also true for the maps in the noncommutative case, but in a more subtle sense. The most natural statement to be made is that these maps lift to weak homotopy equivalences of simplicial sets; on the left-hand side of \cref{eqn:main} we have then a simplicial set corresponding (under the Dold-Kan equivalence) to the nondegenerate locus of a truncated negative cyclic complex, and on the right-hand side, the simplicial set of solutions to the Maurer-Cartan equation as described by \cite{hinich1996descent,getzler2009lie}.\\

\noindent \textit{Acknowledgments:} We would like to thank Damien Calaque, Sheel Ganatra, Ludmil Katzarkov, Bernhard Keller, Joost Nuiten, Manuel Rivera, Vivek Shende, Bertrand To\"en, Bruno Vallette and Zhengfang Wang for helpful conversations. We would like to thank IHES for the wonderful working conditions provided. This work was supported by the Simons Collaboration on Homological Mirror Symmetry.

\subsection*{Notation and conventions}
Throughout this paper, we fix a field $\kk$ of characteristic zero, and will denote simply by $\otimes,\Hom$ the tensor product/hom (of vector spaces, complexes, modules etc) over $\kk$.

In this paper we will assume everything is $\ZZ$-graded, but all of the results also follow for $\ZZ/2\ZZ$-grading. Given an $A_\infty$ category $A$, and an element $a \in A$ of homogeneous degree, we denote by $\deg(a)$ its degree in $A$ and by $\bar a$ its degree in $A[1]$. As for the degrees in other complexes, we will often switch between homological degree and cohomological degree; we made an effort to explicitly specify which degree we mean when there is room for confusion, and as is conventional, denote by upper indices cohomological degree and lower indices homological degree, which are related by $(-)^i = (-)_{-i}$.

We will always assume our $A_\infty$-algebras/categories are strictly unital. For ease of notation, by $C^*(A,M)$ and $C_*(A,M)$, we will always denote the \emph{reduced} Hochschild cochain/chain complexes. That is, if $A$ is an $A_\infty$-algebra, we have
\[ C^*(A,M) = \prod_{n \ge 0} \Hom(\overline{A}[1]^{\otimes n}, M), \quad C_*(A,M) = \prod_{n \ge 0} M \otimes \overline{A}[1]^{\otimes n} \]
where $\overline{A} = A/\kk.1_A$. These complexes can be obtained by taking the quotient of the usual complexes by the elements that have $1_A \in A[1]$ somewhere. Analogously, for a category we take the quotient by the strict units, see \cite{sheridan2020formulae}. We will denote by $b$ the chain differential and by $d$ the cochain differential.

Finally, we denote by $A_\Delta$ and $A^!$ the diagonal bimodule of $A$ and the inverse dualizing bimodule, respectively; for the definition of the bimodule structure maps for these objects, together with all the signs, see \cite{ganatra2013symplectic}. Our conventions agree with the signs there, with the single difference that we write the arguments in the structure maps $\mu(\dots)$ in the opposite order.

\section{Smooth Calabi-Yau structures}\label{sec:smoothCY}
An $A_\infty$-category is \emph{(homologically) smooth} if its diagonal bimodule $A_\Delta$ is a compact object; equivalently, if that object is quasi-isomorphic to a retract of a finite complex of `Yoneda bimodules' (see \cite{ganatra2013symplectic}). For example, if $A$ is a dg algebra such that $A_\Delta$ has a finite resolution by direct sums of the free bimodule $A\otimes A$, then $A$ is smooth.

When $A$ is smooth, there is a quasi-isomorphism of complexes
\[ C_*(A) \simeq \Hom_{A-A}(A^!,A_\Delta) \]
between Hochschild chains and morphisms of $A$-bimodules from the inverse dualizing bimodule to the diagonal bimodule. Recall also that $C_*(A)$ carries the action of the homology of a circle, and the homotopy fixed points of this action are calculated by the negative cyclic complex $CC^-_*(A) = (C_*(A)[[u]],b+uB)$, where $B$ is the Connes differential, of homological degree $+1$.

The following definition is a refinement of the notion of a Ginzburg CY algebra \cite{ginzburg2006calabiyau}.
\begin{definition}
	A smooth Calabi-Yau structure of dimension $d$ on $A$ is a negative cyclic chain $\lambda = \lambda_0 + \lambda_1 u + \lambda_2 u^2 + \dots \in CC^-_d(A)$ whose image $\lambda_0 \in C_d(A) \cong \Hom_{A-A}(A^!,A_\Delta[-d])$ is a quasi-isomorphism.
\end{definition}
Requiring this lift to negative cyclic homology was suggested by two of us some years ago. This notion has since appeared in many places; for a dg discussion see \cite{brav2019relative,brav2018relative} and for an $A_\infty$ discussion see \cite{ganatra2013symplectic,ganatra2019cyclic}.

\subsection{Chain-level nondegeneracy}\label{sec:chainLevelNondeg}
We now use the graphical calculus described in \cite{kontsevich2021precalabiyau} in order to formulate the nondegeneracy condition of smooth Calabi-Yau structures. Given $\lambda = \lambda_0 + \lambda_1 u + \lambda_2 u^2 + \dots \in CC^-_d(A)$ a negative cyclic $d$-chain, let us express the nondegeneracy condition on $\lambda_0$ at the chain level: under the quasi-isomorphism $C_*(A) \simeq \Hom_{A-A}(A^!,A_\Delta)$, it must map to a quasi-isomorphism of $A$-bimodules, that is, it must have a quasi-inverse
\[ \alpha = ``(\lambda_0)^{-1}" \in  \Hom_{A-A}(A_\Delta,A^!) \simeq C^*_{(2)}(A )\]
where $C^*_{(2)}(A)$ is the complex of higher Hochschild cochains (with two outputs) with the quasi-isomorphism we explained in \cite[Sec.4.2]{kontsevich2021precalabiyau}.

In terms of morphisms of bimodules, there is evidently a composition map
\[ \Hom_{A-A}(A_\Delta,A^!) \otimes \Hom_{A-A}(A^!,A_\Delta) \to \Hom_{A-A}(A_\Delta,A_\Delta) \]
which, using these quasi-isomorphisms, can be represented by a map of complexes
\[ C^*_{(2)}(A) \otimes C_*(A) \to C^*(A), \]
for which we give the following explicit representative.  
\begin{lemma}
	The composition map can be represented by the map of complexes $(\alpha,\lambda_0) \mapsto \lambda_0 \circ \alpha \in C^*(A)$ given by the ribbon quiver
	
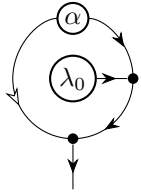
\[\begin{tikzpicture}[auto,baseline={([yshift=-.5ex]current bounding box.center)}]
	\node [vertex] (v3) at (0,0.8) {$\alpha$};
	\node [vertex] (v2) at (0,0) {$\lambda_0$};
	\node [bullet] (v4) at (0.8,0) {};
	\node [bullet] (v5) at (0,-0.8) {};
	\node (v6) at (0,-1.6) {};
	\draw [->-] (v2) to (v4);
	\draw [->-,shorten <=6pt] (0,0.8) arc (90:0:0.8);
	\draw [-w-,shorten <=6pt] (0,0.8) arc (90:270:0.8);
	\draw [->-] (0.8,0) arc (0:-90:0.8);
	\draw [->-] (v5) to (v6);
	\end{tikzpicture}\]
\end{lemma}
That is, we produce a map $C^*_{(2)}(A) \otimes C_*(A) \to C^*(A)$ given by plugging in $\alpha$ into the 2-valent vertex at the top, $\lambda_0$ into the source at the center and evaluate this diagram using the prescription in \cite[Sec.6.1.4]{kontsevich2021precalabiyau}. The fact that this represents the desired map follows from using the explicit descriptions of the quasi-isomorphisms, together with the calculations in \cite[Sec.2]{ganatra2013symplectic}.

Since we assume $A$ was strictly unital, there is a distinguished element $1 \in C^0(A)$, the \emph{unit cochain}, which only has a nonzero component of length zero, giving the unit element.\footnote{In the category case, that is, when $A$ has multiple objects, recall that the regions around our diagrams get labeled with objects, so the cochain $1$ just returns the identity morphism $1_X \in \End_A(X)$ when the region around the vertex is labeled by any object $X$.} Moreover, under the quasi-isomorphism $C^*(A) \simeq \Hom_{A-A}(A_\Delta,A_\Delta)$ the unit cochain maps to the identity.
\begin{proposition}\label{prop:chainLevelSmoothCY}
	Let $(A,\mu)$ be a smooth $A_\infty$-category. The elements $\lambda_0 \in  C_d(A)$ and $\alpha \in C^d_{(2)}(A)$ represent inverse classes if and only if they are closed under the relevant differentials, and there is an element $\beta \in C^{-1}(A)$ such that
	\[ [\mu,\beta] = \quad \tikzfig{
	\node [vertex] (v3) at (0,0.8) {$\alpha$};
	\node [vertex] (v2) at (0,0) {$\lambda_0$};
	\node [bullet] (v4) at (0.8,0) {};
	\node [bullet] (v5) at (0,-0.8) {};
	\node (v6) at (0,-1.6) {};
	\draw [->-] (v2) to (v4);
	\draw [->-,shorten <=6pt] (0,0.8) arc (90:0:0.8);
	\draw [-w-,shorten <=6pt] (0,0.8) arc (90:270:0.8);
	\draw [->-] (0.8,0) arc (0:-90:0.8);
	\draw [->-] (v5) to (v6);
	} \quad - \quad 
	\tikzfig{
	\node [vertex] (v) at (0,0) {$1$};
	\draw [->-] (v) to (0,-1.5);}\]
	where $[-,-]$ is the Gerstenhaber bracket on Hochschild cochains. In other words, $\lambda = \lambda_0 + \lambda_1 u^1 + \dots$ represents a smooth Calabi-Yau structure if and only if there are $\alpha$ and $\beta$ such that the first term $\lambda_0$ satisfies the equation above.
\end{proposition}
One of those directions obviously follows from the Lemma above; this is exactly the condition $[\lambda_0 \circ \alpha] = [\id]$ in $\Hom_{A\mh A}(A,A)$, so the equation says that $\alpha$ represents a right-inverse to $\lambda_0$.

We will now argue that it is also a left-inverse, once we assume that $A$ is smooth. Before that, let us present an analogy using a finite-dimensional vector space $V$ and its linear dual $V^\vee$: let $f:V\to V^\vee$ and $g: V^\vee \to V$ be linear maps such that $g \circ f = \id_V$. Then obviously both $f$ and $g$ have full rank, are bijective and so $f\circ g = \id_{V^\vee}$.

At the risk of boring the reader, let us give another proof for this easy fact: since $V$ is finite-dimensional, the canonical map $V \to (V^\vee)^\vee$ is an isomorphism, so dualize the composition $g\circ f$ and use this identification
\[ (V \overset{f}{\rightarrow} V^\vee \overset{g}{\rightarrow} V) \mapsto  (V^\vee \overset{f^\vee}{\leftarrow} V \overset{g^\vee}{\leftarrow} V^\vee) \]
to conclude that the composition $f^\vee \circ g^\vee$ is the identity on $V^\vee$. Note now that if we pick any symmetric bilinear form on $V$, we can identify the maps $f,g$ with matrices; the maps $f^\vee,g^\vee$ are then the transposes of those matrices. Now, in finite dimensions, every matrix is conjugate to its transpose. So $g$ also has a left-inverse which is constrained to be $f$ by the equation $f\circ g\circ f = f$.

We now explain the analog of this reasoning for our smooth category $A$, first by identifying the role of the transposed map.
\begin{lemma}
	Let $A$ be a smooth category, and $\alpha \in C^*_{(2)}(A) \cong \Hom_{A\mh A}(A^!,A)$. Let $\alpha^! \in \Hom_{A\mh A}(A^!,(A^!)^!)$ be the morphism obtained by taking bimodule duality. Upon identifying $(A^!)^! \cong A$, the class of $\alpha^!$ is represented by the $\ZZ/2$ rotation of the vertex of $\alpha$.
\end{lemma}
\begin{proof}
	This follows from some diagrammatic calculus as we developed in \cite{kontsevich2021precalabiyau}. Recall that for any 
	$A_\infty$-category $A$ we have a quasi-isomorphism $A_\Delta \otimes_{A\mh A} A_\Delta \overset{\sim}{\rightarrow} A_\Delta$; we regard an element of the former as a set of $A[1]$-arrows, traveling down a strip with an $A_\Delta$ arrow on each side:
	\[\tikzfig{
		\draw [->-] (0,1.5) to (0,0);
		\draw [->-] (1,1.5) to (1,0);
		\draw [->-] (0.25,1.3) to (0.25,0.2);
		\draw [->-] (0.5,1.3) to (0.5,0.2);
		\draw [->-] (0.75,1.3) to (0.75,0.2);
		\node at (0,1.8) {$A_\Delta$};
		\node at (1,1.8) {$A_\Delta$};
		\draw [decorate,decoration={brace,mirror,amplitude = 2mm}] (0.1,0.1) -- (0.9,0.1) node [midway,yshift=-0.8cm] {$T(A[1])$};
	}\]
	More precisely, an element of $A_\Delta \otimes_{A\mh A} A_\Delta \overset{\sim}{\rightarrow} A_\Delta$ is something we can input into the top of this diagram.
	
	Analogously, we represent an element of $A^!$ as a strip where the inner arrows travel \emph{up the strip}:
	\[\tikzfig{
		\draw [->-] (0,1.5) to (0,0);
		\draw [->-] (1,1.5) to (1,0);
		\draw [->-] (0.25,0.2) to (0.25,1.3);
		\draw [->-] (0.5,0.2) to (0.5,1.3);
		\draw [->-] (0.75,0.2) to (0.75,1.3);
		\node at (0,1.8) {$A_\Delta$};
		\node at (1,1.8) {$A_\Delta$};
		\draw [decorate,decoration={brace,mirror,amplitude = 2mm}] (0.1,0.1) -- (0.9,0.1) node [midway,yshift=-0.8cm] {$T(A[1])$};
	}\]
	Again, more precisely an element of $A^!$ is something we can input into the top of this diagram. This can be seen from the explicit representative for the left dual $A^!$ given in \cite[Def.2.40]{ganatra2013symplectic}. More generally, for any perfect $(A,A)$-bimodule $M$, its left dual is represented by the strip
	\[\tikzfig{
		\draw [->-] (0,1.5) to (0,0);
		\draw [->-] (1,0) to (1,1.5);
		\draw [->-] (0.25,0.2) to (0.25,1.3);
		\draw [->-] (0.5,0.2) to (0.5,1.3);
		\draw [->-] (0.75,0.2) to (0.75,1.3);
		\draw [->-] (2,1.5) to (2,0);
		\draw [->-] (1.25,0.2) to (1.25,1.3);
		\draw [->-] (1.5,0.2) to (1.5,1.3);
		\draw [->-] (1.75,0.2) to (1.75,1.3);
		\node at (0,1.8) {$A_\Delta$};
		\node at (1,1.8) {$M$};
		\node at (2,1.8) {$A_\Delta$};
		\draw [decorate,decoration={brace,mirror,amplitude = 2mm}] (0.1,0.1) -- (0.9,0.1) node [midway,yshift=-0.8cm,xshift=-0.2cm] {$T(A[1])$};
		\draw [decorate,decoration={brace,mirror,amplitude = 2mm}] (1.1,0.1) -- (1.9,0.1) node [midway,yshift=-0.8cm,xshift=0.2cm] {$T(A[1])$};
	}\]
	with an $M$ arrow going up in the middle.
	
	Under these identifications, given an element $\alpha \in C^*_{(2)}(A)$, the map $A^!\to A_\Delta$ it gives by dualizing is represented by the diagram
	\[\tikzfig{
		\node at (0,1.8) {$A_\Delta$};
		\node at (1.4,1.8) {$A_\Delta$};
		\node [vertex] (center) at (0.7,0.75) {$\alpha$};
		\draw [->-] (0,1.5) to (0,0.5);
		\draw [->-] (1.4,1.5) to (1.4,0.5);
		\draw [->-] (0,0.5) to (0,0);
		\draw [->-] (1.4,0.5) to (1.4,0);
		\draw [-w-] (center) to (1.4,0.75);
		\draw [->-] (center) to (0,0.75);
	}\]
	where as usual the white arrow marks the first output of $\alpha$.
	
	The map induced on the left duals $\alpha^!:(A^!)^! \to A^!$ is then given by reversing this diagram and inserting it in the place of the $M$ arrow. But since $A$ is smooth, $A^!$ is perfect and the map $A \to (A^!)^!$ is a quasi-isomorphism. So we can simplify the diagram obtained and conclude that the diagram
	\[\tikzfig{
		\node at (0,1.8) {$A_\Delta$};
		\node at (1.4,1.8) {$A_\Delta$};
		\node [vertex] (center) at (0.7,0.75) {$\alpha$};
		\draw [->-] (0,1.5) to (0,0.5);
		\draw [->-] (1.4,1.5) to (1.4,0.5);
		\draw [->-] (0,0.5) to (0,0);
		\draw [->-] (1.4,0.5) to (1.4,0);
		\draw [->-] (center) to (1.4,0.75);
		\draw [-w-] (center) to (0,0.75);
	}\]
	is also a representative for $\alpha^!$; comparing it to the previous diagram we have the desired statement.
\end{proof}

We are now ready to prove \cref{prop:chainLevelSmoothCY}.
\begin{proof}
	It remains to prove that if $\alpha$ and $\lambda_0$ satisfy the given equation, then $[\alpha]\circ[\lambda_0] = [\id_{A^!}]$. We use the analogy with the discussion about finite-dimensional vector spaces above: we already know that the composition
	\[ A_\Delta \overset{\alpha}{\rightarrow} A^! \overset{\lambda_0}{\rightarrow} A_\Delta \]
	is quasi-isomorphic to the identity on $A_\Delta$, so taking left duals we know that the composition
	\[ A^! \overset{\alpha^!}{\leftarrow} (A^!)^! \overset{(\lambda_0)^!}{\leftarrow} A^! \]
	is quasi-isomorphic to the identity on $A^!$. Thus $[\alpha^!]$ has a right-inverse.
	
	Consider now the element of $C^*_{(2)}(A)$ given by the diagram
	\[\tikzfig{
		\node [vertex] (x) at (0,0) {$\lambda_0$};
		\node at (2,0.3) {};
		\node at (-2,0.3) {};
		\node [vertex] (n) at (0,1) {$\alpha$};
		\node [vertex] (s) at (0,-1) {$\alpha$};
		\node [bullet] (ne) at (0.7,0.7) {};
		\node [bullet] (e) at (1,0) {};
		\node [bullet] (w) at (-1,0) {};
		\draw [->-] (x) to (ne);
		\draw [->-=1] (w) to (-2,0);
		\draw [->-=1] (e) to (2,0);
		\draw [-w-,shorten <=6pt] (0,1) arc (90:45:1);
		\draw [->-] (0.7,0.7) arc (45:0:1);
		\draw [->-,shorten <=6pt] (0,1) arc (90:180:1);
		\draw [-w-,shorten <=6pt] (0,-1) arc (-90:-180:1);
		\draw [->-,shorten <=6pt] (0,-1) arc (-90:0:1);
	}\]
	
	Using the diagrammatic calculus for ribbon quivers, when $\alpha$ and $\lambda_0$ are closed this element is cohomologous to the element given by
	\[ \tikzfig{
		\node [vertex] (v3) at (0,0.8) {$\alpha$};
		\node [vertex] (v2) at (0,0) {$\lambda_0$};
		\node [bullet] (v4) at (0.8,0) {};
		\node [bullet] (v5) at (0,-0.8) {};
		\node [bullet] (v6) at (0,-1.5) {};
		\node [vertex] (v1) at (-1,-1.5) {$\alpha$};
		\draw [->-] (v2) to (v4);
		\draw [-w-,shorten <=6pt] (0,0.8) arc (90:0:0.8);
		\draw [->-,shorten <=6pt] (0,0.8) arc (90:270:0.8);
		\draw [->-] (0.8,0) arc (0:-90:0.8);
		\draw [->-] (v5) to (v6);
		\draw [-w-] (v1) to (-2.2,-1.5);
		\draw [->-] (v1) to (v6);
		\draw [->-] (v6) to (+1.1,-1.5);
	}\]
	and therefore by the assumption that $\alpha$ and $\lambda_0$ satisfy the equation in the statement of the proposition, this element is also cohomologous to $\tikzfig{
		\node [vertex] (v1) at (0,0) {$\alpha$};
		\draw [-w-] (v1) to (-1,0);
		\draw [->-] (v1) to (1,0);
	}$, which we showed above represents the left dual map $\alpha^!$. A similar calculation shows that it is also cohomologous to $\alpha$, that is, $\tikzfig{
	\node [vertex] (v1) at (0,0) {$\alpha$};
	\draw [->-] (v1) to (-1,0);
	\draw [-w-] (v1) to (1,0);
}$. Therefore $[\alpha]$ also has a right-inverse, which is then constrained to be the same class as $[\lambda_0]$.
\end{proof}

\subsection{Examples}
So far, we have explained that given a smooth CY structure $\lambda$ on an $A_\infty$-category $A$, one can in principle find a chain-level representative, that is, a solution $\alpha$ to the equation in \cref{prop:chainLevelSmoothCY}, which we will later use to construct the desired pre-CY structure on $A$.

In general, finding an explicit solution for $\alpha$ may be difficult, since it involves solving an inverse function problem for a morphism between bimodules. However, in some specific cases of interest, it is possible to solve this problem explicitly.

\subsubsection{Chains on the based loop space of orientable manifolds}
Consider a pointed path-connected topological space $(X,x)$; concatenation of loops gives a morphism
\[ \Omega_x X \times \Omega_x X \to \Omega_x X\]
inducing a product on the complex of chains $C_*(\Omega_{x}X,\kk)$, making it into a dg algebra.

It has been long understood that structures on this algebra are intimately related to operations of string topology. In the 80s, work of Goodwillie \cite{goodwillie1985cyclic} and Burghelea and Fiedorowicz \cite{burghelea1986cyclic} gave an equivalence
\[ H_*(LX,\kk) \cong HH_*(C_*(\Omega_x X),\kk) \]
between the homology of the free loop space and the Hochschild homology of the dg algebra $A = C_*(\Omega_x X,\kk)$. This equivalence relate certain BV algebra structures on each side, constructed algebraically on the Hochschild complex and topologically on loop space homology.

Much has been written about this relation between Hochschild theory of $A$ and loop space operations, but here we would like to focus on the effect of Poincaré duality. Let $X$ be a $n$-dimensional ($\kk$-)Poincaré duality space with a choice of fundamental chain, that is, endowed with a $n$-chain $c_X \in H_n(X,\kk)$ such that the cap product $c_X \cap: H^*(X,\kk) \to H_{n-*}(X,\kk)$ is an isomorphism.

By now, it is a well-known fact that in that case $A = C_*(\Omega_x X,\kk)$ is a smooth CY algebra of dimension $n$. At the level of weak CY structures (that is, without the lift to negative cyclic homology), this is explicitly proven in \cite{abbaspour2013algebraic} following an argument of \cite{felix1995differential}, and the lift to negative cyclic complex is discussed in a draft \cite{cohencalabi} of Cohen-Ganatra. We summarize these results:
\begin{proposition}
	There is a map $\iota: C_*(X,\kk) \to CC^-_*(A)$ such that if $c_X \in C_d(X,\kk)$ such that $[c_X] \in H_d(X,\kk)$ is the fundamental class of $X$ then the image $\iota(c_X)$ is a smooth CY structure of dimension $d$ on $A = C_*(\Omega_x X,\kk)$.
\end{proposition}
The map $\iota$ (or rather, its composition with the canonical map $CC^-_*(A)\to C_*(A)$ to Hochschild chains) should be seen as an algebraic incarnation of the map $X \hookrightarrow LX$ given by inclusion of $X$ as constant loops. 

We would like to have a chain-level inverse for the image of $c_X$, that is, a appropriate value for the $\alpha$ vertex in the statement of \cref{prop:chainLevelSmoothCY}; by the result above it is always possible to find one, but doing so algorithmically for a general Poincar\'e duality space turns out to be quite involved. We will leave the full description for future work, but at least we can present the general lines and a simple example.

\subsubsection{Path dg category}\label{sec:pathDg}
We start by replacing the algebra $A$ by an equivalent dg category, with a more local, combinatorial description. Let $\Lambda$ be a simplicial complex endowed with a total ordering on its set $\Lambda_0$ of vertices; this defines a corresponding simplicial set. We will denote an $n$-simplex $\sigma$ in $\Lambda$ with vertices $v_0,\dots,v_n$ (in order) by the notation $(v_0\dots v_n)_\sigma$, or more simply $(v_0\dots v_n)$ when there is no ambiguity.

\begin{definition}
	The path dg category $P_\Lambda$ of the simplicial complex $\Lambda$ has as object set $\Lambda_0$ and as morphism space between vertices $s$ and $t$, the graded $\kk$-vector space spanned by symbols of the form
	\[ (v^1_0 v^1_1 \dots v^1_{n_1})_{\sigma_1} * (v^2_0 \dots v^2_{n_2})_{\sigma_2} * \dots * (v^j_1 v^j_0)_{\sigma_i}^{-1} \dots (v^N_0 \dots v^N_{n_N})_{\sigma_N}, \]
	where $v^k_{n_j} = v^{k+1}_0$, $v^1_0 = s$ and $v^N_{n_N} = t$, modulo the relation generated by
	\[ x * (u_0 u_1)_\sigma * (u_0 u_1)_\sigma^{-1} * y = x * y \]
	where $x$ and $y$ are any sequences as above. A generator of the form above is placed in degree $\sum_k (1-n_k)$. If $u=v$ we also add the identity morphism $e_u$.
	
	That is, generators are composable sequences whose elements are either simplices of $\Lambda$ (of any nonzero dimension), or formal inverses of 1-simplices. We place $n$-simplices in degree $n-1$, and endow this vector space with the differential given by
	\[ d(v_0 \dots v_n) = \sum_{i=1}^{n-1} (-1)^i \left( (v_0\dots \hat{v_i} \dots v_n) - (v_0\dots v_i)*(v_i \dots v_n) \right) \]
	and by the Leibniz rule with respect to $*$.
\end{definition}

These compositions of simplices are sometimes referred to as `necklaces' in the literature (not to be confused with the `necklace bracket' we defined in \cite{kontsevich2021precalabiyau}), as one can imagine the simplices as beads in an unfastened necklace. It was suggested by one of us in \cite{kontsevich2009symplectic} that these categories of necklaces (after formally inverting 1-simplices as we did above) give models for based loop spaces. This relation was studied in \cite{dugger2011mapping,dugger2011rigidification,rivera2018cubical,rivera2019adams,hess2010loop}; in particular, \cite{rivera2019combinatorial} describes in detail the localization at 1-simplices to give a functor $\hat{\mathfrak{C}}$ from simplicial sets into some category of `necklical sets', making this assignment functorial.

Our dg category $P_\Lambda$ is just the image on objects of this functor, composed with taking $C_*(-,\kk)$. We will not need the full simplicial description developed in the references above, so we summarize:
\begin{proposition}
	If $(X,x)$ is a pointed connected topological space such that $X$ is homotopy equivalent to the geometric realization of $\Lambda$, the dg category $P_\Lambda$ is equivalent to the dg algebra $C_*(\Omega_x X)$, seen as a dg category with a single object, in the sense that there is a zig-zag of essentially surjective dg functors, all inducing quasi-isomorphisms on morphism complexes.
\end{proposition}
Concretely, each morphism $x\to y$ of degree $-n$ describes an $n$-dimensional \emph{cube} in the space of paths between $x$ and $y$; for example, the morphism given by $(v_0 v_1) * (v_1 v_2 v_3)$, which has boundary $-(v_0 v_1) * (v_1 v_3) + (v_0 v_1)*(v_1 v_2)*(v_2 v_3)$, describes the family of paths over the 1-cube (that is, the interval) which sweeps from the path $v_0 \to v_1 \to v_3$ to the path $v_0 \to v_1 \to v_2 \to v_3$, deforming it across the 2-simplex $(v_1 v_2 v_3)$. The comparison result with the algebra $C_*(\Omega_x X)$ above relies on the fact that this cubical complex computes the homology of the based loop space.

\subsubsection{Inclusion of constant loops}
One can use this model to describe the map $\iota$, which corresponds to the inclusion of constant loops into the free loop space. An explicit representative for a map
\[ C_*(X) \to CC^-_*(C_*(\Omega_x X)) \]
from simplicial chains on $X$ is described in \cite[App.B]{abouzaid2011cotangent}, following constructions in the classical work of Adams \cite{adams1956cobar}. We can rephrase this in terms of the dg category $P_\Lambda$:

\begin{proposition}
	There is a chain map $\iota: C_*(\Lambda) \to CC^-_*(P_\Lambda)$, whose composition with the canonical map $CC^-_*(P_\Lambda) \to C_*(P_\Lambda)$ agrees with the the map induced by the inclusion of constant loops, under the identification $HH_*(P_\Lambda) \cong H_*(L(|\Lambda|)$.
\end{proposition}
\begin{proof}
	One can construct this map locally on each simplex, and inductively in dimension. We start by sending each 0-simplex $(v) \mapsto e_v$ (that is, the length 1 zero-chain in $CC^-_*(P_\Lambda)$ consisting solely of the identity morphism $e_v \in P_\Lambda(v,v)$).
	
	Suppose now that we have the map for all simplices of dimension up to $n-1$, and moreover, that for each such simplex $\tau$, the image lies in the subcomplex $CC^-_*(P_{\overline{\tau}}) \subseteq CC^-_*(P_\Lambda)$ on its closure $\overline{\tau}$. Let $\sigma$ be some $n$-simplex; in order to extend the map to $\sigma$ it is sufficient to find some degree $n$ element $x$ of $CC^-_*(P_{\overline{\sigma}})$ such that
	\[ (b+uB)x = \iota(\del \sigma) \]
	But by assumption $(b+uB)\iota(\del \sigma) = 0$ so $[\iota(\del \sigma)]$ is a class in $HC^-_{n-1}(P_{\overline{\sigma}})$ which is zero for $n \ge 1$ since $\overline\sigma$ is contractible.
\end{proof}
The argument above, however, does not give us a way to explicitly construct a representative for $\iota(\sigma)$, which may involve quite complicated expressions for higher dimensions. For a 1-simplex $(v_0 v_1)$, we can choose its corresponding Hochschild chain to be $(v_0 v_1)[(v_0 v_1)^{-1}]$, which for simplicity of notation we denote $01[10]$. This lifts to the negative cyclic chain
\[ 01[10] - 01[10|01|10] u + 01[10|01|10|01|10] u^2 \dots \]
To a 2-simplex $(v_0 v_1 v_2)$, we need to then assign a Hochschild 2-chain whose boundary is $01[10] + 12[21] - 02[20]$; one possible choice is
\[ 01*12[21|10] - 01*12[20|02*21*10]+01*12*20[012*21*10]-012*20*012*21*10. \]
As for the chain we associated to the 1-simplex, the Hochschild chain above has some lift to a negative cyclic chain, whose expression is not particularly enlightening, and so on.

Now, given a simplicial triangulation $\Lambda$ of a Poincar\'e duality space $X$, we can look at the image $\lambda = \iota(c_X)$ of its fundamental chain and find the inverse $\alpha$ of its $u=0$ component $\lambda_0$, that is, the element $\alpha$ in $C^*_{(2)}(P_\Lambda)$ which satisfies the equation in the statement of \cref{prop:chainLevelSmoothCY}.

\subsubsection{The circle}\label{sec:theCircle}
Let us illustrate how to do this for the simplest non-trivial case, that is, for the circle. We pick a triangulation that exhibits it as the boundary of the 2-simplex $(v_0 v_1 v_2)$. The fundamental chain $(v_0 v_1) + (v_1 v_2) - (v_0 v_2)$ maps to the Hochschild chain
\[ \lambda_0 = 01[10] + 12[20] - 02[20], \] 
again using our shorthand notation.

We are looking for an inverse $\alpha$ to $\lambda_0$. Recall that this is a vertex that receives any number of $A[1]$ arrows above and below, and outputs two arrows in $A$; if $A$ were an $A_\infty$-algebra this would be the space $\Hom(T(A[1]\otimes T(A[1]),A\otimes A)$; as our chosen $A$ has multiple objects one replaces the $A$ factors by morphism spaces and sum over objects.

We proceed inductively on the length of the inputs. Since $A$ is concentrated in non-negative degrees and we want $\alpha$ of (cohomological) degree +1, its component $\alpha^{0,0}$ with zero entries on both sides is necessarily zero.

The first non-trivial degree to be specified is the component $\alpha^{0,1}$, that is, with zero inputs on top and one input on the bottom. Note that since $A$ is a category, and not an algebra, the regions in a diagram are labeled by objects, which we will denote by writing $\underline{0},\underline{1}$ and $\underline{2}$ for each of the objects of $A$ (vertices of the triangle).

Recall that our convention is to use the \emph{reduced} Hochschild complex, therefore the element $\alpha$ evaluates to zero whenever one of the inputs is an identity morphism. Now, every non-identity morphism in $A$ is either `counter-clockwise' (for example, the morphisms $(01),(01*12),(20)$ etc.) or `clockwise' (for example, $(10),(12*20*01)$ etc.)

Let $P: i\to j$ be one of these morphisms, that is, a path from some vertex $i$ to a vertex $j$. We define
\[\tikzfig{
	\node [vertex] (v1) at (0,0) {$\alpha$};
	\draw [-w-] (v1) to (-1,0);
	\draw [->-] (v1) to (1,0);
	\draw [->-] (0,-1) to (v1);
	\node at (0,0.5) {$\underline{k}$};
	\node at (0,-1.2) {$P$};
	\node at (0.6,-0.5) {$\underline{i}$};
	\node at (-0.6,-0.5) {$\underline{j}$};
} = \begin{cases} \frac{1}{2}\left(\delta_{jk} e_k \otimes P - \delta_{ik}P \otimes e_k \right), \text{\ if\ } P \text{\ counter-clockwise} \\
-\frac{1}{2}\left(\delta_{jk} e_k \otimes P - \delta_{ik}P \otimes e_k \right), \text{\ if\ } P \text{\ clockwise} \end{cases}
\]
where $\delta_{\dots}$ is the usual delta function on the set of pairs of vertices.

\begin{proposition}
	The prescription above extends uniquely to a closed element $\alpha \in C^1_{(2)}(A)$, symmetric under the $\ZZ/2$ action, which moreover satisfies the equation
	\[\tikzfig{
		\node [vertex] (v3) at (0,0.8) {$\alpha$};
		\node [vertex] (v2) at (0,0) {$\lambda_0$};
		\node [bullet] (v4) at (0.8,0) {};
		\node [bullet] (v5) at (0,-0.8) {};
		\node (v6) at (0,-1.6) {};
		\draw [->-] (v2) to (v4);
		\draw [->-,shorten <=6pt] (0,0.8) arc (90:0:0.8);
		\draw [-w-,shorten <=6pt] (0,0.8) arc (90:270:0.8);
		\draw [->-] (0.8,0) arc (0:-90:0.8);
		\draw [->-] (v5) to (v6);
	} \quad = \quad 
	\tikzfig{
		\node [vertex] (v) at (0,0) {$1$};
		\draw [->-] (v) to (0,-1.5);}
	\]
\end{proposition}
\begin{proof}
	For the element $\alpha$ to be closed, it needs to satisfy compatibility equations such as
	\[ d \left( \tikzfig{
		\node [vertex] (v1) at (0,0) {$\alpha$};
		\draw [-w-] (v1) to (-1,0);
		\draw [->-] (v1) to (1,0);
		\draw [->-] (0.5,-1) to (v1);
		\draw [->-] (-0.5,-1) to (v1);
		\node at (0,0.5) {$\underline{\ell}$};
		\node at (0.6,-1.2) {$P_1$};
		\node at (-0.6,-1.2) {$P_2$};
		\node at (0.6,-0.5) {$\underline{i}$};
		\node at (0,-0.7) {$\underline{j}$};
		\node at (-0.6,-0.5) {$\underline{k}$};
	} \right)\ = \quad \tikzfig{
		\node [vertex] (v1) at (0,0) {$\alpha$};
		\draw [-w-] (v1) to (-1,0);
		\draw [->-] (v1) to (1,0);
		\draw [->-] (0,-1) to (v1);
		\node at (0,0.5) {$\underline{\ell}$};
		\node at (0,-1.2) {$P_1 * P_2$};
		\node at (0.6,-0.5) {$\underline{i}$};
		\node at (-0.6,-0.5) {$\underline{k}$};
	} \quad - \quad \tikzfig{
		\node [vertex] (v1) at (0,0) {$\alpha$};
		\node [bullet] (v2) at (-0.8,0) {};
		\draw [-w-] (v1) to (v2);
		\draw [->-] (v2) to (-1.4,0);
		\draw [->-] (-0.8,-1) to (v2);
		\draw [->-] (v1) to (1,0);
		\draw [->-] (0.5,-1) to (v1);
		\node at (0,0.5) {$\underline{\ell}$};
		\node at (0.6,-1.2) {$P_1$};
		\node at (-0.8,-1.2) {$P_2$};
		\node at (0.7,-0.5) {$\underline{i}$};
		\node at (-0.3,-0.6) {$\underline{j}$};
		\node at (-1.2,-0.5) {$\underline{k}$};
	} \quad - \quad \tikzfig{
		\node [vertex] (v1) at (0,0) {$\alpha$};
		\node [bullet] (v2) at (0.8,0) {};
		\draw [-w-] (v1) to (-1,0);
		\draw [->-] (0.8,-1) to (v2);
		\draw [->-] (v1) to (v2);
		\draw [->-] (v2) to (1.4,0);
		\draw [->-] (-0.5,-1) to (v1);
		\node at (0,0.5) {$\underline{\ell}$};
		\node at (0.8,-1.2) {$P_1$};
		\node at (-0.6,-1.2) {$P_2$};
		\node at (1.2,-0.5) {$\underline{i}$};
		\node at (0.3,-0.6) {$\underline{j}$};
		\node at (-0.7,-0.5) {$\underline{k}$};
	}
\]
	together with a similar equation with one input on the top and one input on the bottom.
	
	Since $A$ is concentrated in degree zero, any element of $\alpha$ with two or more inputs vanishes, so the left-hand side of these equations is always zero; we must check that the right-hand side is zero, which we can explicitly calculate.
	
	Now, it remains to check that this element is an inverse to the Hochschild chain $\lambda_0 = 01[10] + 12[20] - 02[20]$. Let us calculate the Hochschild cochain given by the diagram
	\[\tikzfig{
		\node [vertex] (v3) at (0,0.8) {$\alpha$};
		\node [vertex] (v2) at (0,0) {$\lambda_0$};
		\node [bullet] (v4) at (0.8,0) {};
		\node [bullet] (v5) at (0,-0.8) {};
		\node (v6) at (0,-1.6) {};
		\draw [->-] (v2) to (v4);
		\draw [->-,shorten <=6pt] (0,0.8) arc (90:0:0.8);
		\draw [-w-,shorten <=6pt] (0,0.8) arc (90:270:0.8);
		\draw [->-] (0.8,0) arc (0:-90:0.8);
		\draw [->-] (v5) to (v6);
	}\]
	Since $A$ is a dg category, all the higher structure maps $\mu^{\ge 3}_A$ are zero, so the only nontrivial terms occur when the $A[1]$-arrow coming from $\lambda_0$ lands in $\alpha$. We calculate the values of this cochain on zero inputs; for example, when we label the region around the diagram with the object $0$, we have
	\[\tikzfig{
		\node [vertex] (v3) at (0,0.8) {$\alpha$};
		\node [bullet] (v4) at (0.8,0) {};
		\node [bullet] (v5) at (0,-0.8) {};
		\node (v6) at (0,-1.6) {};
		\draw [->-] (0.2,-0.2) to (v4);
		\draw [->-] (0,0.5) to (v3);
		\node at (0,0.1) {$(10)$};
		\node at (0,-0.4) {$(01)$};
		\node at (0,1.3) {$\underline{0}$};
		\draw [->-,shorten <=6pt] (0,0.8) arc (90:0:0.8);
		\draw [-w-,shorten <=6pt] (0,0.8) arc (90:270:0.8);
		\draw [->-] (0.8,0) arc (0:-90:0.8);
		\draw [->-] (v5) to (v6);
	} \quad + \quad \tikzfig{
		\node [vertex] (v3) at (0,0.8) {$\alpha$};
		\node [bullet] (v4) at (0.8,0) {};
		\node [bullet] (v5) at (0,-0.8) {};
		\node (v6) at (0,-1.6) {};
		\draw [->-] (0.2,-0.2) to (v4);
		\draw [->-] (0,0.5) to (v3);
		\node at (0,0.1) {$(21)$};
		\node at (0,-0.4) {$(12)$};
		\node at (0,1.3) {$\underline{0}$};
		\draw [->-,shorten <=6pt] (0,0.8) arc (90:0:0.8);
		\draw [-w-,shorten <=6pt] (0,0.8) arc (90:270:0.8);
		\draw [->-] (0.8,0) arc (0:-90:0.8);
		\draw [->-] (v5) to (v6);
	} \quad - \quad \tikzfig{
		\node [vertex] (v3) at (0,0.8) {$\alpha$};
		\node [bullet] (v4) at (0.8,0) {};
		\node [bullet] (v5) at (0,-0.8) {};
		\node (v6) at (0,-1.6) {};
		\draw [->-] (0.2,-0.2) to (v4);
		\draw [->-] (0,0.5) to (v3);
		\node at (0,0.1) {$(20)$};
		\node at (0,-0.4) {$(02)$};
		\node at (0,1.3) {$\underline{0}$};
		\draw [->-,shorten <=6pt] (0,0.8) arc (90:0:0.8);
		\draw [-w-,shorten <=6pt] (0,0.8) arc (90:270:0.8);
		\draw [->-] (0.8,0) arc (0:-90:0.8);
		\draw [->-] (v5) to (v6);
	}\]
	which evaluates to $\frac{1}{2} e_0*(10)*(01) + 0 + \frac{1}{2} (02)*(20)*e_0 = e_0$, and similarly for the diagrams with the outside region labeled with $\underline{1}$ or $\underline{2}$. Moreover, this diagram evaluates to zero on any input of length $\ge 1$; it is therefore exactly the unit cochain in $C^0(A)$.
\end{proof}

We note that the guess for the element $\alpha$ above can be derived from a smaller set of data by using the closedness condition: we can just specify that $\alpha$ gives some sort of `local' pairing, by setting
\[\tikzfig{
	\node [vertex] (v1) at (0,0) {$\alpha$};
	\draw [-w-] (v1) to (-1,0);
	\draw [->-] (v1) to (1,0);
	\draw [->-] (0,-1) to (v1);
	\node at (0,0.5) {$\underline{0}$};
	\node at (0,-1.2) {$(10)$};
	\node at (0.6,-0.5) {$\underline{1}$};
	\node at (-0.6,-0.5) {$\underline{0}$};
} = \frac{1}{2}e_0\otimes (10), \quad \tikzfig{
	\node [vertex] (v1) at (0,0) {$\alpha$};
	\draw [-w-] (v1) to (-1,0);
	\draw [->-] (v1) to (1,0);
	\draw [->-] (0,-1) to (v1);
	\node at (0,0.5) {$\underline{0}$};
	\node at (0,-1.2) {$(02)$};
	\node at (0.6,-0.5) {$\underline{0}$};
	\node at (-0.6,-0.5) {$\underline{2}$};
	} = -\frac{1}{2}(02)\otimes e_0, \quad \tikzfig{
	\node [vertex] (v1) at (0,0) {$\alpha$};
	\draw [-w-] (v1) to (-1,0);
	\draw [->-] (v1) to (1,0);
	\draw [->-] (0,-1) to (v1);
	\node at (0,0.5) {$\underline{0}$};
	\node at (0,-1.2) {$(21)$};
	\node at (0.6,-0.5) {$\underline{2}$};
	\node at (-0.6,-0.5) {$\underline{1}$};
	}=0
\]
and analogously for $\underline{1},\underline{2}$ above; that is, assigning zero when the simplex above and the simplex below do not intersect, and some appropriately signed local pair of paths when they do. In this sense, our expression for the element $\alpha$ is `localized' to a neighborhood of the diagonal in the product space $S^1 \times S^1$.

This smooth Calabi-Yau structure on chains on $\Omega S^1$ is well-known in the literature; for a recent explicit description of this structure from another angle, see the recent works \cite{bozec2021calabiyau1,bozec2021calabiyau2}. Under the equivalence between our dg category $A$ and the dg algebra $k[x^{\pm 1}]$, the element above $\lambda_0$ maps to the Hochschild chain cohomologous to $x^{-1}[x]$, corresponding to the noncommutative form $x^{-1} d_{dR}x$ in the description of \cite[Sec.3]{bozec2021calabiyau1}.

\section{Noncommutative Legendre transform}\label{sec:ncLegendre}
We recall from \cite{kontsevich2006notes,kontsevich2021precalabiyau} that in the language of noncommutative geometry, an $A_\infty$-algebra $(A,\mu)$ can be thought of as a noncommutative pointed dg manifold $X_A$ with an integrable vector field $Q_\mu$; the space $C^*_{[d]}(A)$ where pre-CY structures of dimension $d$ live is then the space of shifted polyvector fields on $X_A$, with the necklace bracket playing the role of Schouten-Nijenhuis bracket. 

A pre-CY structure $m$ then satisfies the (quadratic) Maurer-Cartan equation $[m,m] = 0$. We denote by $\cM_\mathrm{pre-CY}$ the space of such solutions; at a given point $m$ this space has tangent complex given by
\[ T_m \cM_\mathrm{pre-CY} = (C^*_{[d]}(A), [m,-]_\nec), \]
that is, polyvector fields with differential given by necklace bracket with $m$.

As mentioned in the introduction, we will eventually construct a map that takes a smooth CY structure, represented by some negative cyclic chain $\lambda \in CC^-_*(A)$, and produces a point of $\cM_\mathrm{pre-CY}$; in the process we will construct another map in the inverse direction, and in \cref{sec:simplicial} we argue that this map is `one-to-one', in a certain homotopical sense. 

In this section we will explain why such a relation should be seen as a noncommutative analog of the Legendre transform, and then we will build this transform using the formalism of ribbon quivers developed in \cite{kontsevich2021precalabiyau}. But first let us take a digression through the theory of usual Legendre transforms.

\subsection{Commutative and odd Legendre transform}

\subsubsection{Legendre transform on vector bundles}
Recall that the classical Legendre transform can also be defined (fiberwise) on a vector bundle. Let $M$ be a manifold and $E \to M$ an $N$-dimensional vector bundle, with a real-valued smooth function $L: E \to \RR$, fiberwise convex. For simplicity we assume $L$ is bounded below by some positive-definite quadratic function. We describe the Legendre transform in two steps: we first take the \emph{fiberwise derivative}
\[ FL : E \to E^\vee, \]
defined in dual coordinates by 
\[ FL(x_1,\dots,x_N,v_1,\dots,v_N) = (x_1,\dots,x_N,\del L/\del v_1,\dots,\del L/ \del v_N)\]
which gives a diffeomorphism, given our assumptions on $L$.

We then construct the \emph{energy function} associated to $L$ given by
\[ e_L = \sum_i v_i \frac{\del L}{\del v_i} - L \]
and then we can define the Legendre transform $H:E^\vee \to \RR$ by
\[ H = \cL(L) := e_L \circ (FL)^{-1} \]
which, using the pairing between $E$ and $E^\vee$, can also be expressed on each fiber by the classical Legendre transform
\[ H_x(p) = \mathrm{crit.value}_{v} (vp - L_x(v)). \]

This function has the property that $FH$ is the inverse diffeomorphism to $FL$, and $\cL(H) = L$. Moreover, we have the following fact.
\begin{proposition}\label{prop:tangentLegendre}
	The pullback under the fiberwise derivative corresponding to $H$ is minus the variational derivative of the Legendre transform at $L$, that is:
	\[ (FH)^* = - \frac{\delta \cL(f)}{\delta f}|_{f=L}. \] 
\end{proposition}
\begin{proof}
	Note that $(FH)^*$ is a ring isomorphism on functions $\cO(E) \to \cO(E^\vee)$. We vary $L \to L + \delta L$ and calculate the variation $H \to H + \delta H$ by expanding the relation $\cE_{H+\delta H} = (L + \delta L) \circ F(H + \delta H)$ to first order.
\end{proof}

\subsubsection{Odd Legendre transform}\label{sec:oddLegendre}
We now discuss an analogue of the classical Legendre transform, but that now goes between odd tangent and cotangent bundles. A description of this odd Legendre transform appears in the early work \cite{alexandrov1997geometry}, and its idea appears as a motivation for the discussion in \cite{pantev2013shifted}, relating shifted symplectic structures to nondegenerate shifted Poisson structures.

Let us denote by $\Pi T M$ the odd cotangent bundle of $M$, and its dual $\Pi T^* M$ its odd tangent bundle. In precise terms, we will consider maps between the spaces of (formal) functions on those bundles, namely the space of polyvector fields
\[ \cO(\Pi T^* M) = \wedge^* T M = \cO(M)[\alpha_i] \]
where the anticommuting variables $\alpha_i$ represent the vector fields $\del/\del x^i$, and the space of differential forms
\[ \cO(\Pi T M) = \wedge^* \Omega^1 \cO(M)[\beta_i] \]
where the anticommuting variables $\beta^i$ represent the one-forms $dx^i$. We make the convention that the degree of $\alpha_i$ is $+1$ and of $\beta_i$ is $-1$.

We will write a $p$-vector field in coordinates as 
\[ P = \frac{1}{p^!}P^{i_1,\dots,i_p} \alpha_{i_1} \dots \alpha_{i_p} \]
using the summation convention for repeated indices. We can also regard this as an element of the tensor algebra of $TM$ by using the antisymmetric embedding:
\[ P = \frac{1}{p^!} P^{j_1,\dots,j_p} \delta^{i_1,\dots,i_p}_{j_1,\dots,j_p} \alpha_{i_1} \dots \alpha_{i_p} \]
using the Kronecker symbol giving the sign of the permutation. This allows us to calculate derivatives: we have that the derivative $\del/\del\alpha_i$ acts as
\[ \frac{\del}{\del \alpha_{i_n}} \alpha_{i_1} \dots \alpha_{i_p} = p (-1)^n \alpha_{i_1} \dots \hat\alpha_{i_n} \alpha_{i_p} \]
where the hat denotes omission. 

Let $\gamma \in \cO(\Pi T^* M)$ be a polyvector field \emph{without degree one term}, which we write in coordinates as
\[ \gamma(x) =  \frac{\gamma^{ij}_2(x)}{2!}\alpha_i\alpha_j + \frac{\gamma^{ijk}_3(x)}{3!}\alpha_i\alpha_j\alpha_k + \dots \]
where $\gamma_p(x)$ is some $p$-tensor depending on $x$. Let us assume that the degree two term $\gamma_2$ is given by a positive definite matrix.

In terms of functions on the odd space $\Pi T^*M$, $\gamma$ should be seen as a convex function with a fiberwise critical point at the `locus $\alpha = 0$'. The odd fiberwise derivative should then send a neighborhood of this locus to the neighborhood of the `locus $\beta = 0$', and the odd Legendre transform should produce a formal series in the variables $\beta_i$, that is, a differential form.

Let us calculate the odd Legendre transform $F\gamma$: first we write
\[ \beta^i = \frac{\del \gamma}{\del \alpha_i}  = \gamma_2^{ij} \alpha_j +  \frac{\gamma_3^{ijk}}{2!} \alpha_j\alpha_k + \dots \]
and then invert this equation, expressing each $\alpha_i$ in terms of the $\beta$ variables, as a function $f_i(\gamma,\beta)$, depending on the matrices $\gamma_p$ for all $p$. This is possible if and only if the matrix $\gamma_2$ is invertible: denoting $\{M_{ij}\}$ for its inverse we can calculate $f_i$ by an iterative procedure. To second order this gives
\[ \alpha_i = f_i(\gamma,\beta) = M_{ij} \beta^j + M_{ij}M_{ka}M_{lb}\frac{\gamma_3^{jkl}}{2!}\beta^a\beta^b + \dots \]
In other words, the functions $f_i$ are the data of the inverse induced map on functions $((F\gamma^*)^{-1})$.

We calculate now that the energy function associated to $\gamma$ is given by
\[ e_\gamma = \alpha_i\frac{\del \gamma}{\del\alpha_i} - \gamma= \frac{\gamma_2^{ij}}{2!}\alpha_i\alpha_j + 2\times \frac{\gamma_3^{ijk}}{3!} \alpha_i\alpha_j\alpha_k + \dots + (p-1)\times \frac{\gamma_p^{ij\dots}}{p!}\alpha_i \alpha_j \dots + \dots \]
that is, we just multiply each degree $p$ term by $p-1$. The odd Legendre transform $\lambda = \cL(\gamma)$ is then calculated by substituting, in the expression above, $\alpha_i \mapsto  f_i(\gamma,\beta)$.
\begin{proposition}\label{prop:oddLegendre}
	The polyvector field $\gamma$ satisfies $[\gamma,\gamma] = 0$, where $[,]$ is the Schouten-Nijenhuis bracket, if and only if $d\lambda = 0$. 
\end{proposition}
\begin{proof}
	We calculate
	\begin{align*}
		d\lambda &= \beta^i \frac{\del \lambda}{\del x^i} = \beta^i \frac{\del}{\del x^i} \left( f_j(\gamma,\beta) \beta_i - \gamma(\alpha_j \mapsto f^j(\gamma,\beta)) \right) \\
		&= \beta^i \frac{\del f_j}{\del x^i} \beta^j - \beta^i \frac{\del\gamma}{\del x^i} - \beta^i \frac{\del \gamma}{\del \alpha_i}\frac{\del f_i}{\del x} = - \beta^i \frac{\del\gamma}{\del x^i}
	\end{align*}
	along the locus $\alpha = f(\gamma,\beta)$ this is equal to
	\[ - \frac{\del \gamma}{\del \alpha}\frac{\del \gamma}{\del x_i} = -[\gamma,\gamma] \]
	proving the proposition.
\end{proof}

Performing the opposite operations, we see that odd Legendre transform gives a bijection between closed forms and polyvector fields satisfying the Maurer-Cartan-type equation $[\gamma,\gamma] = 0$. 

Let us now consider a slightly more general situation, where $\gamma$ may have a \emph{small} nonvanishing first order term, that is:
\[ \gamma = \gamma^i_1 \alpha_i + \frac{\gamma^{ij}_2(x)}{2!}\alpha_i\alpha_j + \frac{\gamma^{ijk}_3(x)}{3!}\alpha_i\alpha_j\alpha_k + \dots\]
In this case, the locus $\alpha = 0$ is not critical, so we should not expect it to be sent to a neighborhood of $\beta = 0$ by the odd Legendre transform. Instead it will be sent to a neighborhood of the `point' of $\Pi T X$ with fiber coordinates given by $\del\gamma/\del\alpha_i|_{\alpha_i =0} = \gamma_1^i$. 

More precisely: given a form $\lambda$ in terms of the $\beta_i$, we shift the coordinates to $\tilde\beta^i = \beta^i - \gamma_1^i$; the shifted form will then satisfy \cref{prop:oddLegendre} with respect to the differential $\tilde d$ in the variables $\tilde\beta^i$. We calculate this in terms of the original variables, expanding to linear order in $\gamma_1$
\begin{align*}
	 \tilde d \tilde \lambda &= \tilde \beta^i \frac{\del\tilde\lambda_0}{\del x^i} + \tilde \beta^i \frac{\del\tilde\lambda^1_j}{\del x^i}\tilde\beta^j + \dots \\
	 &= \beta^i \frac{\del}{\del x^i} \left(\lambda^0 + \lambda^1_j\beta^j + \dots\right) - \gamma_i^i \frac{\del}{\del x^i} \left(\lambda^0 + \lambda^1_j\beta^j + \dots\right) \\
	 &- \beta^i \frac{\del}{\del x^i} \left(\lambda^1_j \gamma_i^j + \dots\right) + O((\gamma_1)^2) \\
	 &=  (d-Lie_{\gamma_1}) \lambda + O((\gamma_1)^2)
 \end{align*}

Therefore we have the following result:
\begin{corollary}
	If the polyvector field $\gamma$ satisfies the equation $[\gamma,\gamma] = 0$, then its odd Legendre transform $\lambda$ satisfies $(d-Lie_{\gamma_1})\lambda = 0$ up to higher corrections in $\gamma_1$.
\end{corollary}

Note that comparing the result above with the noncommutative analogies in \cite{kontsevich2006notes} motivates the statement of the correspondence we mentioned in the introduction (\cref{eqn:main}).

\subsubsection{Inverting the Legendre transform}\label{sec:inverting}
Before we move on to the noncommutative world, let us discuss one last thing about the odd Legendre transform. For simplicity we return to the case where $\gamma_1 = 0$. Given only the degree two matrix $\gamma^{ij}_2$ of $\gamma$, and the expression for the Legendre transform $\lambda = \cL(\gamma)$ (of all of $\gamma$), how can one compute $\gamma_3,\gamma_4,\dots$ \emph{without} explicitly writing down the inverse Legendre transform?

This question may seem silly since in this case we could just directly write down the inverse Legendre transform. But in the noncommutative case we will not have an explicit inverse; instead we will use a version of the implicit function theorem. That is, we know by definition that the pair $(\gamma,\lambda)$ solves the equation
\[ e_\gamma  = (F\gamma)(\lambda) \]
By calculating the expressions for the `energy function' of $\gamma$ and the fiberwise derivative  we have:
\begin{align*} \frac{\gamma_2^{ij}}{2!}\alpha_i\alpha_j + 2 \frac{\gamma_3^{ijk}}{3!} \alpha_i\alpha_j\alpha_k + \dots &= 
\lambda^2_{i_1i_2}\left(\gamma_2^{i_1j_1} \alpha_{j_1} +  \frac{\gamma_3^{i_1j_1k_1}}{2!} \alpha_{j_1}\alpha_{k_1} + \dots\right) \times \\
&\left(\gamma_2^{i_2j_2} \alpha_{j_2} +  \frac{\gamma_3^{i_2j_2k_2}}{2!} \alpha_{j_2}\alpha_{k_2} + \dots\right) + \lambda^3_{i_1i_2i_3}\left( \dots \right) \dots
\end{align*}

When the matrix $\gamma_2$ is nondegenerate, the quadratic term in $\alpha$ gives exactly the matrix equation $\lambda_2 = (\gamma_2)^{-1}$, as expected. In third order the equation is:\footnote{Note that to get the right numerical factors, it is easier to first embed the exterior algebra into the tensor algebra, antisymmetrically.}
\begin{equation}\label{eq:inverseLegendreTransform}
	\frac{\lambda^2_{ab}}{2!} (\gamma_2^{ai}\gamma_3^{bjk} + \gamma_3^{aij}\gamma_2^{bk}) + \lambda^3_{abc}\gamma_2^{ai}\gamma_2^{bj}\gamma_2^{ck}= 2\gamma_3^{ijk}
\end{equation}
which, using the solution for $\lambda_2$ and the skew-symmetry of $\gamma_3$, gives our desired solution $\gamma_3^{ijk} = \lambda^3_{abc}\gamma_2^{ai}\gamma_2^{bj}\gamma_2^{ck}$.

\subsubsection{Roadmap to the noncommutative version}
Above we explained that the odd Legendre transform relates polyvector fields  
\[ \gamma = \gamma_1 + \gamma_2 + \dots \]
that are solutions of the Maurer-Cartan equation $[\gamma,\gamma] = 0$ to differential forms that are closed under $d-Lie_{\gamma_1}$.

On one of the sides related by this noncommutative Legendre transform, we have an $A_\infty$-algebra $(A,\mu)$; in other words, a noncommutative pointed dg manifold $X_A$ with an integrable vector field $Q_\mu$; the space $C^*_{[d]}(A)$ where pre-CY structures of dimension $d$ live is interpreted as the space of shifted polyvector fields on $X_A$, with the necklace bracket playing the role of Schouten-Nijenhuis bracket.

On the other side, we have the space of noncommutative forms on $X_A$; this is the negative cyclic homology complex $CC^-_*(A)$ with differential $b_\mu + uB$. For a discussion of this relation in our context, see e.g. \cite[Sec.3.3]{kontsevich2021precalabiyau}.

The noncommutative Legendre transform then relates these two sides. Let us sketch a roadmap for what follows. We start with a pre-CY structure $m$, a solution of the Maurer-Cartan equation $[m,m] = 0$. From this, we will proceed in a completely parallel fashion to what we described for the odd Legendre transform. We first define the noncommutative fiberwise derivative, which is a map of complexes
\[ Fm: (CC^-_*(A), b_\mu+ uB) \to (C^*_{[d]}(A), [m,-]_\nec) \]

In analogy with the (super)commutative case, we prove that if $m_{(2)}$ is nondegenerate, the map above is a quasi-isomorphism; we then write the `energy function' $e_m$ corresponding to $m$ and define the Legendre transform of $m$ to be $(Fm)^{-1}(e_m)$, which defines a nondegenerate element in the negative cyclic complex.

In order to compute the inverse Legendre transform, going from a nondegenerate cyclic homology class to a pre-CY structure, we use the same `implicit function theorem' strategy of \cref{sec:inverting} to produce a map
\[ \Phi: (CC^-_d(A))_\mathrm{nondeg} \to C^*_{[d]}(A) \]
which lands inside of the space $\cM_\mathrm{pre-CY}$ of solutions to that equation, that is, pre-CY structures.

Just like the ordinary Legendre transform, the map $\Phi$ is also `one-to-one', but in a more sophisticated sense; note that one of the sides has a natural `linear' notion of equivalence ($b_\mu+uB$-cohomology) but the other one does not. We will later explain in \cref{sec:simplicial} what is the correct notion of equivalence.

\subsection{Tube quivers}
Our constructions of the maps $Pm$ and $\Phi$ are described using a specific type of ribbon quivers, which we now explain. We refer to \cite{kontsevich2021precalabiyau} for the general definition of ribbon quivers and for the formalism that calculates their action on Hochschild chains.

\begin{definition}
	For any integer $\ell \ge 1$, the space $\scrT_{(\ell)}$ of \emph{tube quivers with $\ell$ outgoing edges} is the vector space spanned by ribbon quivers corresponding to genus zero surfaces with two boundary components: one boundary component with a closed input (source in $V_\times$) and another boundary component with $\ell$ open outputs (sinks in $V_\mathrm{open-out}$).
\end{definition}

For each integer $d$, a $d$-orientation on a ribbon graph can be specified by a total ordering of all its edges and vertices, with permutations assigning a $\pm$ weight depending on the parity of $d$. We will fix this ordering to be of the following form: if the $\times$-source is denoted by $s$ (with incident edge $e$) and the open outputs are labeled $o_1,\dots,o_\ell$, going \emph{clockwise starting from the right}, we will always put this orientation in the form
\[ \pm (o_1\ o_2\ \dots\ o_\ell\ \dots\ e\ s) \]

Recall also that each ribbon quiver $\vec\Gamma$ has a homological degree which depends on $d$, given by the formula
\[ \hspace{-1cm} \deg_d(\Gamma) =  \sum_{v \neq s, o_1,\dots,o_n} ((2-d) out(v) + d + in(v) -4) \]

\begin{example}
Here are some examples of ribbon quivers appearing in $\scrT_{(2)}$ and $\scrT_{(3)}$, respectively: 
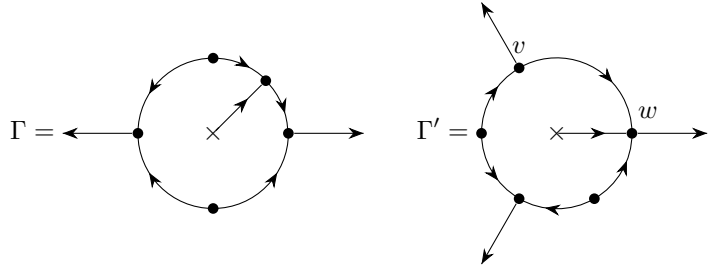
\[ \Gamma = \tikzfig{
	\node [inner sep=0pt] (x) at (0,0) {$\times$};
	\node [bullet] (n) at (0,1) {};
	\node [bullet] (s) at (0,-1) {};
	\node [bullet] (ne) at (0.7,0.7) {};
	\node [bullet] (e) at (1,0) {};
	\node [bullet] (w) at (-1,0) {};
	\draw [->-,shorten <=-3.5pt] (x) to (ne);
	\draw [->-=1] (w) to (-2,0);
	\draw [->-=1] (e) to (2,0);
	\draw [->-] (0,1) arc (90:45:1);
	\draw [->-] (0.7,0.7) arc (45:0:1);
	\draw [->-] (0,1) arc (90:180:1);
	\draw [->-] (0,-1) arc (-90:-180:1);
	\draw [->-] (0,-1) arc (-90:0:1);
}\qquad \Gamma' = 
\tikzfig{
	\node [inner sep=0pt] (x) at (0,0) {$\times$};
	\node [bullet,label={$v$}] (nw) at (-0.5,0.87) {};
	\node [bullet] (se) at (0.5,-0.87) {};
	\node [bullet] (sw) at (-0.5,-0.87) {};
	\node [bullet,label={[xshift=6pt]$w$}] (e) at (1,0) {};
	\node [bullet] (w) at (-1,0) {};
	\draw [->-,shorten <=-3.5pt] (x) to (e);
	\draw [->-=1] (nw) to (-1,1.74);
	\draw [->-=1] (sw) to (-1,-1.74);
	\draw [->-=1] (e) to (2,0);
	\draw [->-] (-0.5,0.87) arc (120:0:1);
	\draw [->-] (-1,0) arc (180:120:1);
	\draw [->-] (-1,0) arc (180:240:1);
	\draw [->-] (0.5,-0.87) arc (-60:-120:1);
	\draw [->-] (0.5,-0.87) arc (-60:0:1);
}\]
When $d=0$, the quivers of degree zero only have the input $\times$, the outputs, and vertices either with two inputs and one output, or with zero inputs and one output. Also, the degree of some quiver is equal to how many edges were contracted to obtain it from any degree zero quiver; for example, for the examples above $\deg_0(\Gamma) = 0$ and $\deg_0(\Gamma') = 2$, since two vertices have degree one (marked $v$ and $w$). For any $d$, these degrees get shifted to $\deg_d(\Gamma) = -2d$ and $\deg_d(\Gamma') = -3d+2$.
\end{example}

Thus picking an integer $d$, we can define the space $\scrT^d_{(\ell)}$ of $d$-oriented tube quivers, which as a vector space is equal to $\scrT_{(\ell)}$ but has a grading depending on $d$.

We consider then inserting a Hochschild chain into the $\times$-vertex and also any numbers of incoming $A[1]$ arrows, in between the $\ell$ outgoing legs. Now, given \emph{any} element $m = m_{(1)} + m_{(2)} + \dots \in C^*_{[d]}(A)$, evaluating this oriented ribbon quiver we then get $\ell$ outgoing $A$ factors. This evaluation gives a linear morphism of graded vector spaces
\[ E: \scrT^d_{(\ell)} \otimes CC^-_*(A) \longrightarrow C^*_{(\ell)}(A) \]
For general $m$, $E$ has no reason to commute with any differentials whatsoever. We now analyze some natural differentials on the space $\scrT^d_{(\ell)}$.

\subsection{Differentials on tube quivers}

\subsubsection{The chain boundary differential}
We first have the differential $\del$ defined by summing over vertex separations. This is the boundary operator on chains, and has homological degree of $-1$, as usual. The evaluation of a ribbon quiver is compatible with this differential, that is,
\[ E: (\scrT^d_{(\ell)},\del) \otimes (C_*(A),b) \longrightarrow (C^*_{(\ell)}(A), [\mu,-]_\nec) \]
is a map of cochain complexes (where we regard everything with \emph{cohomological} grading), and so descends to a map in cohomology.

\begin{proposition}\label{prop:circle}
	When $d=0$, the complex $(\scrT^d_{(\ell)},\del)$ calculates the homology of the circle:
	\[ H_*(\scrT^{d=0}_{(\ell)},\del) = H_*(S^1) \]
	that is, $\kk$ in homological degrees zero and one. The same holds for other values of $d$, but with a shift:
	\[ H_{* + \ell d}(\scrT^d_{(\ell)},\del) = H_*(S^1) \]
\end{proposition}
\begin{proof}
	The first statement is a special case of \cite[Thm.60]{kontsevich2021precalabiyau}. The ribbon quivers which appear in $\scrT_{(\ell)}$ give a cell decomposition of the space
	\[ ``\cM_{0,2}" \times S^1 \times (\RR_{>0})^{2\ell-1} = S^1 \times (\RR_{>0})^{2\ell-1} \]
	here $``\cM_{0,2}"$ denotes just the point, since the genus zero curve with two punctures has no moduli. As explained in the proof of that theorem, $(\scrT^{d=0},\del)$ is the complex of chains on a cell complex that is dual to the cell decomposition above. 
	
	The ribbon quivers with $d$-degree zero correspond to top-dimensional cells. To see that the identification above is correct, we note that for such a quiver we can independently choose $\ell$ positive number to be the lengths of the outgoing legs, and $\ell-1$ positive numbers to be the lengths of distances between the vertices that are not $s,o_1,\dots,o_\ell$; the last length is fixed by the others. To that we add a circle worth of directions where the edge coming out of the $\times$-vertex $s$ can land.
	
	The latter statement is a variation on the $d=0$ case; conceptually, for different $d$, the result is twisted by powers of a line bundle $\cL$ on this moduli space which is trivial up to a shift of $\ell$.
\end{proof}

\subsubsection{The rotation differential}\label{sec:rotationDiff}
We now introduce another differential, corresponding to rotation around the $S^1$-factor. For that, let us first consider the following decomposition of vector spaces
\[ \scrT_{(\ell)} = (\scrT_{(\ell)})^\mathrm{edge} \oplus (\scrT_{(\ell)})^\mathrm{vertex} \]
where we decompose by what is at the end of the edge $e$ (the edge incident to the $\times$-vertex). The subspace $(\scrT_{(\ell)})^\mathrm{edge}$ is spanned by ribbon quivers where $e$ ends on a trivalent vertex with two incoming and one outgoing edge and the subspace $(\scrT_{(\ell)})^\mathrm{vertex}$ is spanned by the other ribbon quivers.

The names come from thinking of each tube ribbon graph without the source vertex; each is a circle with trees attached to the outside. To the interior of the circle we add the source $\tikzfig{\node [inner sep=0pt] (a) at (0,0) {$\times$};
\draw [->-, shorten <=-3.5pt] (a) to (1,0);}$; this arrow can either land on some edge, giving a tube quiver in $(\scrT_{(\ell)})^\mathrm{edge}$, or on an already-existing vertex, giving a tube quiver in $(\scrT_{(\ell)})^\mathrm{vertex}$.

We now define a map $R:(\scrT^d_{(\ell)})^\mathrm{edge} \to (\scrT^d_{(\ell)})^\mathrm{vertex}$ of homological degree $+1$ by the following prescription. Let $\vec\Gamma$ be some tube quiver in $(\scrT_{(\ell)})^\mathrm{edge}$; its edge $e$ lands at a vertex $v$ of either of the two forms:
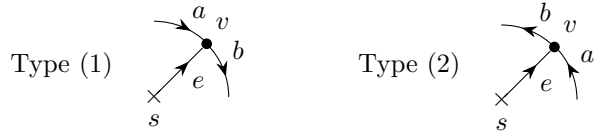
\[
\text{Type } (1) \quad \tikzfig{
	\node [bullet,label={[xshift=6pt]$v$}] (v) at (0,0) {};
	\node [inner sep=0,label=below:{$s$}] (x) at (-0.7,-0.7) {$\times$};
	\draw [->-,shorten <=-3.5pt] (x) to node [auto,swap] {$e$} (v);
	\draw [->-] (-0.7,0.3) arc (90:45:1) node [midway] {$a$};
	\draw [->-] (0,0) arc (45:0:1) node [midway] {$b$};
} \qquad \qquad \text{Type } (2) \quad \tikzfig{
	\node [bullet,label={[xshift=6pt]$v$}] (a) at (0,0) {};
	\node [inner sep=0,label=below:{$s$}] (x) at (-0.7,-0.7) {$\times$};
	\draw [->-,shorten <=-3.5pt] (x) to node [auto,swap] {$e$} (v);
	\draw [->-] (0,0) arc (45:90:1) node [midway,swap] {$b$};
	\draw [->-] (0.3,-0.7) arc (0:45:1) node [midway,swap] {$a$};
}\]

Let us pick a $d$-orientation on $\vec\Gamma$ by some ordering of the form
\[ (b\ v\ \dots\  o_1\ o_2\ \dots\ o_\ell\ \dots\ a\ \dots\ e\ s) \]
that is, putting the letters $bv$ in the beginning.

We now consider all the ribbon quivers $\vec\Gamma_{v'}$ we get by sending the edge $e$ to land in all the other vertices of the circle $v'$ that are not $v$. The result doesn't have the vertex $v$ anymore, and the edge where it previously was is now denoted $a$.
\begin{definition}
	The \emph{rotation differential} $R$ is defined on the subspace $(\scrT^d_{(\ell)})^\mathrm{edge}$ by
	\begin{align*} 
		\hspace{-1cm}  R ( \vec\Gamma,(b\ v\ \dots\ &o_1\ o_2\ \dots\ o_\ell\ \dots\ a\ \dots\ e\ s) )  = \\
		&\sum_{v \neq v' \text{ in circle}} \left( \vec\Gamma_{v'}, (-1)^{\ell + \#} (o_1\ o_2\ \dots\ o_\ell\ \dots\ a\ \dots\ e\ s)  \right)
	\end{align*}
	where in the exponent of the sign, $\# = 0$ if the vertex $v$ is of type (1) above, or $\#=1$ if it is of type (2) above. We extend by zero from the subspace $(\scrT^d_{(\ell)})^\mathrm{edge}$ to a map $R: \scrT^d_{(\ell)} \to \scrT^d_{(\ell)}$, of homological degree +1.
\end{definition}
In other words, under the direct sum decomposition, $R$ is given by a strictly triangular matrix $\begin{pmatrix} 0 & 0 \\ R & 0 \end{pmatrix}$; it is evident that $R^2=0$. We check that the assignment of the orientation above is coherent with respect to the change of ordering, and does define a map from oriented ribbon quivers.

\begin{example}
	Let us compute the differential for the tube quiver with orientation $(\Gamma,\cO)$ given by 
	\[\left( \tikzfig{
		\node [inner sep=0pt,label={$s$}] (x) at (0,0) {$\times$};
		\node at (2,0.3) {$o_1$};
		\node at (-2,0.3) {$o_2$};
		\node [bullet,label=above:{$v_4$}] (n) at (0,1) {};
		\node [bullet,label=below:{$v_2$}] (s) at (0,-1) {};
		\node [bullet,label={[xshift=6pt,yshift=-2pt]$v_5$}] (ne) at (0.7,0.7) {};
		\node [bullet,label=left:{$v_1$}] (e) at (1,0) {};
		\node [bullet,label=right:{$v_3$}] (w) at (-1,0) {};
		\draw [->-,shorten <=-3.5pt] (x) to node [auto] {$e$} (ne);
		\draw [->-=1] (w) to node [auto] {$e_4$}(-2,0);
		\draw [->-=1] (e) to node [auto,swap] {$e_1$} (2,0);
		\draw [->-] (0,1) arc (90:45:1)  node [midway,xshift=-4pt] {$e_6$};
		\draw [->-] (0.7,0.7) arc (45:0:1)  node [midway] {$e_7$};
		\draw [->-] (0,1) arc (90:180:1)  node [midway,swap] {$e_5$};
		\draw [->-] (0,-1) arc (-90:-180:1)  node [midway] {$e_3$};
		\draw [->-] (0,-1) arc (-90:0:1) node [midway,swap] {$e_2$};
	},\ (o_1\ o_2\ v_1\ v_2\ v_3\ v_4\ v_5\ e_1\ e_2\ e_3\ e_4\ e_5\ e_6\ e_7\ e\ s) \right)\]
To calculate the signs in $R(\Gamma,\cO)$, we first bring the pair $e_7\ v_5$ to the beginning of the string, getting a sign $(-1)^{6\times(d-1) + 6\times d} = +1$ from the permutation (recall that swapping two edges gives $(-1)^{d-1}$ and two vertices, $(-1)^d$). For the orientation of the terms in $R(\Gamma, \cO)$ we just take this sign into consideration and delete the pair $e_7\ v_5$, getting:
\begin{align*}
	R(\Gamma, \cO) &= \tikzfig{
	\node [inner sep=0pt,label={$s$}] (x) at (0,0) {$\times$};
	\node at (2,0.3) {$o_1$};
	\node at (-2,0.3) {$o_2$};
	\node [bullet,label=above:{$v_4$}] (n) at (0,1) {};
	\node [bullet,label=below:{$v_2$}] (s) at (0,-1) {};
	\node [bullet,label={[xshift=+5pt]$v_1$}] (e) at (1,0) {};
	\node [bullet,label=right:{$v_3$}] (w) at (-1,0) {};
	\draw [->-,shorten <=-3.5pt] (x) to node [auto,swap] {$e$} (e);
	\draw [->-=1] (w) to node [auto] {$e_4$}(-2,0);
	\draw [->-=1] (e) to node [auto,swap] {$e_1$} (2,0);
	\draw [->-] (0,1) arc (90:0:1)  node [midway,xshift=-4pt] {$e_6$};
	\draw [->-] (0,1) arc (90:180:1)  node [midway,swap] {$e_5$};
	\draw [->-] (0,-1) arc (-90:-180:1)  node [midway] {$e_3$};
	\draw [->-] (0,-1) arc (-90:0:1) node [midway,swap] {$e_2$};
} \quad + \quad \tikzfig{
\node [inner sep=0pt,label=below:{$s$}] (x) at (0,0) {$\times$};
\node at (2,0.3) {$o_1$};
\node at (-2,0.3) {$o_2$};
\node [bullet,label=above:{$v_4$}] (n) at (0,1) {};
\node [bullet,label=below:{$v_2$}] (s) at (0,-1) {};
\node [bullet,label=left:{$v_1$}] (e) at (1,0) {};
\node [bullet,label=right:{$v_3$}] (w) at (-1,0) {};
\draw [->-,shorten <=-3.5pt] (x) to node [auto] {$e$} (n);
\draw [->-=1] (w) to node [auto] {$e_4$}(-2,0);
\draw [->-=1] (e) to node [auto,swap] {$e_1$} (2,0);
\draw [->-] (0,1) arc (90:0:1)  node [midway,xshift=-4pt] {$e_6$};
\draw [->-] (0,1) arc (90:180:1)  node [midway,swap] {$e_5$};
\draw [->-] (0,-1) arc (-90:-180:1)  node [midway] {$e_3$};
\draw [->-] (0,-1) arc (-90:0:1) node [midway,swap] {$e_2$};
} \\
&+ \quad \tikzfig{
	\node [inner sep=0pt,label=below:{$s$}] (x) at (0,0) {$\times$};
	\node at (2,0.3) {$o_1$};
	\node at (-2,0.3) {$o_2$};
	\node [bullet,label=above:{$v_4$}] (n) at (0,1) {};
	\node [bullet,label=below:{$v_2$}] (s) at (0,-1) {};
	\node [bullet,label=left:{$v_1$}] (e) at (1,0) {};
	\node [bullet,label={[xshift=-5pt]$v_3$}] (w) at (-1,0) {};
	\draw [->-,shorten <=-3.5pt] (x) to node [auto,swap] {$e$} (w);
	\draw [->-=1] (w) to node [auto] {$e_4$}(-2,0);
	\draw [->-=1] (e) to node [auto,swap] {$e_1$} (2,0);
	\draw [->-] (0,1) arc (90:0:1)  node [midway,xshift=-4pt] {$e_6$};
	\draw [->-] (0,1) arc (90:180:1)  node [midway,swap] {$e_5$};
	\draw [->-] (0,-1) arc (-90:-180:1)  node [midway] {$e_3$};
	\draw [->-] (0,-1) arc (-90:0:1) node [midway,swap] {$e_2$};
} \quad + \quad
\tikzfig{
	\node [inner sep=0pt,label=above:{$s$}] (x) at (0,0) {$\times$};
	\node at (2,0.3) {$o_1$};
	\node at (-2,0.3) {$o_2$};
	\node [bullet,label=above:{$v_4$}] (n) at (0,1) {};
	\node [bullet,label=below:{$v_2$}] (s) at (0,-1) {};
	\node [bullet,label=left:{$v_1$}] (e) at (1,0) {};
	\node [bullet,label=right:{$v_3$}] (w) at (-1,0) {};
	\draw [->-,shorten <=-3.5pt] (x) to node [auto,swap] {$e$} (s);
	\draw [->-=1] (w) to node [auto] {$e_4$}(-2,0);
	\draw [->-=1] (e) to node [auto,swap] {$e_1$} (2,0);
	\draw [->-] (0,1) arc (90:0:1)  node [midway,xshift=-4pt] {$e_6$};
	\draw [->-] (0,1) arc (90:180:1)  node [midway,swap] {$e_5$};
	\draw [->-] (0,-1) arc (-90:-180:1)  node [midway] {$e_3$};
	\draw [->-] (0,-1) arc (-90:0:1) node [midway,swap] {$e_2$};
}\quad ,
\end{align*}
all with orientation $+(o_1\ o_2\ v_1\ v_2\ v_3\ v_4\ e_1\ e_2\ e_3\ e_4\ e_5\ e_6\ e_7\ e\ s)$.
\end{example}
\vspace{1cm}
A direct calculation using the sign conventions for $\del$ and $R$ shows that they graded-commute with each other, that is,
\[ \del R + R \del = 0 \]
Therefore $(\scrT^d_{(\ell)}, \del, R)$ has the structure of a \emph{mixed complex}, or equivalently, a complex with a chain-level action of the circle.

\subsection{Cyclicity and negative cyclic homology}
Let us describe how the tube quivers with differentials $\del$ and $R$ interact with the Hochschild and Connes differential on Hochschild chains. We write just $\Gamma$ for some oriented tube quiver $(\Gamma,\cO)$, and will only specify the orientation when actually necessary.

\subsubsection{Cyclic complex of tube quivers}
Let $u$ be a variable of homological degree $-2$.
\begin{definition}
	For a fixed $\ell, d$, the \emph{cyclic complex of tube quivers} is the $\kk[u]$-module
	\[ \mathscr{CT}^d_{(\ell)} :=  \scrT^d_{(\ell)} \otimes_\kk \kk[u,u^{-1}]/k[u] \]
	together with the differential $\del - uR$.
\end{definition}

That is, an element of homological degree $n$ in $\mathscr{CT}^d_{(\ell)}$ is given by an expression
\[ \vec\Gamma = \vec\Gamma^0 + \vec\Gamma^1 u^{-1} + \vec\Gamma^2 u^{-1} \]
where $\vec\Gamma^i$ is an element of $\scrT^d_{(\ell)}$ of homological degree $n-2i$.

\begin{proposition}\label{prop:homologyPoint}
	Up to a shift, the complex $\mathscr{CT}^d_{(\ell)}$ calculates the homology of the point. That is,
	\[ H_{\ell d}(\mathscr{CT}^d_{(\ell)},\del-uR) = \kk \]
	and $H_i(\mathscr{CT}^d_{(\ell)},\del) = 0$ for $i \neq \ell d$.
\end{proposition}
\begin{proof}
	Let us show the case $d=0$; as in \cref{prop:circle}, the general case follows from that one by twisting with a line bundle that is trivial up to an overall shift. When $d=0$, $\mathscr{CT}^d_{(\ell)}$ is a non-negatively graded mixed complex. So the Connes long exact sequence in low degrees splits as
	\[ 0 \to H_2(\mathscr{CT}^d_{(\ell)},\del-uR) \to H_0(\mathscr{CT}^d_{(\ell)},\del-uR) \to H_1(\scrT^{d=0}_{(\ell)},\del) \to H_1(\mathscr{CT}^d_{(\ell)},\del-uR) \to 0 \]
	and
	\[ 0 \to H_0(\scrT^{d=0}_{(\ell)},\del) \to H_0(\mathscr{CT}^d_{(\ell)},\del-uR) \to 0 \]
	Thus $H_0(\mathscr{CT}^d_{(\ell)},\del-uR) = \kk$, and since $H_1(\scrT^d_{(\ell)},\del) = \kk$ it is enough to calculate that
	\[ H_1(\mathscr{CT}^d_{(\ell)},\del-uR) = 0 \]
	This follows from the fact that the nontrivial class in $H_1(\scrT^d_{(\ell)},\del)$ can be represented by a chain in the image of $R$.
\end{proof}

\subsubsection{Action on negative cyclic homology}
Recall that the negative cyclic homology of $(A,\mu)$ can be computed by the complex
\[ CC^-_*(A) = (C_*(A)[[u]], b + uB) \]
where $b$ is the Hochschild differential of homological degree $-1$ (depends on $\mu$) and $B$ is the Connes differential of homological degree $+1$ (does not depend on $\mu$).

Suppose now that we are given a pre-CY structure $m$. We extend the evaluation map $E$ to a map of $\kk[u]$-modules
\[ E|_{u^{-1}=0}: \mathscr{CT}^d_{(\ell)} \otimes CC^-_*(A) \to C^*_{(\ell)}(A) \]
by taking the part of degree zero in $u$, that is, by adding all the evaluations $\vec\Gamma^i(\lambda_i)$. For ease of notation let us also denote this map by $E$.

\begin{proposition}
	The map $E|_{u^{-1}=0}$ is compatible with the differentials, and descends to a map in co/homology
	\[ H_*(\mathscr{CT}^d_{(\ell)}, \del-uR) \otimes HC^-_*(A) \to H^*_{(\ell)}(A). \]
\end{proposition}
\begin{proof}
	It is enough to prove that for any $\vec\Gamma \in \scrT^d_{(\ell)}$ and $\lambda \in C_*(A)$, we have
	\[ [\mu,E(\vec\Gamma,\lambda)]_\nec = E((\del-uR)\vec\Gamma,\lambda) + (-1)^{\deg(\vec\Gamma)} E(\vec\Gamma,(b+uB)\lambda). \]
	But we already know that
	\[ [\mu,E(\vec\Gamma,\lambda)]_\nec = E(\del\vec\Gamma,\lambda) + (-1)^{\deg(\vec\Gamma)} E(\vec\Gamma,b\lambda), \]
	so it remains to prove that $E(R\vec\Gamma,\lambda) = (-1)^{\deg(\vec\Gamma)}E(\vec\Gamma, B\lambda)$. This follows from the fact that the unit of $A$ satisfies
	\[ \mu^2(1_A,a) = (-1)^{\bar a+1}\mu^2(a,1_A) = a \]
	for all $a$, so inputting the result of $B$ into a ribbon quiver in $(\scrT^d_{(\ell)})^\mathrm{edge}$ has the same result as redistributing that arrow around the other vertices of the cycle. We check that the signs in the differentials give the correct relation.	
\end{proof}

\subsubsection{Rotation invariant tube quivers}
For any $d$, we have the dimension $d$-action of $\ZZ_\ell$ on the higher Hochschild cochains $C^*_{(\ell)}(A)$, as defined in \cite[Sec.3.1]{kontsevich2021precalabiyau}: in this definition, the rotation of an angle $2\pi/\ell$ of the vertex comes with the Koszul sign together with an extra sign $(d-1)(\ell-1)$. The invariants under this action, when properly shifted, define what we called the dimension $d$ higher cyclic cochains:
\[ C_{(\ell, d)}^*(A) := (C_{(\ell)}^*(A))^{(\ZZ_\ell,d)}[(d-2)(\ell-1)] \]

We now define this action analogously on the other side:
\begin{definition}
	The dimension $d$-action of $\ZZ/\ell$ on the complex $\scrT^d_{(\ell)}$ is given by rotating the quiver by an angle of $+2\pi/\ell$, together with cyclically permuting the output vertices in the orientation, sending
	\[ (o_1\ o_2\ \dots\ o_\ell\ \dots\ e\ s) \mapsto (o_\ell\ o_1\ o_2\ \dots\ o_{\ell-1}\ \dots\ e\ s).  \]
\end{definition}
This action extends $k[u]$-linearly to $\cC\cT^d_{(\ell)}$, and we denote
\[ \mathscr{CT}_{(\ell,d)} := (\mathscr{CT}^d_{(\ell)})^{(\ZZ/\ell,d)}[(d-2)(\ell-1)]\]
for its shifted invariants. We also put all of these complexes for $\ell \ge 2$ together, into a complex
\[ \mathscr{CT}_{[d]} := \prod_{\ell \ge 2} \mathscr{CT}_{(\ell,d)} \]

Note that since vertices enter in the orientation with weight $(d-1)$, the sign of this permutation in the orientations is $(-1)^{(d-1)(\ell-1)}$, already accounting for the extra sign of the dimension $d$ action. We also calculate that this action commutes with the differentials $\del$ and $R$, so the map $E|_{u^{-1}=0}$ restricts to $\cC\cT_{(\ell,d)}$, and gives a map of graded vector spaces
\[ \mathscr{CT}_{[d]} \otimes CC^-_*(A) \to C^*_{[d]}(A). \]

Recall that given a pre-CY structure $m$ on $A$, the necklace bracket $[m,-]_\nec$ defines a differential on $C^*_{[d]}(A)$; by definition this differential is a sum
\[ [m,-]_\nec = d_\mu + [m_{(2)},-]_\nec +  [m_{(3)},-]_\nec + \dots \]
where $d_\mu$ is just the differential on (usual) Hochschild cochains. Each term $[m_{(k)},-]_\nec$ maps $C^*_{(\ell,d)}(A) \to C^*_{(\ell + k-1,d)}(A)[1]$.

We mimic this differential on the tube quiver side, by defining differentials
\[ \del_k: \mathscr{CT}_{(\ell,d)} \to \mathscr{CT}_{(\ell+k-1,d)} \]
by taking the necklace bracket of the tube quiver with a vertex of $k$ outgoing legs. After checking all the signs and degrees, we conclude that:
\begin{proposition}\label{prop:Ecompatible}
	The map $E|_{u^{-1}=0}$ is compatible with the differential $(\del-uR + \del_2 + \del_3 + \dots)$ on $\mathscr{CT}_{[d]}$, and gives a map in co/homology
	\[ H_*(\mathscr{CT}_{[d]}, \del - uR + \del_2 + \del_3 + \dots) \otimes HC^-_*(A) \to C^*_{[d]}(A). \]
\end{proposition}

\subsection{Defining the Legendre transform}
\subsubsection{The fiberwise derivative}\label{sec:fiberwise}
From the proposition above, given any closed class in $\mathscr{CT}_{[d]}$, any negative cyclic homology class, and a pre-CY structure $m$, we produce another element $n$ of $C^*_{[d]}(A)$ satisfying the equation $[m,n]_\nec = 0$.

This will play the role of the \emph{fiberwise derivative} in the usual Legendre transform. Recall from \cref{prop:tangentLegendre} that the fiberwise derivative can be understood as the variational derivative of the Legendre transform; thus in the noncommutative case it is natural that it would land in the tangent complex $(C^*_{[d]}(A), [m,-]_\nec)$, which calculates the tangent space of $\cM_\mathrm{pre-CY}$ at the point $m$.

\begin{proposition}\label{prop:extend}
	Any $\del$-closed element of cohomological degree $2d$ in $\mathscr{CT}^d_{(2)}$, invariant under the $(\ZZ_2,d)$ action, extends to a $(\del-uR+\del_2+\del_3+\dots)$-closed element of cohomological degree $d+2$ in $\mathscr{CT}_{[d]}$.
\end{proposition}
\begin{proof}
	We prove this by induction in the number $\ell$ of outgoing legs. Let us say that we have an extension
	\[ \Gamma_{(2)} + \Gamma_{(3)} + \dots + \Gamma_{(\ell)} \]
	where $\Gamma_{(\ell)} \in \mathscr{CT}_{(\ell,d)}$. Unraveling the definitions and degree shifts, it means we have
	\begin{align*}
		\Gamma_{(2)} &= \Gamma_{(2)}^0 \\
		\Gamma_{(3)} &= \Gamma_{(3)}^0 + \Gamma_{(3)}^1 u^{-1} \\
		\dots \\
		\Gamma_{(\ell)} &= \Gamma_{(\ell)}^0 +  \Gamma_{(\ell)}^1 u^{-1} +\dots + \Gamma_{(\ell)}^{\ell-2} u^{-\ell+2}
	\end{align*}
	where $\Gamma_{(k)}^i$ is of homological degree $-dk+2k-4+2i$ in $\scrT^d_{(k)}$; the terms above are the only ones that can be nonzero for degree reasons.
	
	Suppose that we have
	\[ (\del-uR+\del_2+\dots)(\Gamma_{(2)} + \Gamma_{(3)} + \dots + \Gamma_{(\ell)}) = 0 \]
	up to terms with $\ell$ outgoing legs. The next equation to be solved, with exactly $\ell+1$ outgoing legs, is to find
	\[\Gamma_{(\ell+1)} = \Gamma_{(\ell+1)}^0 + \Gamma_{(\ell+1)}^1 u^{-1} +\dots + \Gamma_{(\ell+1)}^{\ell-1} u^{-\ell+1}\]
	solving
	\[ (\del - uR)\Gamma_{(\ell+1)} = \del_2\Gamma_{(\ell)} + \dots + \del_{\ell}\Gamma_{(2)} \]
	From the fact that $(\del-uR)^2$, and using the induction hypothesis, we find that $(\del-uR)$ applied to the right-hand side gives zero. But since this equation is in homological degree $-d(\ell+1)+2\ell-2 > -d(\ell + 1)$, due to \cref{prop:homologyPoint} it must have a solution for $\Gamma_{(\ell+1)}$; using symmetrization under the $(\ZZ_{\ell+1},d)$ action we get the desired $\Gamma_{(\ell+1)}$.
\end{proof}

Therefore, all we need to do is to specify a $\del$-closed element $\Gamma_{(2)}^0$ of $\scrT^d_{(2)}$ which is $\ZZ/2$-invariant if $d$ is odd and anti-invariant if $d$ is even.

\begin{definition}
	We define the following element of homological degree $-2d$ in $\scrT^d_{(2)}$.
	\[ \Gamma_{(2)} = \frac{1}{2} \left( \tikzfig{
		\node [inner sep=0pt,label=below:{$s$}] (x) at (0,0) {$\times$};
		\node at (2,0.3) {$o_1$};
		\node at (-2,0.3) {$o_2$};
		\node [bullet,label=above:{$v_4$}] (n) at (0,1) {};
		\node [bullet,label=below:{$v_2$}] (s) at (0,-1) {};
		\node [bullet,label={[xshift=6pt,yshift=-2pt]$v$}] (ne) at (0.7,0.7) {};
		\node [bullet,label=left:{$v_1$}] (e) at (1,0) {};
		\node [bullet,label=right:{$v_3$}] (w) at (-1,0) {};
		\draw [->-,shorten <=-3.5pt] (x) to node [auto] {$e$} (ne);
		\draw [->-=1] (w) to node [auto] {$e_4$}(-2,0);
		\draw [->-=1] (e) to node [auto,swap] {$e_1$} (2,0);
		\draw [->-] (0,1) arc (90:45:1)  node [midway,xshift=-4pt] {$e_6$};
		\draw [->-] (0.7,0.7) arc (45:0:1)  node [midway] {$b$};
		\draw [->-] (0,1) arc (90:180:1)  node [midway,swap] {$e_5$};
		\draw [->-] (0,-1) arc (-90:-180:1)  node [midway] {$e_3$};
		\draw [->-] (0,-1) arc (-90:0:1) node [midway,swap] {$e_2$};
	} \quad + \quad \tikzfig{
	\node [inner sep=0pt,label={$s$}] (x) at (0,0) {$\times$};
	\node at (2,0.3) {$o_1$};
	\node at (-2,0.3) {$o_2$};
	\node [bullet,label=above:{$v_4$}] (n) at (0,1) {};
	\node [bullet,label=below:{$v_2$}] (s) at (0,-1) {};
	\node [bullet,label={[xshift=-5pt,yshift=-15pt]$v$}] (sw) at (-0.7,-0.7) {};
	\node [bullet,label=left:{$v_1$}] (e) at (1,0) {};
	\node [bullet,label=right:{$v_3$}] (w) at (-1,0) {};
	\draw [->-,shorten <=-3.5pt] (x) to node [auto] {$e$} (sw);
	\draw [->-=1] (w) to node [auto] {$e_4$}(-2,0);
	\draw [->-=1] (e) to node [auto,swap] {$e_1$} (2,0);
	\draw [->-] (0,1) arc (90:0:1)  node [midway] {$e_6$};
	\draw [->-] (-0.7,-0.7) arc (-135:-180:1)  node [midway,yshift=3pt] {$b$};
	\draw [->-] (0,1) arc (90:180:1)  node [midway,swap] {$e_5$};
	\draw [->-] (0,-1) arc (-90:-135:1)  node [midway,xshift=4pt] {$e_3$};
	\draw [->-] (0,-1) arc (-90:0:1) node [midway,swap] {$e_2$};
	}\right) \]
	both quivers with orientation $\ (o_1\ o_2\ v_1\ v_2\ v_3\ v_4\ v\ e_1\ e_2\ e_3\ e_4\ e_5\ e_6\ b\ e\ s)$.
\end{definition}

Note that $\Gamma_{(2)}$ is of top cohomological degree $+2d$, so $\del\Gamma_{(2)} = 0$. To know that it defines an element of $\mathscr{CT}_{[d]}$ (of cohomological degree $+2d-(d-2)(2-1) = d+2$), we must check that the generator of $\ZZ/2$ acts on it with a sign $(-1)^{(d-1)(2-1)} = (-1)^{d-1}$; to see that we calculate the sign of the permutation of the sequence $(v_1\ v_2\ v_3\ v_4\ v\ e_1\ e_2\ e_3\ e_4\ e_5\ e_6\ b\ e\ s)$ induced by the 180-degree rotation; note that the labels $o_1$ and $o_2$ because we still want to read the outputs of these quivers starting from the right.

\begin{definition}\label{def:gammaQuiver}
	We define $\Gamma = \Gamma_{(2)} + \Gamma_{(3)} + \dots$ to be the element of cohomological degree $d+2$ of $\mathscr{CT}_{[d]}$ given by some fixed extension of $\Gamma_{(2)}$.
\end{definition}
Note that $\Gamma$ is only defined up to some $(\del -uR + \del_2 + \del_3 + \dots)$-exact term.

\begin{lemma}
	For any $\ell > 2$, the class of $\Gamma^{\ell-2}_{(\ell)}$ in $H_{-\ell d}(\scrT^d_{\ell}, \del)$ is nonzero.
\end{lemma}
\begin{proof}
	Note that for degree reasons, the term $\Gamma^{\ell-2}_{(\ell)}$ is the lowest homological degree term that appears in $\Gamma_{(\ell)}$. For $\ell = 2$ it is the only one; and the claim follows from the fact that we can use the $\del$-differential to move the edge $e$ around the circle, giving the following expression for $\Gamma_{(2)}$ up to a $\del\text{-exact term}$:
	\[ \left( \tikzfig{
		\node [inner sep=0pt,label=below:{$s$}] (x) at (0,0) {$\times$};
		\node at (2,0.3) {$o_1$};
		\node at (-2,0.3) {$o_2$};
		\node [bullet,label=above:{$v_4$}] (n) at (0,1) {};
		\node [bullet,label=below:{$v_2$}] (s) at (0,-1) {};
		\node [bullet,label={[xshift=6pt,yshift=-2pt]$v$}] (ne) at (0.7,0.7) {};
		\node [bullet,label=left:{$v_1$}] (e) at (1,0) {};
		\node [bullet,label=right:{$v_3$}] (w) at (-1,0) {};
		\draw [->-,shorten <=-3.5pt] (x) to node [auto] {$e$} (ne);
		\draw [->-=1] (w) to node [auto] {$e_4$}(-2,0);
		\draw [->-=1] (e) to node [auto,swap] {$e_1$} (2,0);
		\draw [->-] (0,1) arc (90:45:1)  node [midway,xshift=-4pt] {$e_6$};
		\draw [->-] (0.7,0.7) arc (45:0:1)  node [midway] {$b$};
		\draw [->-] (0,1) arc (90:180:1)  node [midway,swap] {$e_5$};
		\draw [->-] (0,-1) arc (-90:-180:1)  node [midway] {$e_3$};
		\draw [->-] (0,-1) arc (-90:0:1) node [midway,swap] {$e_2$};
	}, (o_1\ o_2\ v_1\ v_2\ v_3\ v_4\ v\ e_1\ e_2\ e_3\ e_4\ e_5\ e_6\ b\ e\ s) \right)  \]
	In other words, $\Gamma_{(2)}$ is the class of $\pm$ one point in the moduli space.
	
	For any $\ell \ge 3$, from expanding out the equation $(\del-uR +\del_2+\dots)\Gamma$, we see that $\Gamma^{\ell-2}_{(\ell)}$ is defined by solving the equation
	\[ \del \Gamma^{\ell-3}_{(\ell)} - R \Gamma^{\ell-2}_{(\ell)} = - \del_2 \Gamma^{\ell-3}_{(\ell-1)} \]
	so to prove the statement it is sufficient to show the class
	\[ [\del_2 \Gamma^{\ell-3}_{(\ell-1)}] \in H_{-\ell d+1}(\scrT^d_{(\ell)},\del) \]
	is nonzero. The easiest way to show this is to explicitly write down a $\ZZ/2$-valued cocycle on $\scrT^d_{(\ell)}$ that evaluates to 1 when paired with $\del_2 \Gamma^{\ell-3}_{(\ell-1)}$.
	
	The correct (cohomological) degree for this cocycle $\alpha$ is $-\ell d+1$; and we need to specify how it acts on any tube quiver of the same (homological) degree; these are exactly the quivers of degree one higher than the minimum, that is, with exactly one contracted edge.
	
	For any tube quiver $X$, take the closest vertex in the circle to the first output $o_1$. If this vertex is directly connected to the source vertex $s$, we set $\alpha(X)=1$, otherwise, $\alpha(X) = 0$. By thinking of all the possible kinds of edge that can be contracted from a minimum homological degree quiver, we calculate that $\alpha$ is a cocycle. Evaluating $\alpha$ on $\del_2 \Gamma^{\ell-3}_{(\ell-1)}$ gives one (there is exactly one tube quiver there where the source and the first output `are aligned'). \footnote{In fact, if we wanted to we could have worked with a $\ZZ$-valued cocycle instead, and by being careful with the orientations, prove that not only $[\Gamma^{\ell-2}_{(\ell)}] \neq 0$, but also that it is a generator of $H_{-\ell d}(\scrT^d_{(\ell)},\del) = \ZZ$.}
\end{proof}

\begin{definition}
	The \emph{fiberwise derivative} at the pre-CY structure $m$ is the map
	\[ Fm = E(\Gamma,-)|_{u^{-1}=0}: CC^-_*(A) \to C^*_{[d]}(A)[d+2] \]
\end{definition}
By the closedness of $\Gamma$ under the differential $\del-uR + \sum_i\del_i$ we get that the fiberwise derivative is a map of complexes. For simplicity from now on we denote $\Gamma(\lambda)$ for the evaluation $E(\Gamma,\lambda)|_{u^{-1}=0}$.

\begin{proposition}\label{prop:quasiIso}
	If $m_{(2)}$ is nondegenerate, that is, maps to a quasi-isomorphism of bimodules under $C^*_{(2)}(A) \cong \Hom_{A-A}(A_\Delta,A^!)$, then $Fm$ is a quasi-isomorphism.
\end{proposition}
This proposition should be seen a noncommutative version of the implicit function theorem; the element $m_{(2)}$ is the analog of the Jacobian matrix.
\begin{proof}
	As in the proof of the theorem above, for degree reasons, the term of $\Gamma$ with $\ell$ outgoing legs is of the form
	\[ \Gamma_{(\ell)} = \Gamma_{(\ell)}^0 +  \Gamma_{(\ell)}^1 u^{-1} +\dots + \Gamma_{(\ell)}^{\ell-2} u^{-\ell+2} \]
	so for some negative cyclic class $\lambda= \lambda_0 + \lambda_1 u + \lambda_2 u^2 + \dots$ we have
	\[\Gamma_{(\ell)}(\lambda) = \sum_{i=0}^{i=\ell-2} \Gamma_{(\ell)}^i(\lambda_i) \]
	that is, as we increase $\ell$, each new term of $Fm(\lambda)$ depends only on one new term $\lambda_{\ell-2}$. It suffices then to show that the map
	\[ \Gamma_{(\ell)}^{\ell-2}(-): (C_*(A),b_\mu) \to (C^*_{(\ell)}(A),[\mu,-]_\nec) \]
	is a quasi-isomorphism. But recall from \cref{sec:chainLevelNondeg} when $A$ is smooth and $m_{(2)}$ is nondegenerate (that is, gives a quasi-isomorphism $A^! \simeq A_\Delta$), we have the following quasi-isomorphisms
	\begin{align*}
		(C_*(A),b_\mu) &\simeq \Hom_{A-A}(A^!, A_\Delta) \simeq \Hom_{A-A}(A_\Delta, A_\Delta)\\
		(C^*_{(\ell)}(A),[\mu,-]_\nec)  & \simeq \Hom_{A-A}((A^!)^{\otimes_A (k-1)}, A_\Delta) \simeq \Hom_{A-A}(A_\Delta, A_\Delta)
	\end{align*}
	That is, up to quasi-isomorphisms all these invariants are just Hochschild cohomology. We see that all tube quivers of degree $-\ell d$ give cohomologous maps $(C_*(A),b_\mu) \to (C^*_{(\ell)}(A),[\mu,-]_\nec)$: they all involve applying the quasi-isomorphism coming from $m_{(2)}$ $\ell$ times in a sequence, which is the same operation given by the composition of the quasi-isomorphisms above. So the map $\Gamma_{(\ell)}^{\ell-2}(-)$ is cohomologous to (a scalar multiple) of this composition.
\end{proof}

\subsubsection{The energy function and the Legendre transform}
In our discussion of the odd Legendre transform between polyvector fields and forms, we calculated that the correct analog of sending a Lagrangian function $L$ to its energy function $e_L = v_i \del L/\del v_i - L$ was given by sending a polyvector field $\gamma = \gamma_2 + \gamma_3 + \dots$ to the polyvector field
\[ e_\gamma = \gamma_2 + 2 \gamma_3 + 3 \gamma_4 + \dots \]

We make the same definition in the noncommutative case:
\begin{definition}
	The energy function associated to an element $m = \mu + \sum_{\ell \ge 2} m_{(\ell)} \in C^*_{[d]}(A)$ is the element
	\[ e_m =  \sum_{\ell \ge 2} (\ell-1) m_{(\ell)} \in C^*_{[d]}(A). \]
\end{definition}
We calculate that $[m,e_m]_\nec = 0$, that is, this element $e_m$ gives a closed element under the differential on $C^*_{[d]}(A)$.

We then define the Legendre transform:
\begin{definition}\label{def:ncLegendre}
	The noncommutative Legendre transform is the map
	\begin{align*}
		\cL: \left(\cM_{d-\mathrm{preCY}}(A)\right)_\mathrm{nondeg} &\to HC^-_d(A)\\
			m &\mapsto [(Fm)^{-1}(e_m)]
	\end{align*}
	where $\left(\cM_{d-\mathrm{preCY}}(A)\right)_\mathrm{nondeg} \subset C^d_{[d]}(A)$ is the set of $d$-dimensional pre-CY structures $m = \mu + m_{(2)} + m_{(3)}+ \dots \in C^2_{[d]}(A)$ such that $m_{(2)}$ is nondegenerate.
\end{definition}
Note that the map $\cL$ is \emph{not} linear, and strictly speaking, as a map of sets, depends on our choice of quasi-inverse $(Fm)^{-1}$. Nevertheless, in some sense, also like the ordinary Legendre transform on convex functions, it is `one-to-one'; in the next section we explain what this means.

\section{The nondegenerate locus}
We now focus on the case where $A$ is smooth, and continue the study of the nondegenerate locus $\left(\cM_{d-\mathrm{preCY}}(A)\right)_\mathrm{nondeg}$ of the space of pre-CY structures on $A$.

The main result of this section is that the noncommutative Legendre transform $\cL$ we defined above is invertible, and gives an equivalence between nondegenerate pre-CY structures and smooth CY structures. However, this equivalence does not hold in a strict sense; its correct form is as a weak homotopy equivalence of simplicial sets, see later in \cref{sec:simplicial}. Under the assumption that the Hochschild cohomology of $A$ is concentrated in non-negative degree, this equivalence can be more simply phrased in terms of a groupoid of solutions to the Maurer-Cartan equation, which we describe in \cref{sec:groupoid}.

\subsection{Inverting the noncommutative Legendre transform}
We now characterize the image of the noncommutative Legendre transform $\cL$.
\begin{proposition}\label{prop:preToSmooth}
	Every element of $CC^-_d(A)$ in the image of $\cL$ is a smooth CY structure on $A$.
\end{proposition}
\begin{proof}
	Let $\lambda = \lambda_0 + \lambda_1 u^1+\dots = \cL(m)$ be the image under the Legendre transform of a nondegenerate pre-CY structure $m$. The relation between $\lambda_0 \in C_d(A)$ and $m_{(2)} \in C^d_{(2)}(A)$ is given by
	\[ m_{(2)} = \frac{1}{2} \left( \tikzfig{
		\node [vertex] (x) at (0,0) {$\lambda_0$};
		\node at (2,0.3) {};
		\node at (-2,0.3) {};
		\node [bullet] (n) at (0,1) {};
		\node [bullet] (s) at (0,-1) {};
		\node [bullet] (ne) at (0.7,0.7) {};
		\node [bullet] (e) at (1,0) {};
		\node [bullet] (w) at (-1,0) {};
		\draw [->-] (x) to (ne);
		\draw [->-=1] (w) to (-2,0);
		\draw [->-=1] (e) to (2,0);
		\draw [->-] (0,1) arc (90:45:1);
		\draw [->-] (0.7,0.7) arc (45:0:1);
		\draw [->-] (0,1) arc (90:180:1);
		\draw [->-] (0,-1) arc (-90:-180:1);
		\draw [->-] (0,-1) arc (-90:0:1);
	} +\tikzfig{
		\node [vertex] (x) at (0,0) {$\lambda_0$};
		\node at (2,0.3) {};
		\node at (-2,0.3) {};
		\node [bullet] (n) at (0,1) {};
		\node [bullet] (s) at (0,-1) {};
		\node [bullet] (sw) at (-0.7,-0.7) {};
		\node [bullet] (e) at (1,0) {};
		\node [bullet] (w) at (-1,0) {};
		\draw [->-] (x) to (sw);
		\draw [->-=1] (w) to (-2,0);
		\draw [->-=1] (e) to (2,0);
		\draw [->-] (0,1) arc (90:0:1);
		\draw [->-] (-0.7,-0.7) arc (-135:-180:1);
		\draw [->-] (0,1) arc (90:180:1);
		\draw [->-] (0,-1) arc (-90:-135:1);
		\draw [->-] (0,-1) arc (-90:0:1);
	}\right)\]
	up to a $[\mu,-]_\nec\text{-exact term}$, where as usual we evaluate by inserting $m_{(1)} = \mu$ and $m_{(2)}$ into the $\bullet$ vertices, accordingly, with some sign that we omit. By using the $\del$-differentials we find that
	\[ m_{(2)} = \tikzfig{
		\node [bullet] (v3) at (0,0.8) {};
		\node [vertex] (v2) at (0,0) {$\lambda_0$};
		\node [bullet] (v4) at (0.8,0) {};
		\node [bullet] (v5) at (0,-0.8) {};
		\node [bullet] (v6) at (0,-1.5) {};
		\node [bullet] (v1) at (-1,-1.5) {};
		\draw [->-] (v2) to (v4);
		\draw [->-] (0,0.8) arc (90:0:0.8);
		\draw [->-] (0,0.8) arc (90:270:0.8);
		\draw [->-] (0.8,0) arc (0:-90:0.8);
		\draw [->-] (v5) to (v6);
		\draw [->-] (v1) to (-2.2,-1.5);
		\draw [->-] (v1) to (v6);
		\draw [->-] (v6) to (+1.1,-1.5);
	} + ([\mu,-]_\nec\text{-exact term}) \]
	
	We now use the canonical coevaluation map $\coev: \kk \to A_\Delta \otimes_{A\mh A} A^1$ on both sides, getting an equality in $C^*(A,A^!)$. Because $m_{(2)}$ is nondegenerate, this implies the equality in $C^*(A)$:
	\[ \tikzfig{
		\node [vertex] (v) at (0,0) {$1$};
		\draw [->-] (v) to (0,-1.5);} \quad = \quad
	  \tikzfig{
		\node [bullet] (v3) at (0,0.8) {};
		\node [vertex] (v2) at (0,0) {$\lambda_0$};
		\node [bullet] (v4) at (0.8,0) {};
		\node [bullet] (v5) at (0,-0.8) {};
		\node (v6) at (0,-1.6) {};
		\draw [->-] (v2) to (v4);
		\draw [->-] (0,0.8) arc (90:0:0.8);
		\draw [->-] (0,0.8) arc (90:270:0.8);
		\draw [->-] (0.8,0) arc (0:-90:0.8);
		\draw [->-] (v5) to (v6);
	}  + ([\mu,-]\text{-exact term})  \]
	which is exactly the chain-level description of the nondegeneracy condition on $\lambda_0$.
\end{proof}

So in other words, a nondegenerate pre-CY structure on a smooth $A_\infty$-category gives a smooth CY structure on it. We would like to invert this map. Recall that in \cref{sec:oddLegendre} we described the (odd) commutative version of this inverse, and showed how to calculate the inverse odd Legendre transform by describing it implicitly by the equation it solves.

We now describe the noncommutative analog of this procedure, using the same tube quivers
\[ \Gamma_{(\ell)} = \Gamma^0_{(\ell)} + \Gamma^1_{(\ell)} + \dots + \Gamma^{\ell-2}_{(\ell)} \]
we defined in \cref{def:gammaQuiver}. Suppose that we have a smooth CY structure of dimension $d$ given by an element
\[ \lambda = \lambda_0 + \lambda_1 u + \lambda_2 u^2 + \dots \in CC^-_*(A) \]
As in the proof of \cref{prop:preToSmooth} above, from our chain-level description of nondegeneracy, we know that we can find $\alpha \in C^d_{(2)}(A), \beta_{(2)} \in C^{d-1}_{(2)}(A)$ such that
\begin{equation} \label{eq:alpha}
\alpha = \frac{1}{2} \left( \tikzfig{
	\node [vertex] (x) at (0,0) {$\lambda_0$};
	\node at (2,0.3) {};
	\node at (-2,0.3) {};
	\node [bullet] (n) at (0,1) {};
	\node [bullet] (s) at (0,-1) {};
	\node [bullet] (ne) at (0.7,0.7) {};
	\node [bullet] (e) at (1,0) {};
	\node [bullet] (w) at (-1,0) {};
	\draw [->-] (x) to (ne);
	\draw [->-=1] (w) to (-2,0);
	\draw [->-=1] (e) to (2,0);
	\draw [->-] (0,1) arc (90:45:1);
	\draw [->-] (0.7,0.7) arc (45:0:1);
	\draw [->-] (0,1) arc (90:180:1);
	\draw [->-] (0,-1) arc (-90:-180:1);
	\draw [->-] (0,-1) arc (-90:0:1);
} +\tikzfig{
	\node [vertex] (x) at (0,0) {$\lambda_0$};
	\node at (2,0.3) {};
	\node at (-2,0.3) {};
	\node [bullet] (n) at (0,1) {};
	\node [bullet] (s) at (0,-1) {};
	\node [bullet] (sw) at (-0.7,-0.7) {};
	\node [bullet] (e) at (1,0) {};
	\node [bullet] (w) at (-1,0) {};
	\draw [->-] (x) to (sw);
	\draw [->-=1] (w) to (-2,0);
	\draw [->-=1] (e) to (2,0);
	\draw [->-] (0,1) arc (90:0:1);
	\draw [->-] (-0.7,-0.7) arc (-135:-180:1);
	\draw [->-] (0,1) arc (90:180:1);
	\draw [->-] (0,-1) arc (-90:-135:1);
	\draw [->-] (0,-1) arc (-90:0:1);
}\right) + [\mu,\beta_{(2)}]_\nec
\end{equation}
where the vertices $\tikzfig{\node [bullet] (n) at (0,0) {}; \draw [->-] (n) to (0.5,0); \draw [->-] (n) to (-0.5,0);}$ get assigned $\alpha$.

From this, we will construct a pre-CY structure $m$ with $m_{(1)} = \mu$ and $m_{(2)} = \alpha$; each higher term $m_{(\ell)}$ for $\ell > 3$ can be calculated iteratively in the previous terms. To illustrate this, let us first discuss the case of $m_{(3)}$. By definition this element must solve 
\[ [\mu,m_{(3)}]_\nec = m_{(2)} \circ m_{(2)} \]
which is an equation in $C^{2d-1}(A)$. We multiply by two to express everything in terms of brackets
\[ 2[\mu,m_{(3)}]_\nec = [m_{(2)},m_{(2)}]_\nec \]
and then substitute the second $m_{(2)}=\alpha$ factor using \cref{eq:alpha}, to get
\begin{align*}
	2[\mu,m_{(3)}]_\nec = &\pm\frac{1}{2}\ \tikzfig{
	\node [vertex] (x) at (0,0) {$\lambda_0$};
	\node at (2,0.3) {};
	\node at (-2,0.3) {};
	\node [bullet] (n) at (0,1) {};
	\node [bullet] (s) at (0,-1) {};
	\node [bullet] (ne) at (0.7,0.7) {};
	\node [bullet] (e) at (1,0) {};
	\node [bullet] (w) at (-1,0) {};
	\node [bullet] (ww) at (-1.6,0) {};
	\draw [->-] (x) to (ne);
	\draw [->-=1] (e) to (2,0);
	\draw [->-] (0,1) arc (90:45:1);
	\draw [->-] (0.7,0.7) arc (45:0:1);
	\draw [->-] (0,1) arc (90:180:1);
	\draw [->-] (0,-1) arc (-90:-180:1);
	\draw [->-] (0,-1) arc (-90:0:1);
	\draw [->-] (w) to (ww);
	\draw [->-] (ww) to (-2.5,1.5);
	\draw [->-] (ww) to (-2.5,-1.5);
	} \quad \pm\ \dots \\
	&\pm\frac{1}{2}\ \tikzfig{
	\node [vertex] (x) at (0,0) {$\lambda_0$};
	\node at (2,0.3) {};
	\node at (-2,0.3) {};
	\node [bullet] (n) at (0,1) {};
	\node [bullet] (s) at (0,-1) {};
	\node [bullet] (ne) at (0.7,0.7) {};
	\node [bullet] (e) at (1,0) {};
	\node [bullet] (w) at (-1,0) {};
	\node [bullet] (ww) at (-1.6,-0.6) {};
	\draw [->-] (x) to (ne);
	\draw [->-=1] (e) to (2,0);
	\draw [->-] (0,1) arc (90:45:1);
	\draw [->-] (0.7,0.7) arc (45:0:1);
	\draw [->-] (0,1) arc (90:180:1);
	\draw [->-] (0,-1) arc (-90:-180:1);
	\draw [->-] (0,-1) arc (-90:0:1);
	\draw [->-] (w) to (-2.5,1.5);
	\draw [->-] (ww) to (w);
	\draw [->-] (ww) to (-2.5,-1.5);
} \pm\ \dots
\end{align*}

In other words, $m_{(3)}$ satisfies the equation
\[ [\mu,m_{3}]_\nec = \frac{1}{2}\left( [m_{(2)},[\mu,\beta]_\nec]_\nec + (\del_2 \Gamma_{(2)})(\lambda) \right) \]
where, as in \cref{prop:Ecompatible}, $\del_2$ is the differential on that increases the number of outgoing legs by one, given by taking the necklace bracket of a tube quiver with the valence 2 vertex. 

Using the graded Jacobi relation, the equation satisfied by $m_{(2)}$ and the equation defining $\Gamma_{(3)}$ we then have that $m_{(3)}$ satisfies
\begin{equation}\label{eq:m3}
[\mu,m_{(3)}]_\nec = \frac{1}{2}\left( [\mu,[m_{(2)},\beta]_\nec]_\nec + [\mu, \Gamma^0_{(3)}(\lambda_0) + \Gamma^1_{(3)}(\lambda_1)]_\nec \right)
\end{equation}

This equation is slightly misleading; it looks like it we could write
\[ m_{(3)} = \frac{1}{2}\left( [m_{(2)},\beta]_\nec + \Gamma^0_{(3)}(\lambda_0) + \Gamma^1_{(3)}(\lambda_1) \right) \]
and compute $m_{(3)}$ directly from $m_{(1)} = \mu, m_{(2)} = \alpha$ and $\lambda$, but this is not true. By counting degrees, we see that the homological degree of $\Gamma^0_{(3)}$ in $C^*_{(\ell)}(A)$ is $-\ell d +2$; so among the tube quivers of that degree, there could be some that have a single vertex with 3 outgoing arrows, such as, for example, the vertex $p$ of the following quiver:
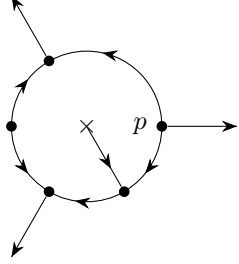
\[\tikzfig{
	\node [inner sep=0pt] (x) at (0,0) {$\times$};
	\node [bullet] (nw) at (-0.5,0.87) {};
	\node [bullet] (se) at (0.5,-0.87) {};
	\node [bullet] (sw) at (-0.5,-0.87) {};
	\node [bullet,label=left:{$p$}] (e) at (1,0) {};
	\node [bullet] (w) at (-1,0) {};
	\draw [->-,shorten <=-3.5pt] (x) to (se);
	\draw [->-=1] (nw) to (-1,1.74);
	\draw [->-=1] (sw) to (-1,-1.74);
	\draw [->-=1] (e) to (2,0);
	\draw [->-] (1,0) arc (0:120:1);
	\draw [->-] (-1,0) arc (180:120:1);
	\draw [->-] (-1,0) arc (180:240:1);
	\draw [->-] (0.5,-0.87) arc (-60:-120:1);
	\draw [->-] (1,0) arc (0:-60:1);
}\]

So in order to evaluate the right-hand side we would need to already know the value of $m_{(3)}$. Nevertheless, we prove that one can solve \cref{eq:m3} up to cohomology; this is also true for each higher term $m_{(\ell)}$. More precisely, we have:
\begin{proposition}\label{prop:inverting}
	For each $\ell \ge 2$, the component of the Maurer-Cartan equation
	\[ 2[\mu,m_{(\ell)}]_\nec = \sum_{i=2}^{\ell-1} [m_{(i)},m_{(\ell-i+1)}]_\nec \]
	has a solution given by
	\[ \hspace{-0.9cm} m_{(\ell)} = \frac{1}{\ell-1}\left( [\mu, \beta_{(\ell)}]_\nec + \dots + [m_{(\ell-1)},\beta_{(2)}]_\nec + \Gamma^0_{(\ell)}(\lambda_0) + \dots + \Gamma^{\ell-2}_{(\ell)}(\lambda_1) \right) \]
	for some element
	\[ \beta_{(2)} + \beta_{(3)} + \dots + \beta_{(\ell-1)} \in C^1_{[d]}(A) \]
	where $\beta_{(i)} \in C^1_{(i,d)}(A)$ depends on all previous $\beta_{(j)}$ and $m_{(j)}$ with $j <i$. 
\end{proposition}

This proposition ultimately follows from a combinatorial fact about tube quivers of a specific degree,which we now explain. Note that $\Gamma^0_{(\ell)}$ is the only `problematic' term; all the other $\Gamma^i_{(\ell)}$ only have vertices with $\ell-1$ or less outgoing legs so by induction we know how to evaluate them. 

Recall that the space $\mathscr{CT}_{(\ell,d)}$ where $\Gamma^0_{(\ell)}$ lives is defined as the cyclically graded-symmetric elements of $\scrT^d_{(\ell)}$ in homological degree $-d\ell+2\ell-4$. We define $\scrT^{d,<\ell}_{(\ell)}$ to be the subcomplex of $\scrT^d_{(\ell)}$ spanned by the tube quivers that \emph{do not} have any vertex with $\ell$ outgoing legs. We now look at the quotient complex $\scrT^d_{(\ell)}/\scrT^{d,<\ell}_{(\ell)}$.

In homological degree $-d\ell+2\ell-4$, every element of degree  $\scrT^d_{(\ell)}/\scrT^{d,<\ell}_{(\ell)}$ can be represented by a linear combination of tube quivers with a single vertex with $\ell$ outgoing legs; all other non-output vertices are `generic', that is, either have two inputs and one output, or are sources with two outputs.
\begin{lemma}\label{lem:homologous}
	Any two elements in $\scrT^d_{(\ell)}/\scrT^{d,<\ell}_{(\ell)}$ are homologous, up to a sign. In other words, given any two tube quivers $\Gamma,\Gamma'$ as above, there is some linear combination of tube quivers $\Gamma''$ such that
	\[ \del\Gamma'' = \Gamma \pm \Gamma' + (\text{term in } \scrT^{d,<\ell}_{(\ell)}). \]
\end{lemma}
\begin{proof}
	Let $Q$ be any such tube quiver as above, that is, with exactly one vertex $p$ with $\ell$ outgoing legs and all other vertices `generic', for instance the tube quiver
	\[\tikzfig{
		\node [inner sep=0pt] (x) at (0,0) {$\times$};
		\node [bullet] (nw) at (-0.5,0.87) {};
		\node [bullet] (se) at (0.5,-0.87) {};
		\node [bullet] (sw) at (-0.5,-0.87) {};
		\node [bullet,label=left:{$p$}] (e) at (1,0) {};
		\node [bullet] (w) at (-1,0) {};
		\draw [->-,shorten <=-3.5pt] (x) to (se);
		\draw [->-=1] (nw) to (-1,1.74);
		\draw [->-=1] (sw) to (-1,-1.74);
		\draw [->-=1] (e) to (2,0);
		\draw [->-=1] (e) to (1.5,1.74);
		\draw [->-=1] (e) to (1.5,-1.74);
		\draw [->-] (1,0) arc (0:120:1);
		\draw [->-] (-1,0) arc (180:120:1);
		\draw [->-] (-1,0) arc (180:240:1);
		\draw [->-] (0.5,-0.87) arc (-60:-120:1);
		\draw [->-] (1,0) arc (0:-60:1);
	}\]
	for $\ell=5$, where $v$ is the only non-generic vertex, of homological degree $d\times 5-d-2\times5+4 = 4d-6$.
	
	We pick any edge $e:v_1\to v_2$ (not connecting to the outputs) of $Q$ and contract it to a vertex $w$, giving some other tube quiver $P$. Calculating the differential gives
	\[ \del P = \pm Q \pm Q' \pm (\text{term in } \scrT^{d,<\ell}_{(\ell)}) \]
	where $Q'$ is obtained from $P$ by expanding $w$ in another direction. We deduce this from checking separately the three possibilities: $v_1 =p, v_2 = p$ or $e$ not incident to $p$. For example
	
	But we can go from any of these tube quivers $Q$ in $\scrT^d_{(\ell)}/\scrT^{d,<\ell}_{(\ell)}$ to any other one by a sequence of edge contractions and expansions. So we can find a sequence of $P$s going from $\Gamma$ to $\Gamma'$ whose sum $\Gamma''$ solves the desired equation.
\end{proof}

\begin{proof}(of \cref{prop:inverting})
	We prove this by induction. The case $\ell = 2$ is just the chain-level description of nondegeneracy; so it follows by the assumption that $\lambda_0$ is nondegenerate. For any fixed $\ell$ we write the component of the Maurer-Cartan equation
	\[ 2[\mu,m_{(\ell)}]_\nec = \sum_{i=2}^{\ell-1} [m_{(i)},m_{(\ell-i+1)}]_\nec \]
	and using the result for $m_{(j)}$ with $j < \ell$ we get that $m_{(\ell)}$ satisfies the equation
	\begin{align*} 
		(\ell -1)[\mu,m_{(\ell)}] &= [\mu,[m_{(2)},\beta_{(\ell-1)}]_\nec + \dots + [m_{(\ell-1)},\beta_{(2)}]_\nec]\\
		&+ [\mu,\Gamma^0_{(\ell)}(\lambda_0) + \dots + \Gamma^{\ell-2}_{(\ell)}(\lambda_1)]
	\end{align*}
	
	But by \cref{lem:homologous}, we can rewrite
	\[ \Gamma^0_{(\ell)} = \Theta + \del \Gamma' + \widetilde\Gamma^0_{(\ell)}, \]
	where $\widetilde\Gamma^0_{(\ell)} \in \scrT^{d,<\ell}_{(\ell)}$, and $\Theta$ is the specific `problematic quiver' of the form
	\[ \Theta = \tikzfig{
		\node [bullet] (v3) at (0,0.8) {};
		\node [vertex] (v2) at (0,0) {$\lambda_0$};
		\node [bullet] (v4) at (0.8,0) {};
		\node [bullet] (v5) at (0,-0.8) {};
		\node [bullet] (v6) at (0,-1.5) {};
		\node [bullet,label={[xshift=2pt]$p$}] (v1) at (-1,-1.5) {};
		\draw [->-] (v2) to (v4);
		\draw [->-] (0,0.8) arc (90:0:0.8);
		\draw [->-] (0,0.8) arc (90:270:0.8);
		\draw [->-] (0.8,0) arc (0:-90:0.8);
		\draw [->-] (v5) to (v6);
		\draw [->-] (v1) to (-2.2,-1.5);
		\draw [->-] (v1) to (-2,-0.8);
		\draw [->-] (v1) to (-1.8,-0.1);
		\draw [->-] (v1) to (-1.6,-1.9);
		\draw [->-] (v1) to (-0.4,-1.9);
		\draw [->-] (v1) to (v6);
		\draw [->-] (v6) to (+1.1,-1.5);
	}\]
	with one vertex $p$ with $\ell$ outgoing arrows (for example, in the drawing above, $\ell=6$). We now rearrange the equation as
	\begin{align*} 
		[\mu,(\ell-1)m_{(\ell)} - \Theta(\lambda_0)] &= [\mu,[m_{(2)},\beta_{(\ell-1)}]_\nec + \dots + [m_{(\ell-1)},\beta_{(2)}]_\nec] \\
		&+[\mu, \widetilde\Gamma^0_{(\ell)}(\lambda_0) + \dots + \Gamma^{\ell-2}_{(\ell)}(\lambda_1)]
	\end{align*}
	where now the right-hand side does have any vertices with $\ell$ or more outgoing legs, so we can evaluate it using the $m_{(j)}$ that we already know.
		
	As for the left-hand side, by the nondegeneracy condition we know that the `bubble' to the right of $\Theta$ evaluates to a cochain that is cohomologous to the unit cochain $1 \in C^0(A)$. Thus, if we replace $\Theta(\lambda_0)$ by $m_{(\ell)}$ in the equation above and solve for $m_{(\ell)}$, we can find an element that solves
	\[ (\ell-1)m_{(\ell)} = [m_{(2)},\beta_{(\ell-1)}]_\nec + \dots + [m_{(\ell-1)},\beta_{(2)}]_\nec + \Gamma^0_{(\ell)}(\lambda_0) + \dots + \Gamma^{\ell-2}_{(\ell)}(\lambda_1), \]
	up to $[\mu,-]_\nec\text{-exact terms}$, so we can just pick $\beta_{(\ell)}$ to be a primitive of these exact terms.
\end{proof}

Repackaging the statement of \cref{prop:inverting}, by picking an appropriate element $\beta = \sum_{i \ge 2} \beta_{(i)} \in C^1_{[d]}(A)$, we get a map to pre-CY structures
\begin{definition}\label{def:phi}
	The map $\Phi: (CC^-_d(A))_\mathrm{nondeg} \to (\cM_{d-\mathrm{preCY}}(A))_\mathrm{nondeg} \subset C^2_{[d]}(A)$ is defined by sending $\lambda \mapsto m = \mu + m_{(2)} + \dots$, where the $m_{(i)}$ are defined using the tube quivers $\Gamma_{(i)}$ and an appropriately chosen element $\beta$.
\end{definition}
By definition, the image $m$ satisfies the equations $m \circ m=0$ and $e_m = \Gamma(m,\lambda) + [m,\beta]_\nec$, with the `energy function' $e_m = \sum_{\ell \ge 2}(\ell-1)m_{(\ell)}$. In other words, $\Phi$ maps smooth CY structures of dimension $d$ to pre-CY structures of dimension $d$ with nondegenerate $m_{(2)}$.

\subsubsection{Example: the circle, continued}\label{sec:exampleContinued}
In order to illustrate the statement of \cref{prop:inverting} in action, we return to the simple example discussed in \cref{sec:theCircle}, that is, the dg category $A$ corresponding to the circle (more precisely, to its realization as the boundary of the 2-simplex).

Recall that the negative cyclic chain representing the smooth CY structure is given by the following element of $C_*(A)[[u]]$:
\[ \lambda = 01[10] + 12[21] - 02[20] + (-01[10|01|10] - 12[21|12|21] + 02[20|02|20])u + \dots\]
The element $\alpha \in C^*_{(2)}(A)$ inverse to the $u^0$ component $\lambda_0$ is given by the formula
\[\tikzfig{
	\node [vertex] (v1) at (0,0) {$\alpha$};
	\draw [-w-] (v1) to (-1,0);
	\draw [->-] (v1) to (1,0);
	\draw [->-] (0,-1) to (v1);
	\node at (0,0.5) {$\underline{k}$};
	\node at (0,-1.2) {$P$};
	\node at (0.6,-0.5) {$\underline{i}$};
	\node at (-0.6,-0.5) {$\underline{j}$};
} = \begin{cases} \frac{1}{2}\left(\delta_{jk} e_k \otimes P - \delta_{ik}P \otimes e_k \right), \text{\ if\ } P \text{\ counter-clockwise} \\
-\frac{1}{2}\left(\delta_{jk} e_k \otimes P - \delta_{ik}P \otimes e_k \right), \text{\ if\ } P \text{\ clockwise} \end{cases}
\]
Together with its $\ZZ/2$-rotation, this specifies all the nontrivial values of $\alpha$.

By the previous results of this section, the element $\alpha$ is the first component $m_{(2)}$ of a pre-CY structure. Each higher component $m_{(k)}$ has cohomological degree $dk-d-2k+4 = 3-k$ since $d=1$. Since $A$ is concentrated in degree zero, we have $C^{3-k}_{(k)}(A) = 0$ for $k \ge 4$, so the terms $m_{(\ge 4)}$ are all zero. Moreover, the only component of $m_{(3)}$ that could be nontrivial is the term $m^{0,0,0}_{(3)}$ with zero inputs.

By \cref{prop:inverting}, we can find linear combinations of tube quivers $\Gamma^0_{(3)}, \Gamma^1_{(3)}$, both with three outputs, such that a solution for the equation $[\mu,m_{(3)}] = [m_{(2)},m_{(2)}]$ is given by
\[ m_{(3)} = \frac{1}{2}(\Gamma^0_{(3)}(\lambda_0) + \Gamma^1_{(3)}(\lambda_1)) \]

Let us start from the second term. We know that the terms in $\Gamma^1_{(3)}$ are of maximum (cohomological) degree among the tube quivers with three outgoing legs; all of those are cohomologous and we can pick any of those representatives; so $\Gamma^1_{(3)}(\lambda_1)$ is given by the diagram
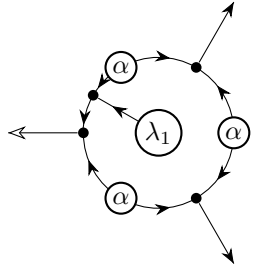
\[\tikzfig{
	\node [vertex] (x) at (0,0) {$\lambda_1$};
	\node [bullet] (se) at (0.5,-0.87) {};
	\node [bullet] (ne) at (0.5,0.87) {};
	\node [bullet] (w) at (-1,0) {};
	\node [bullet] (t) at (-0.87,0.5) {};
	\draw [->-] (x) to (t);
	\draw [->-=1] (ne) to (1,1.74);
	\draw [->-=1] (se) to (1,-1.74);
	\draw [-w-=1] (w) to (-2,0);
	\draw [->-] (1,0) arc (0:60:1);
	\draw [->-] (1,0) arc (0:-60:1);
	\draw [->-] (-0.5,0.87) arc (120:60:1);
	\draw [->-=0.8] (-0.5,0.87) arc (120:150:1);
	\draw [->-=0.8] (-0.87,0.5) arc (150:180:1);
	\draw [->-] (-0.5,-0.87) arc (-120:-60:1);
	\draw [->-] (-0.5,-0.87) arc (-120:-180:1);
	\node [vertex,fill=white] (e) at (1,0) {$\alpha$};
	\node [vertex,fill=white] (nw) at (-0.5,0.87) {$\alpha$};
	\node [vertex,fill=white] (sw) at (-0.5,-0.87) {$\alpha$};
}\]
where we input $\lambda_1 = 01[10|01|10] + 12[21|12|21] - 02[20|02|20]$. Since $\alpha$ is only nonzero when there is exactly one arrow as input, the only nonzero terms we get upon evaluating the diagram are when sending one arrow to each $\alpha$.

We evaluate this diagram separately for each choice of labeling of the three regions around the circle with the objects $0,1,2$; we see that if the three labels are the same, we get cancelling contributions, and that the only nonzero terms happen when exactly two labels are the same: we have
\[\tikzfig{
	\node [vertex] (x) at (0,0) {$\Gamma^{1}_{(3)}(\lambda_1)$};
	\draw [-w-] (x) to (-2,0);
	\draw [->-] (x) to (1,1.74);
	\draw [->-] (x) to (1,-1.74);
	\node at (-0.8,1.34) {$\underline{i}$};
	\node at (-0.8,-1.34) {$\underline{j}$};
	\node at (1.2,0) {$\underline{i}$};
} = \begin{cases} -\frac{1}{8}(ij)\otimes e_i \otimes {ji}, \text{\ if\ } (ij) \text{\ counter-clockwise}	\\ \frac{1}{8}(ij)\otimes e_i \otimes {ji}, \text{\ if\ } (ij) \text{\ clockwise}
\end{cases}
\]

The first term $\Gamma^0_{(3)}(\lambda_0)$ is more difficult to compute, since the combination of tube quivers $\Gamma^0_{(3)}$ has a complicated expression. However, in this case we can calculate it without expressing this entire combination. Recall from \cref{sec:rotationDiff} that we have a decomposition
\[ \scrT_{(\ell)} = (\scrT_{(\ell)})^\mathrm{edge} \oplus (\scrT_{(\ell)})^\mathrm{vertex} \]
between tubes whose central arrow lands onto an edge or a vertex of the circle. For any choice of dimension $d$ (which puts a grading on this space, and defines the signs of the differentials), the differential $\del$ preserves $(\scrT_{(\ell)})^\mathrm{edge}$ and decomposes as $\del = \del^\mathrm{edge} + \del^\mathrm{vertex}$ on $(\scrT_{(\ell)})^\mathrm{vertex}$, where $\del^v$ preserves this space.

In our case, $d = 1$ and the homological degrees of $(\scrT^1_{(3)})$ lie in $-3,-2,-1,0,1,2$ and this decomposition becomes: 
\begin{align*}
	(\scrT^1_{(3)})_{-3} &= (\scrT^1_{(3)})^\mathrm{edge}_{-3} \\
	(\scrT^1_{(3)})_{-2} &= (\scrT^1_{(3)})^\mathrm{edge}_{-2} \otimes (\scrT^1_{(3)})^\mathrm{vertex}_{-2} \\
	\dots \\
	(\scrT^1_{(3)})_{2} &= (\scrT^1_{(3)})^\mathrm{vertex}_{2}
\end{align*}

The element $\Gamma^0_{(2)} \in (\scrT^1_{(3)})_{-1}$ decomposes under the direct sum above as $(\Gamma^0_{(2)})^\mathrm{edge} + (\Gamma^0_{(2)})^\mathrm{vertex}$. We observe that inputting the choice of $\alpha$ above into the 2-valent vertices of any diagram in $(\scrT^1_{(3)})^\mathrm{edge}_{-1}$ gives zero. Thus the value we want is determined by $(\Gamma^0_{(2)})^\mathrm{vertex}$.

The equation satisfied by $\Gamma_{(2)}$ says that $\del\Gamma^0_{(2)} = -R(\Gamma^1_{(2)})$; together with the long exact sequence in cohomology associated to the short exact sequence $(\scrT^1_{(3)})^\mathrm{edge} \to (\scrT^1_{(3)}) \to (\scrT^1_{(3)})^\mathrm{vertex}$, these facts imply that we can pick any solution $X$ to the equation
\[ \del^\mathrm{vertex}( X ) = R(\Gamma^1_{(2)}) \]
and the term $\Gamma^0_{(3)}(\lambda_0)$ will be equal to $X(\lambda_0)$.

We now provide such a solution, given by $X = R(Y)$, where $Y$ is the following linear combination: \footnote{We omit the orientations in the expression for $Y$, but they must be chosen coherently}
\[
	\frac{1}{6} \left(\tikzfig{
	\node [inner sep=0pt] (x) at (0,0) {$\times$};
	\node [bullet] (se) at (0.5,-0.87) {};
	\node [bullet] (ne) at (0.5,0.87) {};
	\node [bullet] (w) at (-1,0) {};
	\node [bullet] (t) at (-0.87,0.5) {};
	\draw [->-,shorten <=-3.5pt] (x) to (t);
	\draw [->-=1] (ne) to (1,1.74);
	\draw [->-=1] (se) to (1,-1.74);
	\draw [-w-=1] (w) to (-2,0);
	\draw [->-] (1,0) arc (0:60:1);
	\draw [->-] (1,0) arc (0:-60:1);
	\draw [->-] (-0.5,0.87) arc (120:60:1);
	\draw [->-=0.8] (-0.5,0.87) arc (120:150:1);
	\draw [->-=0.8] (-0.87,0.5) arc (150:180:1);
	\draw [->-] (-1,0) arc (-180:-60:1);
	\node [bullet] (e) at (1,0) {};
	\node [bullet] (nw) at (-0.5,0.87) {};
	} + \tikzfig{
	\node [inner sep=0pt] (x) at (0,0) {$\times$};
	\node [bullet] (se) at (0.5,-0.87) {};
	\node [bullet] (ne) at (0.5,0.87) {};
	\node [bullet] (w) at (-1,0) {};
	\node [bullet] (t) at (-0.87,0.5) {};
	\draw [->-,shorten <=-3.5pt] (x) to (t);
	\draw [->-=1] (ne) to (1,1.74);
	\draw [->-=1] (se) to (1,-1.74);
	\draw [-w-=1] (w) to (-2,0);
	\draw [->-] (1,0) arc (0:60:1);
	\draw [->-] (1,0) arc (0:-60:1);
	\draw [->-=0.8] (-1,0) arc (180:150:1);
	\draw [->-=0.8] (-0.87,0.5) arc (150:60:1);
	\draw [->-] (-0.5,-0.87) arc (-120:-60:1);
	\draw [->-] (-0.5,-0.87) arc (-120:-180:1);
	\node [bullet] (e) at (1,0) {};
	\node [bullet] (sw) at (-0.5,-0.87) {};
	} + \tikzfig{
	\node [inner sep=0pt] (x) at (0,0) {$\times$};
	\node [bullet] (se) at (0.5,-0.87) {};
	\node [bullet] (ne) at (0.5,0.87) {};
	\node [bullet] (w) at (-1,0) {};
	\node [bullet] (ww) at (-1.5,0) {};
	\node [bullet] (t) at (-0.87,0.5) {};
	\draw [->-,shorten <=-3.5pt] (x) to (t);
	\draw [->-=1] (ne) to (1,1.74);
	\draw [->-=1] (se) to (1,-1.74);
	\draw [->-] (ww) to (w);
	\draw [-w-=1] (ww) to (-2,0);
	\draw [->-] (1,0) arc (0:60:1);
	\draw [->-] (1,0) arc (0:-60:1);
	\draw [->-=0.8] (-1,0) arc (180:150:1);
	\draw [->-=0.8] (-0.87,0.5) arc (150:60:1);
	\draw [->-] (-1,0) arc (-180:-60:1);
	\node [bullet] (e) at (1,0) {};
	\node [bullet] (ww) at (-1.5,0) {};
	} +\dots \right)
\]
where the ellipses indicate we sum over the other two cyclic permutations of these diagrams. Evaluating $X(\lambda_0)$ with the given value of $\alpha$, and labels $i,i,j\neq i$ on the three regions gives us $18$ non-zero terms, all equal; taking in account the $1/6$ factor gives
\[ \begin{cases} -\frac{3}{8}(ij)\otimes e_i \otimes {ji}, \text{\ if\ } (ij) \text{\ counter-clockwise}	\\ \frac{3}{8}(ij)\otimes e_i \otimes {ji}, \text{\ if\ } (ij) \text{\ clockwise} \end{cases} \]
with all other cases zero.

Thus we get that the element $m_{(3)} = \frac{1}{2}(\Gamma^0_{(3)}(\lambda_0) + \Gamma^1_{(3)}(\lambda_1))$ we wanted to calculate is given by
\[\tikzfig{
	\node [vertex] (x) at (0,0) {$m_{(3)}$};
	\draw [-w-] (x) to (-2,0);
	\draw [->-] (x) to (1,1.74);
	\draw [->-] (x) to (1,-1.74);
	\node at (-0.8,1.34) {$\underline{i}$};
	\node at (-0.8,-1.34) {$\underline{j}$};
	\node at (1.2,0) {$\underline{i}$};
} = \begin{cases} -\frac{1}{4}(ij)\otimes e_i \otimes {ji}, \text{\ if\ } (ij) \text{\ counter-clockwise}	\\ \frac{1}{4}(ij)\otimes e_i \otimes {ji}, \text{\ if\ } (ij) \text{\ clockwise}
\end{cases}
\]

One can check that this element satisfies $[\mu, m_{(3)}] + \alpha \circ \alpha = 0$ and $[\alpha,m_{(3)}] = 0$. In other words, we have a fully explicit description of the pre-CY structure on this dg category:
\begin{corollary}
	Taking $m_{2} = \alpha$ and $m_{(3)}$ as above, and $m_{(\ge 4)} = 0$, defines a pre-CY structure of dimension 1 on $A$.
\end{corollary}

\subsection{The simplicial lift}\label{sec:simplicial}
It is clear that the map $\Phi$ we constructed above should be some sort of inverse to the noncommutative Legendre transform $\cL$ of \cref{def:ncLegendre}. But this is not true strictly; for one, the definitions of both $\cL$ and $\Phi$ involve making choices. Also, the two sides related by these maps look different: the locus of nondegenerate elements in negative cyclic homology is a conical locus inside of a linear space, while the set of Maurer-Cartan solutions is the solution set of a quadratic equation.

However, looking at the iterative way in which we constructed these maps, we see that in each step we had to solve an equation relating \emph{linearly} one new component $\lambda_{(k-2)}$ of the negative cyclic chain $\lambda$ to one new component $m_{(k)}$ of the pre-CY structure $m$. Moreover, this relation came from a quasi-isomorphism between these linear spaces of choices; the space $\left(\cM_{d-\mathrm{preCY}}(A)\right)_\mathrm{nondeg}$ should have the structure of an iterated fibration of linear spaces, with the fiber at each step related by a quasi-isomorphism to a graded piece of $CC^*_-(A)$

It turns out that this can be made precise by using the theory of simplicial sets of solutions to Maurer-Cartan equations, as developed in \cite{hinich1996descent,getzler2009lie}, among others. We will show that the map $\Phi$ we constructed in the previous section admits a simplicial lift to a weak equivalence of simplicial sets. 

\subsubsection{The simplicial Maurer-Cartan set}
Let $(\mathfrak{g}^*, \delta)$ be a nilpotent dg Lie algebra over $\kk$. One can look at its naive set of solutions to the Maurer-Cartan equation
\[ MC(\mathfrak{g}^*) = \{ y \in \mathfrak{g}^1\ |\ \delta y + [y,y]/2 = 0\} \]
which in principle has only the structure of a set. Following the exposition in \cite{getzler2009lie}, we recall how to upgrade this to a simplicial set. For any $n \ge 0$, denote by
\[ \Omega^*(\Delta^n) = \kk[t_0,\dots,t_n]/\langle t_0 + \dots + t_n - 1, dt_0 + \dots dt_n \rangle \]
the graded commutative dg algebra of polynomial differential forms on the $n$-simplex. Here, the $t_i$ are in degree zero and the $dt_i$ are in degree one; the differential is $d(t_i) dt_i$ and $d(dt_i) = 0$.

The following proposition/definition is due to Hinich \cite{hinich1996descent}.
\begin{proposition}
	There is a simplicial set $MC_*(\mathfrak{g}^*)$ whose $n$-simplices are given by
	\[ MC_n(\mathfrak{g}^*) = MC(\mathfrak{g}^* \otimes \Omega^*(\Delta^n)), \]
	that is, by the solution set of the Maurer-Cartan equation $(\delta+d)y + [y,y]/2 = 0$ on the dg Lie algebra of $\mathfrak{g}$-forms on the simplex. 
\end{proposition}

\begin{remark}
	In the literature of this topic, the name `Maurer-Cartan' is applied to two different formulations of the equation; in order to avoid confusion let us be clear about their relation. On any \emph{dg Lie algebra} $(\mathfrak{g}^*,d)$, one can look at the equation $dx + [x,x]/2 = 0$, and on any \emph{graded Lie algebra} $\mathfrak{h}^*$ one can look at the equation $[y,y]=0$, as we did earlier in this paper.
	
	Both are often called the Maurer-Cartan equation; the relation is that if $\mathfrak{g}^* \subset \mathfrak{h}^*$ as graded Lie algebras and there is an element $\mu \in \mathfrak{h}^1$ such that $[\mu,\mu]= 0$ and $dx = [\mu,x]$ for all $x \in \mathfrak{g}^*$, then the two equations are equivalent if we look for a solution of the form $y = \mu + x$. In our case, we have
	\[ \mathfrak{g}^* = \prod_{\ell \ge 2} C^*_{(\ell,d)}(A), \qquad \mathfrak{h}^* = C^*_{[d]}(A) = \prod_{\ell \ge 1} C^*_{(\ell,d)}(A) \]
	with $\mu$ given by our fixed $A_\infty$-structure; the solution $x$ is then the sum of all the $m_{(\ell)}$ with $\ell \ge 2$ and $y$ is the full pre-CY structure $m$, also including $m_{(1)} = \mu$.
\end{remark}

The set of zero-simplices is exactly our naive set $MC(\mathfrak{g}^*)$. We would like to apply this to the dg Lie algebra $\mathfrak{g}^* = \prod_{\ell \ge 2} C^*_{(\ell,d)}(A)$. This does not make sense exactly since this algebra is not nilpotent. Nevertheless, note that we can truncate it at any finite $\ell$ and obtain a nilpotent algebra, and that the Maurer-Cartan solutions we want are a limit over solutions on these truncated algebras.

Let us be more precise. Suppose that we have a graded Lie algebra $\mathfrak{a}^*$, endowed with a descending filtration
\[ \mathfrak{a}^* = F^0 \mathfrak{a}^* \supset F^1 \mathfrak{a}^* \supset F^2 \mathfrak{a}^* \supset \dots \]
with the property that for $x \in F^i \mathfrak{a}^*, y \in F^j \mathfrak{a}^*$, we have $[x,y] \in F^{i+j} \mathfrak{a}^*$. We then consider the completions under the filtration $F$
\[ \mathfrak{g}^* = \lim_{\stackrel{\longleftarrow}{i \ge 1}} F^1 \mathfrak{a}^*/F^i \mathfrak{a}^*, \qquad \mathfrak{h}^* := \widehat{\mathfrak{a}^*} = \lim_{\stackrel{\longleftarrow}{i \ge 0}} \mathfrak{a}^*/F^i \mathfrak{a}^* \]
Note that each truncated piece $F^1 \mathfrak{a}^*/F^i \mathfrak{a}^*$ is a nilpotent graded Lie algebra, and we have a natural injection $\mathfrak{g}^* \subset \mathfrak{h}^*$.

If we have an element $\mu \in \mathfrak{h}^1$ such that $[\mu,\mu] = 0$ and $[\mu,-]$ preserves the filtration, this defines a differential on each $F^1 \mathfrak{a}^*/F^i \mathfrak{a}^*$, and each map $F^1 \mathfrak{a}^*/F^{i+1} \mathfrak{a}^* \to F^1 \mathfrak{a}^*/F^i \mathfrak{a}^*$ is a surjection of nilpotent dg algebras. By \cite{hinich1996descent}, this induces a Kan fibration
\[ MC_*(F^1 \mathfrak{a}^*/F^{i+1} \mathfrak{a}^*) \to MC_*(F^1 \mathfrak{a}^*/F^i \mathfrak{a}^*) \]
between the Maurer-Cartan simplicial sets.
\begin{definition}
	The Maurer-Cartan simplicial set of the dg Lie algebra $\mathfrak{g}^*$ is the limit of simplicial sets
	\[ MC_*(\mathfrak{g}^*) := \lim_{\stackrel{\longleftarrow}{i \ge 0}} MC_*(F^1 \mathfrak{a}^*/F^{i+1} \mathfrak{a}^*). \]
\end{definition}

The case we are interested is when 
\[ \mathfrak{a}^* = \bigoplus_{\ell \ge 1} C^*_{(\ell,d)}(A)[1]  \]
endowed with the descending filtration
\[ F^i \mathfrak{a}^* = \bigoplus_{\ell \ge i+1} C^*_{(\ell,d)}(A)[1]. \]
Then we have that the graded Lie algebra $\mathfrak{h}^*$ is exactly $C^*_{[d]}(A)[1]$, which is the dg Lie algebra where pre-CY structures of dimension $d$ live. The condition on the element $\mu$ is exactly the condition for an $A_\infty$-structure on $A$; taking $\mathfrak{g}^*$ to be the dg Lie algebra where the rest of the pre-CY structure lives, with differential $[\mu,-]$, we get an identification between the set of zero-simplices $MC_0(\mathfrak{g}^*)$ and the naive set of $d$-pre-CY structures $\cM_{d-\mathrm{preCY}}(A)$ of the previous section.

\subsubsection{The Deligne groupoid}
We now recall another type of structure on the solutions to the Maurer-Cartan equation, the `Deligne groupoid'. Let us describe it in the graded Lie algebra picture. If $\mathfrak{n}^*$ is a nilpotent graded Lie algebra, there is an exponential action of its degree zero part
\[ e^{\ad(-)}: \mathfrak{n}^0 \times \mathfrak{n}^* \to \mathfrak{n}^* \]
given by
\[ e^{\ad(x)}(y) = y + [x,y] + \frac{1}{2!}[x,[x,y]] + \frac{1}{3!}[x,[x,[x,y]]] + \dots \]
This exponential action extends to an action of the group-like elements of the (completed) universal enveloping algebra $\widehat U(\mathfrak{n}^0)$. 

Note now that $e^{\ad(x)}$ preserves the solution set of the (graded) Maurer-Cartan equation
\[ MC(\mathfrak{n}^*) = \{ y \in \mathfrak{n}^1\ |\ [y,y] = 0\} \]
since the adjoint action preserves the equation; therefore we can regard $MC(\mathfrak{n}^*)//\mathfrak{n}^0$ as a groupoid.

Again we must be a little careful because we want to apply this formalism to the graded Lie algebra $C^*_{[d]}(A)[1]$, which is not nilpotent. Considering again the case of the dg algebra
\[ \mathfrak{g}^* = \lim_{\stackrel{\longleftarrow}{i \ge 1}} F^1 \mathfrak{a}^*/F^i \mathfrak{a}^* \]
with differential $[\mu,-]$ sitting inside of the graded algebra
\[ \mathfrak{h}^* = \lim_{\stackrel{\longleftarrow}{i \ge 1}} \mathfrak{a}^*/F^i \mathfrak{a}^* \]
the we see that the exponential action of $\mathfrak{g}^0$ is well-defined on each truncated piece $F \mathfrak{a}^*/F^i \mathfrak{a}^*$ and therefore can be lifted to an action on the graded Lie algebra $\mathfrak{h}^*$, preserving the Maurer-Cartan equation. Therefore we also have a groupoid $MC(\mathfrak{h}^*)//\mathfrak{g}^0$. 

We choose this notation (with both $\mathfrak{h}^*$ and $\mathfrak{g}^*$) to remind us that the action of $\mathfrak{g}^*$ also involves the element $\mu \in \mathfrak{h^*}$, producing higher order terms $[x,\mu], \frac{1}{2}[x,[x,\mu]]$, etc., but does not change $\mu$ itself. So we can look for solutions of the Maurer-Cartan equation on $\mathfrak{h}^*$ of the form $\mu + x$ where $x \in F^2 \mathfrak{h}^*$; this is a subgroupoid $MC(\mathfrak{h}^*, \mu)//\mathfrak{g}^0$. 

From comparing this groupoid to the simplicial set we defined before, in the case where the algebra in question is supported in non-negative degrees, we have the following fact \cite{getzler2009lie}:
\begin{proposition}\label{prop:equiv}
	If $\mathfrak{g^*}$ vanishes in negative degrees, there is a natural bijection of sets
	\[ \pi_0(MC_*(\mathfrak{g}^*)) \cong \pi_0(MC(\mathfrak{h}^*, \mu)//\mathfrak{g}^0) \]
	between the connected components of the Maurer-Cartan simplicial set and the set of orbits of the Deligne groupoid.
\end{proposition}

\subsubsection{The simplicial equivalence}
Let us return to the Maurer-Cartan simplicial set and focus on the case of interest $\mathfrak{g}^* = \prod_{\ell \ge 2} C^*_{(\ell,d)}(A)$ for smooth $A$. We now prove the main result of this Section, lifting the map $\Phi: CC^-_d(A) \to \left(\cM_{d-\mathrm{preCY}}(A)\right)_\mathrm{nondeg}$ to a weak simplicial equivalence.

The target for this lift is evidently the nondegenerate locus in the Maurer-Cartan simplicial set corresponding to the dg Lie algebra above:
\[ \cM^\Delta_{d-\mathrm{preCY}}(A) := MC_*(\prod_{\ell \ge 2} C^*_{(\ell,d)}(A)) \]
The source is given by replacing the negative cyclic chain complex by its corresponding simplicial set under the Dold-Kan correspondence. Recall the Kan functor
\[ K_*: \mathrm{Ch}_{\ge 0}(\mathrm{Ab}) \to \mathrm{sAb} \]
which gives an equivalence between chain complexes of abelian groups supported in non-negative degree and simplicial abelian groups. We can further forget the abelian group structure and get a simplicial set.

The functor $K_*$ assigns to the chain complex $(V,\delta)$ the $n$-simplices
\[ K_n(V) = Z^0(C^*(\Delta^n) \otimes V, d + \delta) \]
where $(C^*(\Delta^n),d)$ is the normalized simplicial cochain complex on the $n$-simplex. One possible representation for this complex is in terms of \emph{linear differential forms}
\[ \omega_{i_0,\dots,i_k} = k! \sum_{0 \le j \le k} (-1)^j t_{i_j} dt_{i_0}\dots \widehat{dt_{i_j}} \dots dt_{i_k} \]
for any $1 \le k \le n$.

We now consider the chain complex $\tau_{\ge 0} (CC^-_{*+d}(A))$, i.e. the object of  $\mathrm{Ch}_{\ge 0}(\mathrm{Ab})$ given by shifting the negative cyclic complex down by $d$ and truncating it to lie in non-negative degrees.

A degree zero cycle in this complex is represented by a negative cyclic chain of degree $d$
\[ \lambda = \lambda_0 + \lambda_1 u + \lambda_2 u^2 + \dots \]
closed under $b+uB$, where $\lambda_i \in C_{d+2i}(A)$. Let us fix a choice of a nondegenerate `first component' $\lambda_0 = \nu$ and its inverse $\alpha \in C^d_{(2)}(A)$, representing inverse morphisms of bimodules in $\Hom_{A-A}(A^!,A[d])$ and $\Hom_{A-A}(A[d],A^!)$, respectively.

We now consider simplicial subsets on each side by requiring the `first components' to be constant simplices at $\lambda_0 = nu$ and $m_{(2)} = \alpha$. Let us describe this more precisely. On the source side $K_*(\tau_{\ge 0} (CC^-_{*+d}(A)))$, we decompose each $n$-simplex $\sigma$ as a sum over powers of $u$ and over the basis of forms $\{\omega_{i_0,\dots,i_k}\}$
\[ \sigma(\underline{t}) = \sum_{p=0}^{\infty} \lambda_{p,k,\{i_0,\dots,i_k\}} u^p \otimes \omega_{i_0,\dots,i_k} \]
Given any such $n$-simplex, we can require that its $p=0$ component (that is, its `value' at $u=0$) be constant along the simplicial coordinates $t_i$, and equal to $\nu$. That is, we require $\lambda_{0,k,\{i_0,\dots,i_k\}} = 0$ for all $k >0$ and $\lambda_{0,0,\{\}} = \nu$. This condition defines a simplicial subset of $K_*(\tau_{\ge 0} (CC^-_{*+d}(A)))$, which we denote by $K_*(\tau_{\ge 0} (CC^-_{*+d}(A)))_{\lambda_0 = \nu}$.

On the other side, we can do the same thing and fix the `first component' $m_{(2)}$ to be constant and equal to our chosen quasi-inverse $\alpha$. Each $n$-simplex of $\cM^\Delta_{d-\mathrm{preCY}}(A)$ is a solution to the Maurer-Cartan equation on the dg Lie algebra $\prod_{\ell \ge 2} C^*_{(\ell,d)}(A))\otimes \Omega^*(\Delta^n)$. Here $\Omega^*(\Delta^n)$ is spanned by polynomial differential forms on the simplicial coordinates $t_0,\dots,t_n$; we take the $\ell=2$ part of the solution and demand that it be degree zero and constant as a form, equal to $\alpha$. This again defines a simplicial subset which we denote $\cM^\Delta_{d-\mathrm{preCY}}(A)_{m_{(2)} = \alpha}$.

We are now ready to state the main result of this section. Recall the map of sets $\Phi$ that we defined in \cref{def:phi} takes a nondegenerate negative cyclic chain $\lambda = \lambda_0 + \dots$ whose first term $\lambda_0 = \nu$ has a quasi-inverse $\alpha$ and gives a pre-CY structure $m = \mu + m_{(2)} + \dots$ with $m_{(2)} = \alpha$.
\begin{theorem}\label{thm:simplicial}
	The map $\Phi$ lifts to a weak equivalence of simplicial sets
	\[ \Phi^\Delta: K_*(\tau_{\ge 0} (CC^-_{*+d}(A)))_{\lambda_0 = \nu} \xrightarrow{\simeq} \cM^\Delta_{d-\mathrm{preCY}}(A)_{m_{(2)} = \alpha}  \]
	Taking connected components and putting the resulting bijections together for pair of inverse classes $[\lambda]$ and $[\alpha]$, we get a bijection of sets
	\[ HC^-_d(A)_\mathrm{nondeg} \simeq \pi_0(\cM^\Delta_{d-\mathrm{preCY}}(A)_\mathrm{nondeg}) \]
	between (classes of) smooth CY structures and connected components of the space of nondegenerate pre-CY structures, both of dimension $d$.
\end{theorem}

\begin{proof}
	Note that on the left hand side we have the normalized cochains on the simplex, while on the right-hand side we have differential forms; using the representatives above $\omega_{i_0,\dots,i_k}$, we embed the former into the latter.
	
	Given that embedding, to construct the simplicial lift $\Phi^\Delta$, we simply extend the evaluation map of ribbon quivers linearly over $\Omega^*(\Delta^n)$, and use the same formulas we did for defining $\Phi$. For example, given $m_{(2)}$, in the proof of \cref{prop:inverting} we showed that we can find a solution for $m_{(3)}$ of the form
	\[ m_{(3)} = \tilde\Gamma^0_{(3)}(\lambda_0) + \Gamma^1_{(3)}(\lambda_1) + [\mu, \beta_{(3)}] + [m_{(2)}, \beta_{(2)}] \]
	where the ribbon quivers in $\tilde\Gamma^0$ and $\Gamma^1_{(3)}$ only have vertices with 1 and 2 outgoing arrows.
	
	Given an $n$-simplex on the left hand side given by a linear form $\lambda(\underline{t})$, we can input this instead of $\lambda$ and get a form $m_{(3)}(\underline{t})$; by the same argument as we used to prove \cref{prop:inverting}, but now extended linearly over differential forms, this form will satisfy the equation
	\[ [\mu + d, m_{(3)}(\underline{t})] = [m_{(2)}, m_{(2)}] \]
	which is the new component of the Maurer-Cartan equation on the truncated piece
	$ \left( C^*_{[d]}(A)[1] / \prod_{i > \ell} C^*_{(\ell,d)}(A) \right) \otimes \Omega^*(\Delta^n)$. We continue this iteratively for $\ell = 3,4,\dots$; and in each step we get some polynomial differential forms $m_{(\ell)}(\underline{t})$ solving a new component of that equation, with $m_{(2)}$ fixed to be constant with value $\alpha$.
	
	It follows from the compatibility of evaluation with all the differentials involved that at each new step this defines a map of simplicial sets. Recall that from the definition of $\Phi$, at the step $\ell$ this map depends only on $\lambda$ up to the term with $u$-exponent $\ell-2$.
	
	These maps intertwine the maps induced by truncation, so we get a map between towers of simplicial sets
	\[\xymatrix{
		\dots \ar[r] & K_*(\tau_{\ge 0} (CC^-_{*+d}(A)_{u^2=0}))_{\lambda_0 = \nu} \ar[r] \ar[d] & K_*(\tau_{\ge 0} (CC^-_{*+d}(A)_{u=0}))|_{\lambda_0 = \nu} \ar[d]  \\
		\dots \ar[r] & MC_*\left(\frac{\prod_{\ell \ge 2} C^*_{(\ell,d)}(A)[1]}{\prod_{\ell \ge 4} C^*_{(\ell,d)}(A)[1]}\right)_{m_{(2)} = \alpha} \ar[r] & MC_*\left(\frac{\prod_{\ell \ge 2} C^*_{(\ell,d)}(A)[1]}{\prod_{\ell \ge 3} C^*_{(\ell,d)}(A)[1]}\right)_{m_{(2)} = \alpha}
	}\]
	and the desired map $\Phi^\Delta$ is the map induced between the limits of these towers.
	
	We now prove that the map $\Phi^\Delta$ so defined is a weak equivalence of simplicial sets. Each horizontal map is a Kan fibration, so it is enough to prove that each vertical map is a weak equivalence. We do this by induction; the last column is actually just the identity map of the point (seen as a the totally degenerate $n$-simplices); this is because we fixed by hand the $\lambda_0$ and $m_{(2)}$ components to be exactly $\nu$ and $\alpha$.
	 
	For the induction step we focus on a single square
	\[\xymatrix{
		 K_*(\tau_{\ge 0} (CC^-_{*+d}(A)_{u^{i}=0}))_{\lambda_0 = \nu} \ar[r] \ar[d] & K_*(\tau_{\ge 0} (CC^-_{*+d}(A)_{u^{i-1}=0}))|_{\lambda_0 = \nu} \ar[d]  \\
		 MC_*\left(\frac{\prod_{\ell \ge 2} C^*_{(\ell,d)}(A)[1]}{\prod_{\ell \ge i+2} C^*_{(\ell,d)}(A)[1]}\right)_{m_{(2)} = \alpha} \ar[r] & MC_*\left(\frac{\prod_{\ell \ge 2} C^*_{(\ell,d)}(A)[1]}{\prod_{\ell \ge i+1} C^*_{(\ell,d)}(A)[1]}\right)_{m_{(2)} = \alpha}
	}\]
	If the right column is a weak equivalence, by \cite{lurie2009higher} it is enough to show that the left vertical map induces weak equivalences for each pair of fibers of the horizontal maps over \emph{points} (i.e. 0-simplices).
	
	In down-to-earth terms, we have an `actual' solution of the Maurer-Cartan equation on $C^*_{[d]}(A)[1]$ up to the term $m_{(i)}$, corresponding to a truncated negative cyclic chain $\lambda = \lambda_0 + \lambda_1 u + \dots + \lambda_{i-2} u^{i-2}$. We then look at the map $\Phi^\Delta$ applied to a differential form
	\[ \lambda = \lambda_0 + \lambda_1 u + \dots + \lambda_{i-2} u^{i-2} + \lambda_{i-1}(\underline{t}) u^{i-1} \]
	where $\lambda_{i-1}(\underline{t})$ is linear on the $t_i$ coordinates on the $n$-simplex. The only dependence on this form is in the evaluation of the last ribbon quiver in $\Gamma_{(i+1)}$, that is, the term
	\[ \Gamma^{i-1}_{(i+1)}(\lambda_{i-1}(\underline{t})) \]
	so we are reduced to proving that the operation $\Gamma^{i-1}_{(i+1)}(-)$, seen as a map
	\[ C_*(A) \otimes C^*(\Delta^n) \to C^*_{(i+1)}(A) \otimes \Omega^*(\Delta^n) \]
	defines a weak equivalence between closed forms and forms satisfying the $i+1$ component of the Maurer-Cartan equation. This follows from the proof of \cref{prop:quasiIso} together with the fact that the inclusion of normalized simplicial cochains into differential forms is a homotopy retract (a simplicial version of the de Rham theorem).
\end{proof}

\subsubsection{Special case: the groupoid}\label{sec:groupoid}
\cref{thm:simplicial} holds for any smooth $A_\infty$ algebra/category $A$, without any assumptions on degrees. Now recall that if a dg Lie algebra $\mathfrak{g}$ vanishes in negative degrees, its set of Maurer-Cartan solutions admits an equivalent, simpler, description than the full simplicial set, given by the Deligne groupoid $MC(\mathfrak{g})/\mathfrak{g}^0$.

The following result can be seen as a slight refinement of \cref{thm:simplicial} in the case where $A$ has vanishing Hochschild cohomology in negative degrees.
\begin{theorem}
	Assume that $HH^i(A) = 0$ for all $i < 0$. Then there is a bijection
	\[ HC^-_*(A)_\mathrm{nondeg} \simeq \pi_0(MC(\mathfrak{g})_\mathrm{nondeg}/\mathfrak{g}^0) \]
	where $\mathfrak{g} = \prod_{\ell \ge 2} C^*_{(\ell,d)}(A)$, between nondegenerate negative cyclic homology classes and orbits in the groupoid of nondegenerate pre-CY structures.
\end{theorem}
This is \emph{almost} a direct corollary of \cref{thm:simplicial} and \cref{prop:equiv}; the only reason why it does not follow directly is because we are not assuming that $\mathfrak{g}$ vanishes at chain level in negative degrees. Nevertheless, we can prove this fact explicitly, by an iterative calculation that we now sketch.
\begin{proof}
	We first note that even though the dg Lie algebra $\mathfrak{g}$ is not nilpotent, the action of $\mathfrak{g}$ is still well-defined; each element $x \in \mathfrak{g}$ is a sum of vertices with at least two outgoing arrows, so $[x,-]$ increases the number of outgoing arrows. So the sum defining $\exp(x)y$ is finite at each truncated level $\prod_{i > 2} C^*_{i,d}(A)[1] / \prod_{i > \ell} C^*_{i,d}(A)$, and the action lifts to the limit.

	Recall that we have maps
	\[ \Phi: CC^-_d(A)_\mathrm{nondeg} \longleftrightarrow \left(\cM_{d-\mathrm{preCY}}(A)\right)_\mathrm{nondeg}: \cL \]
	One of the directions is easier: let us pick $\lambda \in CC^-_d(A)_\mathrm{nondeg}$ and take $m = \Phi(\lambda)$; by definition this satisfies
	\[ e_m = \Gamma(m, \lambda) + [m,p] \]
	where $e_m$ is the `energy function' associated to $m$, $\Gamma$ is the sum of tube quivers we defined previously and $p \in \prod_{i > 2} C^1_{i,d}(A)$. Defining now $\lambda' = \cL(m)$, by definition we have
	\[ e_m = \Gamma(m,\lambda') + [m,q] \]
	for some other element $q \in \prod_{i > 2} C^1_{i,d}(A)$. Therefore $\Gamma(m,\lambda'-\lambda) = [m, q-p]$, and since $\Gamma(m,-)$ is a quasi-isomorphism in cohomology we have $[\lambda'] = [\lambda]$
	
	The remaining direction is harder and has to be done iteratively in $\ell$. We now start with a pre-CY structure $m$, take $\lambda = \cL(m)$ and $n = \Phi(\lambda)$; reusing the symbols $p,q$ these satisfy equations
	\[ e_m = \Gamma(m, \lambda) + [m,p], \quad e_n = \Gamma(n,\lambda) + [n,q] \]
	We have to prove that there is
	\[ x = x_{(2)} + x_{(3)} + \dots \in \prod_{\ell > 2} C^1_{\ell,d}(A) \]
	such that $n = m + [x,m] + \frac{1}{2}[x,[x,m]] + \dots$. At level $\ell = 2$, this is simply $n_{(2)} = m_{(2)} + [x_{(2)},\mu]$.
	
	Let us then write $s_{(2)} = m_{(2)} - n_{(2)}$. Subtracting the equations satisfied by $n_{(2)}, m_{(2)}$ we get
	\begin{align*}
		-s_{(2)} &= \Gamma(n,\lambda) - \Gamma(m,\lambda) + [\mu,q_{(2)}-p_{(2)}] \\
		= &-\frac{1}{2} \left( \tikzfig{
			\node [vertex] (x) at (0,0) {$\lambda_0$};
			\node at (2,0.3) {};
			\node at (-2,0.3) {};
			\node [bullet] (ne) at (0.7,0.7) {};
			\node [bullet] (e) at (1,0) {};
			\node [bullet] (w) at (-1,0) {};
			\draw [->-] (x) to (ne);
			\draw [->-=1] (w) to (-2,0);
			\draw [->-=1] (e) to (2,0);
			\draw [->-] (0,1) arc (90:45:1);
			\draw [->-] (0.7,0.7) arc (45:0:1);
			\draw [->-] (0,1) arc (90:180:1);
			\draw [->-] (0,-1) arc (-90:-180:1);
			\draw [->-] (0,-1) arc (-90:0:1);
			\node [vertex,fill=white] (n) at (0,1) {$s_{(2)}$};
			\node [vertex,fill=white] (s) at (0,-1) {$n_{(2)}$};
		} +\tikzfig{
			\node [vertex] (x) at (0,0) {$\lambda_0$};
			\node at (2,0.3) {};
			\node at (-2,0.3) {};
			\node [bullet] (sw) at (-0.7,-0.7) {};
			\node [bullet] (e) at (1,0) {};
			\node [bullet] (w) at (-1,0) {};
			\draw [->-] (x) to (sw);
			\draw [->-=1] (w) to (-2,0);
			\draw [->-=1] (e) to (2,0);
			\draw [->-] (0,1) arc (90:0:1);
			\draw [->-] (-0.7,-0.7) arc (-135:-180:1);
			\draw [->-] (0,1) arc (90:180:1);
			\draw [->-] (0,-1) arc (-90:-135:1);
			\draw [->-] (0,-1) arc (-90:0:1);
			\node [vertex,fill=white] (n) at (0,1) {$s_{(2)}$};
			\node [vertex,fill=white] (s) at (0,-1) {$n_{(2)}$};
		}\right) \\
	&-\frac{1}{2} \left( \tikzfig{
		\node [vertex] (x) at (0,0) {$\lambda_0$};
		\node at (2,0.3) {};
		\node at (-2,0.3) {};
		\node [bullet] (ne) at (0.7,0.7) {};
		\node [bullet] (e) at (1,0) {};
		\node [bullet] (w) at (-1,0) {};
		\draw [->-] (x) to (ne);
		\draw [->-=1] (w) to (-2,0);
		\draw [->-=1] (e) to (2,0);
		\draw [->-] (0,1) arc (90:45:1);
		\draw [->-] (0.7,0.7) arc (45:0:1);
		\draw [->-] (0,1) arc (90:180:1);
		\draw [->-] (0,-1) arc (-90:-180:1);
		\draw [->-] (0,-1) arc (-90:0:1);
		\node [vertex,fill=white] (n) at (0,1) {$m_{(2)}$};
		\node [vertex,fill=white] (s) at (0,-1) {$s_{(2)}$};
	} +\tikzfig{
		\node [vertex] (x) at (0,0) {$\lambda_0$};
		\node at (2,0.3) {};
		\node at (-2,0.3) {};
		\node [bullet] (sw) at (-0.7,-0.7) {};
		\node [bullet] (e) at (1,0) {};
		\node [bullet] (w) at (-1,0) {};
		\draw [->-] (x) to (sw);
		\draw [->-=1] (w) to (-2,0);
		\draw [->-=1] (e) to (2,0);
		\draw [->-] (0,1) arc (90:0:1);
		\draw [->-] (-0.7,-0.7) arc (-135:-180:1);
		\draw [->-] (0,1) arc (90:180:1);
		\draw [->-] (0,-1) arc (-90:-135:1);
		\draw [->-] (0,-1) arc (-90:0:1);
		\node [vertex,fill=white] (n) at (0,1) {$m_{(2)}$};
		\node [vertex,fill=white] (s) at (0,-1) {$s_{(2)}$};
	}\right) \\
	&+ [\mu,q_{(2)}-p_{(2)}]
	\end{align*}
	But since $\Gamma_{(2)}, \lambda_0$ are closed under the relevant differentials, and both $m_{(2)}$ and $n_{(2)}$ are representatives of the inverse of $\lambda_0$, each of the terms in the parentheses above evaluates to something cohomologous to $s_{(2)}$. In other words, there is $r_{(2)} \in C^{d-1}_{(2)}(A)$ such that
	\begin{align*} [\mu,r_{(2)}] &= 2 s_{(2)} - \frac{1}{2} \left( \tikzfig{
		\node [vertex] (x) at (0,0) {$\lambda_0$};
		\node at (2,0.3) {};
		\node at (-2,0.3) {};
		\node [bullet] (ne) at (0.7,0.7) {};
		\node [bullet] (e) at (1,0) {};
		\node [bullet] (w) at (-1,0) {};
		\draw [->-] (x) to (ne);
		\draw [->-=1] (w) to (-2,0);
		\draw [->-=1] (e) to (2,0);
		\draw [->-] (0,1) arc (90:45:1);
		\draw [->-] (0.7,0.7) arc (45:0:1);
		\draw [->-] (0,1) arc (90:180:1);
		\draw [->-] (0,-1) arc (-90:-180:1);
		\draw [->-] (0,-1) arc (-90:0:1);
		\node [vertex,fill=white] (n) at (0,1) {$s_{(2)}$};
		\node [vertex,fill=white] (s) at (0,-1) {$n_{(2)}$};
	} +\tikzfig{
		\node [vertex] (x) at (0,0) {$\lambda_0$};
		\node at (2,0.3) {};
		\node at (-2,0.3) {};
		\node [bullet] (sw) at (-0.7,-0.7) {};
		\node [bullet] (e) at (1,0) {};
		\node [bullet] (w) at (-1,0) {};
		\draw [->-] (x) to (sw);
		\draw [->-=1] (w) to (-2,0);
		\draw [->-=1] (e) to (2,0);
		\draw [->-] (0,1) arc (90:0:1);
		\draw [->-] (-0.7,-0.7) arc (-135:-180:1);
		\draw [->-] (0,1) arc (90:180:1);
		\draw [->-] (0,-1) arc (-90:-135:1);
		\draw [->-] (0,-1) arc (-90:0:1);
		\node [vertex,fill=white] (n) at (0,1) {$s_{(2)}$};
		\node [vertex,fill=white] (s) at (0,-1) {$n_{(2)}$};
	}\right) \\
	&+\frac{1}{2} \left( \tikzfig{
		\node [vertex] (x) at (0,0) {$\lambda_0$};
		\node at (2,0.3) {};
		\node at (-2,0.3) {};
		\node [bullet] (ne) at (0.7,0.7) {};
		\node [bullet] (e) at (1,0) {};
		\node [bullet] (w) at (-1,0) {};
		\draw [->-] (x) to (ne);
		\draw [->-=1] (w) to (-2,0);
		\draw [->-=1] (e) to (2,0);
		\draw [->-] (0,1) arc (90:45:1);
		\draw [->-] (0.7,0.7) arc (45:0:1);
		\draw [->-] (0,1) arc (90:180:1);
		\draw [->-] (0,-1) arc (-90:-180:1);
		\draw [->-] (0,-1) arc (-90:0:1);
		\node [vertex,fill=white] (n) at (0,1) {$m_{(2)}$};
		\node [vertex,fill=white] (s) at (0,-1) {$s_{(2)}$};
	} +\tikzfig{
		\node [vertex] (x) at (0,0) {$\lambda_0$};
		\node at (2,0.3) {};
		\node at (-2,0.3) {};
		\node [bullet] (sw) at (-0.7,-0.7) {};
		\node [bullet] (e) at (1,0) {};
		\node [bullet] (w) at (-1,0) {};
		\draw [->-] (x) to (sw);
		\draw [->-=1] (w) to (-2,0);
		\draw [->-=1] (e) to (2,0);
		\draw [->-] (0,1) arc (90:0:1);
		\draw [->-] (-0.7,-0.7) arc (-135:-180:1);
		\draw [->-] (0,1) arc (90:180:1);
		\draw [->-] (0,-1) arc (-90:-135:1);
		\draw [->-] (0,-1) arc (-90:0:1);
		\node [vertex,fill=white] (n) at (0,1) {$m_{(2)}$};
		\node [vertex,fill=white] (s) at (0,-1) {$s_{(2)}$};
	}\right)
\end{align*}
Using this fact we find that $s_{(2)} = [\mu, r_{(2)}+q_{(2)}+p_{(2)}]$, so there is a solution $s_{(2)} = [x_{(2)},\mu]$.

In order to solve the equation in the next order $\ell=3$, we also have to show that this solution for $x_{(2)}$ can be of a particular form. Substituting for $s_{(2)}$ in the equation, and again using the closedness of $\Gamma_{(2)},\lambda_0$ we have
\begin{align*}
	[\mu,x_{(2)}] = &\frac{1}{2} \left[ \mu,  \tikzfig{
		\node [vertex] (x) at (0,0) {$\lambda_0$};
		\node at (2,0.3) {};
		\node at (-2,0.3) {};
		\node [bullet] (ne) at (0.7,0.7) {};
		\node [bullet] (e) at (1,0) {};
		\node [bullet] (w) at (-1,0) {};
		\draw [->-] (x) to (ne);
		\draw [->-=1] (w) to (-2,0);
		\draw [->-=1] (e) to (2,0);
		\draw [->-] (0,1) arc (90:45:1);
		\draw [->-] (0.7,0.7) arc (45:0:1);
		\draw [->-] (0,1) arc (90:180:1);
		\draw [->-] (0,-1) arc (-90:-180:1);
		\draw [->-] (0,-1) arc (-90:0:1);
		\node [vertex,fill=white] (n) at (0,1) {$m_{(2)}$};
		\node [vertex,fill=white] (s) at (0,-1) {$x_{(2)}$};
	} +\tikzfig{
		\node [vertex] (x) at (0,0) {$\lambda_0$};
		\node at (2,0.3) {};
		\node at (-2,0.3) {};
		\node [bullet] (sw) at (-0.7,-0.7) {};
		\node [bullet] (e) at (1,0) {};
		\node [bullet] (w) at (-1,0) {};
		\draw [->-] (x) to (sw);
		\draw [->-=1] (w) to (-2,0);
		\draw [->-=1] (e) to (2,0);
		\draw [->-] (0,1) arc (90:0:1);
		\draw [->-] (-0.7,-0.7) arc (-135:-180:1);
		\draw [->-] (0,1) arc (90:180:1);
		\draw [->-] (0,-1) arc (-90:-135:1);
		\draw [->-] (0,-1) arc (-90:0:1);
		\node [vertex,fill=white] (n) at (0,1) {$m_{(2)}$};
		\node [vertex,fill=white] (s) at (0,-1) {$x_{(2)}$};
	} \right. \\
	&+ \left. \tikzfig{
	\node [vertex] (x) at (0,0) {$\lambda_0$};
	\node at (2,0.3) {};
	\node at (-2,0.3) {};
	\node [bullet] (ne) at (0.7,0.7) {};
	\node [bullet] (e) at (1,0) {};
	\node [bullet] (w) at (-1,0) {};
	\draw [->-] (x) to (ne);
	\draw [->-=1] (w) to (-2,0);
	\draw [->-=1] (e) to (2,0);
	\draw [->-] (0,1) arc (90:45:1);
	\draw [->-] (0.7,0.7) arc (45:0:1);
	\draw [->-] (0,1) arc (90:180:1);
	\draw [->-] (0,-1) arc (-90:-180:1);
	\draw [->-] (0,-1) arc (-90:0:1);
	\node [vertex,fill=white] (n) at (0,1) {$x_{(2)}$};
	\node [vertex,fill=white] (s) at (0,-1) {$n_{(2)}$};
} +\tikzfig{
\node [vertex] (x) at (0,0) {$\lambda_0$};
\node at (2,0.3) {};
\node at (-2,0.3) {};
\node [bullet] (sw) at (-0.7,-0.7) {};
\node [bullet] (e) at (1,0) {};
\node [bullet] (w) at (-1,0) {};
\draw [->-] (x) to (sw);
\draw [->-=1] (w) to (-2,0);
\draw [->-=1] (e) to (2,0);
\draw [->-] (0,1) arc (90:0:1);
\draw [->-] (-0.7,-0.7) arc (-135:-180:1);
\draw [->-] (0,1) arc (90:180:1);
\draw [->-] (0,-1) arc (-90:-135:1);
\draw [->-] (0,-1) arc (-90:0:1);
\node [vertex,fill=white] (n) at (0,1) {$x_{(2)}$};
\node [vertex,fill=white] (s) at (0,-1) {$n_{(2)}$};
} \right] \\
&+ [\mu,q_{(2)}-p_{(2)}]
\end{align*}

Note that $[\mu,x_{(2)}]$ is a closed element of degree $d-1$ in $C^*_{(2)}(A)$. Now, since the element $m_{(2)}$ was nondegenerate, we have a quasi-isomorphism of bimodules $A \cong A^![d]$ so we have a quasi-isomorphism of complexes
\[ C^*_{(2)}(A) \simeq \Hom_{A\mh A}(A,A^!) \simeq  \Hom_{A\mh A}(A,A[-d]) \simeq C^{*-d}(A)  \]
so since we assumed $HH^{-1}(A) = 0$ we can find $x_{(2)}$ solving the equation
\begin{align*}
x_{(2)} &= \frac{1}{2} \left(\tikzfig{
	\node [vertex] (x) at (0,0) {$\lambda_0$};
	\node at (2,0.3) {};
	\node at (-2,0.3) {};
	\node [bullet] (ne) at (0.7,0.7) {};
	\node [bullet] (e) at (1,0) {};
	\node [bullet] (w) at (-1,0) {};
	\draw [->-] (x) to (ne);
	\draw [->-=1] (w) to (-2,0);
	\draw [->-=1] (e) to (2,0);
	\draw [->-] (0,1) arc (90:45:1);
	\draw [->-] (0.7,0.7) arc (45:0:1);
	\draw [->-] (0,1) arc (90:180:1);
	\draw [->-] (0,-1) arc (-90:-180:1);
	\draw [->-] (0,-1) arc (-90:0:1);
	\node [vertex,fill=white] (n) at (0,1) {$m_{(2)}$};
	\node [vertex,fill=white] (s) at (0,-1) {$x_{(2)}$};
} +\tikzfig{
	\node [vertex] (x) at (0,0) {$\lambda_0$};
	\node at (2,0.3) {};
	\node at (-2,0.3) {};
	\node [bullet] (sw) at (-0.7,-0.7) {};
	\node [bullet] (e) at (1,0) {};
	\node [bullet] (w) at (-1,0) {};
	\draw [->-] (x) to (sw);
	\draw [->-=1] (w) to (-2,0);
	\draw [->-=1] (e) to (2,0);
	\draw [->-] (0,1) arc (90:0:1);
	\draw [->-] (-0.7,-0.7) arc (-135:-180:1);
	\draw [->-] (0,1) arc (90:180:1);
	\draw [->-] (0,-1) arc (-90:-135:1);
	\draw [->-] (0,-1) arc (-90:0:1);
	\node [vertex,fill=white] (n) at (0,1) {$m_{(2)}$};
	\node [vertex,fill=white] (s) at (0,-1) {$x_{(2)}$};
} \right. \\
&+ \left. \tikzfig{
	\node [vertex] (x) at (0,0) {$\lambda_0$};
	\node at (2,0.3) {};
	\node at (-2,0.3) {};
	\node [bullet] (ne) at (0.7,0.7) {};
	\node [bullet] (e) at (1,0) {};
	\node [bullet] (w) at (-1,0) {};
	\draw [->-] (x) to (ne);
	\draw [->-=1] (w) to (-2,0);
	\draw [->-=1] (e) to (2,0);
	\draw [->-] (0,1) arc (90:45:1);
	\draw [->-] (0.7,0.7) arc (45:0:1);
	\draw [->-] (0,1) arc (90:180:1);
	\draw [->-] (0,-1) arc (-90:-180:1);
	\draw [->-] (0,-1) arc (-90:0:1);
	\node [vertex,fill=white] (n) at (0,1) {$x_{(2)}$};
	\node [vertex,fill=white] (s) at (0,-1) {$n_{(2)}$};
} +\tikzfig{
	\node [vertex] (x) at (0,0) {$\lambda_0$};
	\node at (2,0.3) {};
	\node at (-2,0.3) {};
	\node [bullet] (sw) at (-0.7,-0.7) {};
	\node [bullet] (e) at (1,0) {};
	\node [bullet] (w) at (-1,0) {};
	\draw [->-] (x) to (sw);
	\draw [->-=1] (w) to (-2,0);
	\draw [->-=1] (e) to (2,0);
	\draw [->-] (0,1) arc (90:0:1);
	\draw [->-] (-0.7,-0.7) arc (-135:-180:1);
	\draw [->-] (0,1) arc (90:180:1);
	\draw [->-] (0,-1) arc (-90:-135:1);
	\draw [->-] (0,-1) arc (-90:0:1);
	\node [vertex,fill=white] (n) at (0,1) {$x_{(2)}$};
	\node [vertex,fill=white] (s) at (0,-1) {$n_{(2)}$};
} \right) + q_{(2)}-p_{(2)}
\end{align*}
up to $[\mu,-]$-exact terms. For the next step $\ell=3$ we write $n_{(3)} = m_{(3)} + [x_{(2)},m_{(2)}]+ \frac{1}{2}[x_{(2)},[x_{(2)},\mu]] - s_{(3)}$ and write the analogous equation, substituting in the above solution for $x_{(2)}$, solving for $x_{(3)}$ in $s_{(3)} = [\mu,x_{(3)}]$. Proceeding like that for each $\ell$ gives a solution $x = x_{(2)} + x_{(3)}+ \dots$ for the `gauge transformation' taking $m$ to $n$.	
\end{proof}

\printbibliography

\end{document}